\documentclass[10pt]{amsart}
\usepackage{amssymb, amsthm, amsmath}
\usepackage{graphicx}
\usepackage{latexsym}

\newtheorem{thm}{Theorem}

\theoremstyle{definition}

\newtheorem{example}{Example}
\newtheorem{remark}{Remark}

\newcommand\C{{\mathbb C}}
\newcommand\Q{{\mathbb Q}}
\newcommand\N{{\mathbb N}}
\newcommand\F{{\mathrm F}}
\newcommand\V{{\mathcal V}}
\newcommand\IH{{\mathbb H}}
\newcommand{\ti}{\vartheta}
\newcommand{\Ti}{\Theta}
\newcommand\Eta{H}

\newcommand\bP{{\mathbb P}}
\newcommand\cP{{\mathcal P}}
\newcommand\cO{{\mathcal O}}
\newcommand\cC{{\mathcal C}}
\newcommand\cW{{\mathcal W}}

\newcommand\X{{\mathfrak X}}
\newcommand\Z{{\mathbb Z}}

\newcommand\AS{{\mathfrak S}}
\newcommand\BS{{\mathfrak B}}
\newcommand\CS{{\mathfrak C}}
\newcommand\DS{{\mathfrak D}}
\newcommand\XX{{\mathrm X}}
\newcommand\YY{{\mathrm Y}}
\newcommand\I{{\mathrm I}}
\newcommand\J{{\mathrm J}}

\newcommand\fraka{{\mathfrak a}}
\newcommand\frakb{{\mathfrak b}}

\newcommand\al{\alpha}

\newcommand\la{\lambda}

\newcommand\x{{\mathrm{x}}}
\newcommand\y{{\mathrm{y}}}
\newcommand\om{{\varpi}}
\newcommand\ta{{\tau}}
\newcommand\Sh{{\mathcal S}}

\newcommand\ssm{\smallsetminus}

\newcommand\noin{\noindent}

\newcommand\bull{{\scriptscriptstyle \bullet}}
\newcommand\eqto{\stackrel{\lower1.5pt\hbox{$\scriptstyle\sim\,$}}\to}
\newcommand\ov{\overline}
\newcommand\hra{\hookrightarrow}

\newcommand\wh{\widehat}
\newcommand\wt{\widetilde}
\newcommand\dis{\displaystyle}

\DeclareMathOperator{\Pic}{Pic} 
 \DeclareMathOperator{\Pf}{Pfaffian}
\DeclareMathOperator{\Sp}{Sp} \DeclareMathOperator{\SO}{SO}
 \DeclareMathOperator{\GL}{GL}
\DeclareMathOperator{\LG}{LG} \DeclareMathOperator{\IG}{IG}
 \DeclareMathOperator{\OG}{OG}
\DeclareMathOperator{\G}{G} 
\DeclareMathOperator{\QH}{QH}
\DeclareMathOperator{\U}{U}
\DeclareMathOperator{\HH}{\mathrm{H}} 
 
\DeclareMathOperator{\type}{\mathrm{type}}
\DeclareMathOperator{\rank}{\mathrm{rank}}

\newcommand{\ignore}[1]{}
\newcommand{\pic}[2]{\includegraphics[scale=#1]{#2}}

\begin{document}

\title[Giambelli and degeneracy locus formulas for classical 
$G/P$ spaces]
{Giambelli and degeneracy locus formulas for classical $G/P$ spaces}

\date{January 11, 2016}

\author{Harry~Tamvakis} \address{University of Maryland, Department of
Mathematics, 1301 Mathematics Building, College Park, MD 20742, USA}
\email{harryt@math.umd.edu}



\thanks{The author was supported in part by NSF Grants DMS-0901341
and DMS-1303352.}

\begin{abstract}
Let $G$ be a classical complex Lie group, $P$ any parabolic subgroup
of $G$, and $X=G/P$ the corresponding homogeneous space, which
parametrizes (isotropic) partial flags of subspaces of a fixed vector
space. In the mid 1990s, Fulton, Pragacz, and Ratajski \cite{Fu3, PR2,
  FP} asked for global formulas which express the cohomology classes
of the universal Schubert varieties in flag bundles -- when the space
$X$ varies in an algebraic family -- in terms of the Chern classes of
the vector bundles involved in their definition. This has applications
to the theory of degeneracy loci of vector bundles and is closely
related to the Giambelli problem for the torus-equivariant cohomology
ring of $X$. In this article, we explain the answer to these questions
which was obtained in \cite{T6}, in terms of combinatorial data coming
from the Weyl group.
\end{abstract}

\maketitle

\setcounter{section}{-1}

\section{Introduction}
\label{introsec}

The theory of degeneracy loci of vector bundles has its roots in the
19th century, motivated by questions in elimination theory and
enumerative algebraic geometry. The modern subject began with the work
of Thom and Porteous in topology, which was generalized and extended
to the algebraic setting by Kempf, Laksov, and Lascoux \cite{KL, La0,
  Fu5}. The simplest example involves two complex vector bundles $E,F$
on a smooth algebraic variety $M$. Given a generic map of vector
bundles $f:E\to F$ and $r$ any integer, the locus $M_r$ of points
$m\in M$ where $\rank(f_m)\leq r$ is called a {\em degeneracy
  locus}. Thom \cite{Th} showed that the homology class of $M_r$ must
be Poincar\'e dual to a universal polynomial in the Chern classes of
the vector bundles $E$ and $F$, and Porteous \cite{Po} later found
this representing polynomial. Such degeneracy loci arise frequently in
problems of algebraic geometry and singularity theory, therefore
explicit Chern class formulas for these loci can be quite useful. We
refer to \cite{Tu, P3, FP, FR2, Ka} for surveys, and to \cite{Bertram,
  DP, EvG, FR, Fu3, FL, HT, JLP, Ka1, KTlg, PP, P1, PR2, Sa, SdS, T1,
  T2} for an incomplete list of applications.

In a series of papers in the 1990s, Fulton \cite{Fu1, Fu2, Fu3}
generalized the work of Kempf-Laksov further, to a map of flagged
vector bundles, and studied an analogue of the same problem for the
other classical Lie groups. This involved degeneracy loci given by
incidence relations between a pair of isotropic flags of subbundles of
a fixed vector bundle, which is equipped with a symplectic or
orthogonal form. In all cases, the {\em Schubert polynomials}
representing the cohomology classes of the loci were defined by an
algorithm using divided difference operators (stemming from \cite{BGG,
  D1, D2, LS1}) applied to a `top polynomial' which represented the
class of the diagonal (the locus of points on the base variety where
the two flags coincide). Related computations were performed at much
the same time by Pragacz and Ratajski \cite{PR2}, and other competing
theories of Schubert polynomials in the Lie types B, C, and D were
discovered \cite{BH, FK2, LP1}. In \cite{Fu3} and \cite[\S 9.5]{FP},
Fulton and Pragacz asked for combinatorially explicit, global formulas
for the cohomology classes of degeneracy loci, which have a similar
shape for all the classical groups, and are determinantal whenever
possible. The aim of this article is to describe the answer to this
question which was obtained in \cite{T6}, building on a series of
earlier works, in terms of data coming from the Weyl group.

Graham \cite{Gra} recast the above degeneracy locus problem using the
language of Lie theory, and studied the universal case when the
structure group $G$ of the fibre bundles involved is any complex
reductive group (see also \cite[\S 6.6]{Br}). He observed that the
degeneracy locus question of \cite{Fu3} is essentially equivalent to
the problem of obtaining a formula for the {\em equivariant Schubert
  classes} in the torus-equivariant cohomology ring of the flag
variety $G/B$ (when $G$ is a classical group, there is also the
twisted case, when the bilinear form takes values in a line
bundle). Indeed, from the point of view of a Lie theorist, there seems
to be no reason to exclude the exceptional groups from the degeneracy
locus story. We will suggest two reasons below why the classical
groups appear to be special for this question.

In type A, the double Schubert polynomials of Lascoux and
Sch\"utzenberger \cite{La1, LS1} were characterized as the unique
polynomials that satisfy the general degeneracy locus formula of
\cite{Fu1}. Fomin and Kirillov \cite{FK2} observed that this strong
uniqueness property breaks down in type B, where in fact there is a
plethora of theories of (single) Schubert polynomials. However, the
Schubert polynomials of Billey and Haiman \cite{BH} impressed us as
the most combinatorially explicit theory among those available in the
other classical Lie types. These polynomials enjoyed most of the
properties of the type A single Schubert polynomials, but their
translation (as given in \cite{BH}) into Chern class formulas in
$\HH^*(G/B)$ involved a change of variables and an ensuing loss of
combinatorial control. This problem was first addressed by the author
\cite{T2, T3}, using a more natural and geometric substitution of the
variables, with applications to arithmetic intersection theory. Ikeda,
Mihalcea, and Naruse \cite{IMN} later introduced double versions of
the Billey-Haiman Schubert polynomials and extended the substitution
of \cite{T2, T3} to this setting -- expressing it in a better way, as
a ring homomorphism (the {\em geometrization maps} of \S
\ref{ddgeom}).  With this work, the search for a satisfactory analogue
of the Lascoux-Sch\"utzenberger theory in the other classical Lie
types was finally over.

Although the decision of which theory of Schubert polynomials to use
is clearly important, by construction they only provide formulas in
terms of the {\em Chern roots} of the vector bundles involved. When
the initial degeneracy locus problem carries the symmetries of a
parabolic subgroup $P$ of $G$, we seek an answer which manifestly
exhibits the same symmetries. This should generalize the Jacobi-Trudi
determinants and Schur Pfaffians that appear when the Schubert
polynomials are evaluated on (maximal) Grassmannian elements of the
Weyl group, as in \cite{KL, Fu3, PR2, FP, KT}. In other words, we
desire formulas that are {\em native to $G/P$}, i.e., expressed in
terms of Schubert classes that live in the cohomology ring of the
homogeneous space $G/P$. It turned out that a precise understanding of
the {\em Giambelli problem} for $\HH^*(G/P)$, which is closely related
to the degeneracy locus formulas above, was necessary for further
progress.

The cohomology of $X=G/P$ is a free abelian group on the basis of
Schubert classes, the cohomology classes of the Schubert varieties.
When $G$ is a classical Lie group, there are certain {\em special
  Schubert classes} among these, which generate the ring
$\HH^*(X)$. This is one place where the fact that $G$ is classical is
important: at present, we do not know how to define special classes
for the exceptional groups. For classical $G$, one has a good
definition of special Schubert varieties, which is uniform across the
four types. In this case, the variety $X$ parametrizes partial flags
of subspaces of a vector space, which in types B, C, and D are
required to be isotropic with respect to an orthogonal or symplectic
form. If $X$ is an (isotropic) Grassmannian, then the special Schubert
varieties are defined as the locus of (isotropic) linear subspaces
which meet a given (isotropic or coisotropic) linear subspace
nontrivially, following \cite{Pi, BKT1}. The special Schubert
varieties on any partial flag variety $X$ are the inverse images of
the special Schubert varieties on the Grassmannians to which $X$
projects.  The special Schubert classes are the cohomology classes of
the special Schubert varieties; in most examples, they are equal to
the Chern classes of the universal {\em quotient} bundles over $X$, up
to a factor of two.

The Giambelli problem for $\HH^*(X)$ challenges us to write a
general Schubert class as an explicit polynomial in the above special
classes. The papers \cite{G1, BKT2, BKT4} addressed this question for
all (isotropic) Grassmannians, and \cite{BKTY1, T6} extended the
answer to any classical $G/P$ space. To do this, we had to go beyond
the known hermitian symmetric, fully commutative examples, and invent
a considerable body of new combinatorics. The Schubert classes are
indexed by (typed) {\em $k$-strict partitions}, the Giambelli formulas
are expressed using Young's {\em raising operators} \cite{Y, Lit} and
studied using a new calculus of these operators \cite{BKT2, T4}, the
Schur polynomials are extended to {\em theta} and {\em eta
  polynomials}, and instead of Young tableaux, we count paths in {\em
  $k$-transition trees}. Ultimately, all of these objects can be
understood purely in terms of the combinatorics of the Weyl group of
(signed) permutations.

The degeneracy locus problem is equivalent to the Giambelli problem
when the space $X$ varies in an algebraic family, and thus would
appear to be more difficult. Indeed, in most cases where determinantal
formulas for the double Schubert polynomials representing the loci
were known, these formulas were significantly more complicated than
their single versions -- which address the Giambelli problem in that
case.  The type A paper \cite{BKTY1} changed that paradigm: it
established the surprising fact that if one uses the language of {\em
  quiver polynomials}, then the answer to the degeneracy locus problem
has the same shape as that for the Giambelli problem, and indeed, a
near identical proof! This picture was generalized to all classical
types in \cite{T6}, in a synthesis which used all of the above
ingredients, and added some new ones. The results were combinatorial
{\em splitting formulas} for the Schubert polynomials of \cite{BH,
 IMN}, and direct translations of these into degeneracy locus
formulas, with the symmetries native to the appropriate $G/P$
space. 

The goal of this paper is to explain the above story. The narrative
combines elements from algebraic geometry, Lie theory, and
combinatorics, and we have strived to keep the exposition as
self-contained as possible. We include one original contribution: a
new proof of the main result of \cite{IMN}, which states that the
double Schubert polynomials in types B, C, and D represent the
Schubert classes.  The setup in \cite{IMN} uses localization in
equivariant cohomology, which we do not require here. The key idea --
exploited in \cite{T4, T6, T7} -- is to use the elegant approach to
Schubert polynomials via the nilCoxeter algebra and the Yang-Baxter
equation, pioneered in \cite{FS, FK1, FK2}. One of the advantages of
this approach is that the Schubert polynomials are defined simply and
directly in terms of reduced decompositions in the Weyl group, without
requiring the use of a `top polynomial'. From this point of view, one
can also understand why the {\em stability} property of Schubert
polynomials is needed: it is only in the stable equivariant cohomology
ring that compatibility with divided differences alone (both left and
right!\ -- an important insight of \cite{IMN}) is enough to
characterize the universal Schubert classes, up to a scalar
factor. Anderson and Fulton \cite{AF} have recently also given a
different proof of the main theorem of \cite{IMN}, within the
framework of degeneracy loci, using a geometric argument which 
employs Kazarian's multi-Schur Pfaffians \cite{Ka1}.

We have made no attempt to write a survey, and in particular
the extensive literature on the Schubert calculus and the equivariant
cohomology of homogeneous spaces is barely touched upon. In special
cases, there are alternatives to the combinatorial formulas shown
here; the reader may consult \cite{AF, Ar, Bi2, BJS, BKT2, BKT4, FP,
  Ik, IN0, Ka1, KT, LaSa, La3, Mi, T4, T7, TW} for examples of what is
known, and the papers \cite{Br, BKT1, BKTY1, IMN, T4, T6} for further
references to related research. Throughout this article, we work with
cohomology groups, at times with rational coefficients. However, from
these, one can deduce results for cohomology with integer
coefficients, and also in the algebraic category, for the Chow groups
of algebraic cycles modulo rational equivalence. The necessary
modifications to achieve this are explained in detail in \cite{Br,
  EGr, Fu3, Gra}.

This article is organized as follows. We begin in \S \ref{class} 
and \S \ref{isogiam} with a discussion of Giambelli formulas for 
Grassmannians, expressing them using the language of raising operators. 
Section \ref{cohgp} contains general facts about the cohomology of 
$G/P$ spaces and the Giambelli problem in this context. The
combinatorial data coming from the Weyl group and the algebraic objects
necessary to state the general degeneracy locus formulas are given
in \S \ref{wgtts} and \S \ref{ssss}, respectively. In particular, 
\S \ref{splitsps} contains splitting formulas for Schubert polynomials,
which admit direct translations in \S \ref{dloc} to Chern class 
formulas for degeneracy loci. Section \ref{pfs} outlines the proofs 
of the main theorems, and \S \ref{fut} contains some questions
for the future.

This project would not have been possible without the contributions of
many authors, a list too long to mention here. I am particularly
grateful for the hard work and support of my collaborators Anders
Buch, Andrew Kresch, and Alexander Yong over a period of many years.
I also thank the anonymous referee for comments on an earlier version
of the paper.

\section{The Giambelli formula of classical Schubert calculus}
\label{class}

The main object of study in classical Schubert calculus is the
Grassmannian $X=\G(m,n)$, which is the set of all $m$-dimensional
complex linear subspaces of $V=\C^n$. Given any subset $H$ of $V$, we
let $\langle H \rangle$ denote the $\C$-linear span of $H$. Let
$e_1,\ldots,e_n$ denote the canonical basis of $\C^n$, and $d=n-m$ be
the codimension of the subspaces in $X$. The general linear group
$\GL_n(\C)$ acts transitively on $X$, and the stabilizer of the point 
$\langle e_1,\ldots, e_m\rangle$ under this action can be identified with
the subgroup $P$ of matrices in $\GL_n(\C)$ of the block form
\[
\left(
\begin{array}{c|c}
* & * \\  \hline
0 & *
\end{array}
\right),
\]
where the $0$ in the lower left corner denotes a $d\times m$
zero matrix.  In this way we get a description of $X$ as a coset space
\[
X=\GL_n(\C)/P
\]
from which one can deduce that $X$ is a complex manifold of dimension
$md$.  The subgroup $P$ is a {\em maximal parabolic subgroup} of
$\GL_n(\C)$.  A similar analysis shows that the manifold $X$ is
isomorphic to $\U(n)/(\U(m)\times \U(d))$, and hence is a compact
manifold.  In fact, $X$ is a projective algebraic variety, and may be
described by a system of quadratic polynomial equations, known as the
Pl\"ucker relations. For further details on this and other aspects of
this section, we refer to \cite{Fu4, Ma}.

In the latter half of the 19th century, Hermann Schubert gave a first
systematic treatment of enumerative projective geometry \cite{Sc1}, in
which the Grassmannian $X$ played a prominent part. The course of his
study led him to introduce certain natural closed algebraic subsets of
$X$, later known as the Schubert varieties \cite{Sc2}. To define them,
set $F_i=\langle e_1,\ldots,e_i\rangle$ for each integer $i\in [1,n]$,
and consider the complete flag of subspaces
\[
F_\bull\ :\ 0 = F_0 \subset F_1\subset\cdots\subset F_n=V.
\]
The stabilizer $B\subset \GL_n(\C)$ of $F_\bull$ is the {\em Borel
subgroup} of upper triangular matrices in $\GL_n(\C)$. 
In modern language, the Schubert varieties are the closures of the 
$B$-orbits in $X$. Each $B$-orbit in $X$ is called a Schubert cell;
there are finitely many such cells, and they induce a cell decomposition 
of the manifold $X$. 

We call a subset $\cP \subset [1,n]$ of cardinality $m$ an {\em index set\/}.
Any point $\Sigma \in X$ defines an index set $\cP(\Sigma)$ by
\[
\cP(\Sigma) := \{ p \in [1,n] \mid \Sigma \cap F_p \supsetneq \Sigma \cap
F_{p-1} \} \,.
\]
Observe that $\cP(\Sigma') = \cP(\Sigma)$ for any point $\Sigma'$ in
the orbit $B.\Sigma \subset X$. On the other hand, given any subspace
$\Sigma \subset V$, one can easily construct a basis $\{g_1,\ldots,
g_n\}$ of $V$ such that $F_i = \langle g_1,\dots,g_i \rangle$ for each
$i$ and $\Sigma = \langle \{g_1,\dots,g_n \} \cap \Sigma \rangle$. It
follows from this that any point $\Sigma' \in X$ such that
$\cP(\Sigma') = \cP(\Sigma)$ must be in the orbit $B.\Sigma$.  In
other words, the $B$-orbits (or Schubert cells) in $X$ correspond 1-1
to the index sets $\cP$.  We let $X^\circ_\cP(F_\bull)$ denote the
Schubert cell given by $\cP$, that is,
\begin{equation}
\label{Ptocell}
X^\circ_\cP(F_\bull) := \{ \Sigma \in X \mid \cP(\Sigma) = \cP \} \,. 
\end{equation}
The definition implies that we have a cell decomposition
\[
\G(m,n)=\coprod_\cP X^\circ_\cP(F_\bull).
\]

Suppose that $\cP=\{p_1<\cdots < p_m\}$ is an index set.  Any subspace
$\Sigma \in X^\circ_\cP(F_\bull)$ is spanned by the rows of a unique
$m\times n$ matrix $A=\{a_{ij}\}$ in a special reduced row echelon
form: there is a pivot entry 1 in position $(i,p_i)$, all other
entries in the $i$th row after the pivot are zero, and all entries
below the pivot entries are zero. If $j < p_i$ and $j \neq p_r$ for
all $r<i$, then $a_{ij}$ is a free variable, and gives an affine
coordinate for the Schubert cell $X^\circ_\cP$.  For example, if $m=4$,
$n=10$, and $\cP=\{3,5,6,9\}$, then
\[
A= \left(
\begin{array}{cccccccccc}
* & * & 1 & 0 & 0 & 0 & 0 & 0 & 0 & 0 \\
\ast & * & 0 & * & 1 & 0 & 0 & 0 & 0 & 0 \\
\ast & * & 0 & * & 0 & 1 & 0 & 0 & 0 & 0 \\
\ast & * & 0 & * & 0 & 0 & * & * & 1 & 0
\end{array}
\right).
\]
From this description we see that the $B$-orbit $X^\circ_\cP$ is
isomorphic to affine space $\C^{|\cP|}$, where $|\cP|$ is the number of
$*$'s in the reduced row echelon form of matrices in the cell, namely
\[
|\cP|=\sum_{j=1}^m(p_j-j).
\]

At this point it is convenient to introduce a different
parametrization for the Schubert cells which makes their
{\em codimension} apparent: let
\[
\la_j := d+j-p_j, \ \ 1\leq j\leq m.
\]
It is clear that there is a 1-1 correspondence between the vectors
$\la=(\la_1,\ldots,\la_m)$ and index sets $\cP$ (for fixed $m$ and
$n$); for example $\cP=\{3,5,6,9\}$ corresponds to $\la=(4, 3, 3, 1)$.
The conditions on the index set $\cP$ imply that
\[
d\geq \la_1\geq \la_2\geq \cdots \geq \la_m\geq 0 ,
\]
equivalently, that $\la=(\la_1,\ldots,\la_m)$ is a {\em partition}
whose first part $\la_1$ is at most $d$ and number of nonzero parts
$\la_j$ is at most $m$. Recall that any partition $\la$ can be
represented by a Young diagram of boxes, arranged in left-justified
rows, with $\la_j$ boxes in the $j$th row. The above conditions state
that the diagram of $\la$ is contained in an $m\times d$ rectangle,
which is the Young diagram of the partition $(d^m)=(d,\ldots,d)$. The
example shown below corresponds to a Schubert cell in $\G(4,10)$
indexed by the partition $\la=(5,4,2)$.
\[
\includegraphics[scale=0.2]{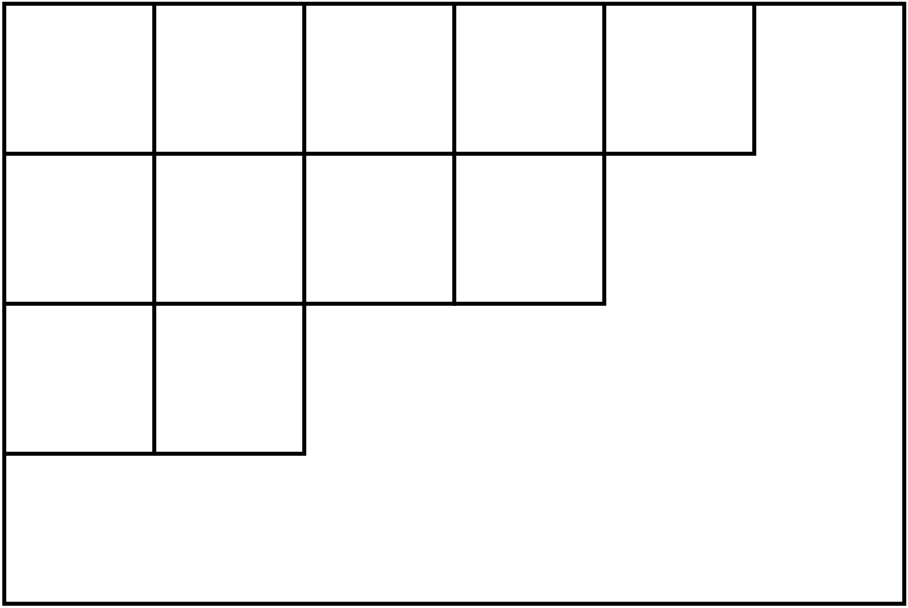} 
\]

We identify a partition with its Young diagram; an inclusion
$\la\subset\mu$ of partitions corresponds to the containment of their
respective diagrams.  The {\em weight} of $\la$, denoted $|\la|$, is
the total number of boxes in $\la$, hence $|\la|=\sum_{j=1}^m \la_j$,
and $\la$ is a partition of the integer $|\la|$. For each $\la$ as
above, we have a Schubert cell $X^\circ_{\la}$, which is equal to
$X^\circ_\cP$ for the index set $\cP$ corresponding to
$\la$. $X^\circ_{\la}$ has (complex) dimension
$$\sum_{j=1}^m(d-\la_j)=md-|\la|=\dim X - |\la|,$$ and therefore
codimension $|\la|$ in $X$.

The Schubert variety $X_\la(F_\bull)$ is the closure of the Schubert
cell $X^\circ_\la(F_\bull)$; it is an algebraic variety also of
codimension $|\la|$ in $X$. We have 
\[
X_{\la}(F_\bull)= \coprod_{\mu\supset\la}X^\circ_\mu(F_\bull) =
\{\Sigma\in X\ |\ \dim(\Sigma\cap F_{d+j-\la_j})\geq j, \
1\leq j \leq m \}.
\]
For each partition $\la$ contained in
$(d^m)$, let $[X_\la]\in \HH^{2|\la|}(X,\Z)$ denote the cohomology class
Poincar\'e dual to the cycle defined by $X_\la(F_\bull)$. If $F'_\bull$
is another complete flag, then there is an element $g$ in $\GL_n(\C)$
such that $g\cdot F_\bull = F'_\bull$. It follows that $[X_\la(F_\bull)]=
[X_\la(F'_\bull)]$, and therefore that the {\em Schubert class} $[X_\la]$ 
only depends on the partition $\la$, and not on the flag $F_\bull$. 
The cell decomposition of $X$ implies that the classes of the Schubert
varieties give a $\Z$-basis for $\HH^*(X,\Z)$. In other words, there is 
a direct sum decomposition
\[
\HH^*(X,\Z)=\bigoplus_{\la\subset (d^m)}\Z\,[X_{\la}].
\]

Of course, the cohomology $\HH^*(X,\Z)$ is also a commutative ring
under the cup product, and dually under the intersection product of
homology cycles. It follows that the structure of this ring is
determined by intersecting Schubert varieties in general position. The
simplest such varieties are the {\em special Schubert varieties}
\[
X_r(F_\bull)= \{\Sigma\in X\ |\ \Sigma\cap F_{d+1-r}\neq 0 \}
\] 
for $1\leq r \leq d$ (here the index $r$ is identified with the
partition $(r,0,\ldots,0)$). Historically, it was natural to focus on
the $X_r$ since these spaces are the easiest to work with
geometrically. The corresponding classes $[X_r]$ are the {\em special
  Schubert classes}. These cohomology classes can be realized as
characteristic classes of certain universal vector bundles over
$\G(m,n)$. Let $E'$ denote the tautological rank $m$ vector bundle
over $X$, $E$ the trivial rank $n$ vector bundle, and $E''=E/E'$ the
rank $d$ quotient bundle, so that we have a short exact sequence
\begin{equation}
\label{ses}
0 \to E' \to E \to E'' \to 0
\end{equation}
of vector bundles over $X$. Then $[X_r]$ is by definition the $r$-th
{\em Segre class} of $E'$, or equivalently, the $r$-th {\em Chern class}
of $E''$, denoted $c_r(E'')$.

The work of Pieri \cite{Pi} and Giambelli \cite{G1} established that
the special classes $c_r = c_r(E'')$ generate the cohomology ring
$\HH^*(X,\Z)$. Giambelli proved the following explicit
formula which writes a general Schubert class $[X_\la]$ as a
polynomial in special classes:
\begin{equation}
\label{giambelli}
[X_\la] = \det(c_{\la_i + j-i}(E''))_{1\leq i,j \leq m}.
\end{equation}
In equation (\ref{giambelli}) and in the remainder of this paper, our
convention is that $c_0=1$ and $c_r=0$ whenever $r<0$. Observe that
there are relations among the $c_r$ in $\HH^*(X,\Z)$, so that the
right hand side of formula (\ref{giambelli}) is not unique. However,
the natural inclusion $\G(m,n)\hra \G(m+1,n+1)$ induces a surjection
$$\HH^*(\G(m+1,n+1),\Z)\to \HH^*(\G(m,n),\Z).$$ For a fixed partition
$\la$ and codimension $d$, the Giambelli polynomial in
(\ref{giambelli}) is the unique one that is preserved under the above
map, for all $m$ greater than or equal to the number of (nonzero)
parts of $\la$.

For our purposes here it will be important to rewrite formula
(\ref{giambelli}) using A.\ Young's raising operators \cite{Y}.
An {\em integer sequence} is a sequence of integers $\al=(\al_1,
\al_2,\ldots)$ only finitely many of which are nonzero.  Given any
integer sequence $\alpha$ and natural numbers $i<j$, we define
\[
R_{ij}(\alpha) := (\alpha_1,\ldots,\alpha_i+1,\ldots,\alpha_j-1,
\ldots).
\] 
A {\em raising operator} $R$ is any monomial in these $R_{ij}$'s.  If
$(c_1,c_2,\ldots)$ is any ordered set of commuting independent
variables, we let $c_{\al} := \prod_{i\geq 1}c_{\al_i}$, with the
understanding that $c_0=1$ and $c_r = 0$ if $r<0$. For any raising
operator $R$, set $R\,c_{\al} := c_{R\al}$ (note that we slightly abuse
the notation here and consider that the raising operator $R$ acts on
the index $\al$, and not on the monomial $c_\al$ itself). Consider the
raising operator expression
\[
R^{0} := \prod_{i<j}(1-R_{ij}).
\]
If we expand the infinite product $R^{0}$ as a formal power series in
the $R_{ij}$ and apply the result to $c_\al$, only finitely many of
the summands are nonzero. Therefore $R^{0}\,c_\al$ is a well defined
polynomial in the variables $c_r$, and in fact we have
\begin{equation}
\label{vand}
R^{0}\,c_\al = \det(c_{\al_i+j-i})_{i,j}.
\end{equation}
Equation (\ref{vand}) is a formal consequence of the Vandermonde 
identity
\[
\prod_{1\leq i<j \leq m} (\x_i-\x_j) =
\det(\x_i^{m-j})_{1\leq i,j \leq m} \ ;
\]
for a proof of this see for example \cite{T5}. It follows that we may
rewrite (\ref{giambelli}) as 
\begin{equation}
\label{giambelli2}
[X_\la] = R^{0}\, c_\la(E''),
\end{equation}
where $c_\la(E'') = \prod_ic_{\la_i}(E'')$ denotes a monomial in the Chern
classes of $E''$, and $R^0$ is applied to $c_\la$ as above.

\begin{example}
We have
\begin{gather*}
[X_{(5,4,2)}] =(1-R_{12})(1-R_{13})(1-R_{23})\, c_{(5,4,2)} \\
= (1-R_{12}-R_{13}-R_{23}+R_{12}R_{13}+R_{12}R_{23}+R_{13}R_{23}-
R_{12}R_{13}R_{23})
\, c_{(5,4,2)} \\
= c_{(5,4,2)}-c_{(6,3,2)}-c_{(6,4,1)}-c_{(5,5,1)}+c_{(7,3,1)}+c_{(6,4,1)} 
+c_{(6,5,0)}-c_{(7,4,0)}
 \\ 
=c_5c_4c_2-c_6c_3c_2-c_5^2c_1+c_7c_3c_1+c_6c_5-c_7c_4
= \left|\begin{array}{ccc}
c_5 & c_6 & c_7 \\ c_3 & c_4 & c_5 \\ 1 & c_1 & c_2
\end{array}\right|.
\end{gather*}
\end{example}

Soon after he proved (\ref{giambelli}), Giambelli published a second 
paper \cite{G2} where he studied a parallel formalism in the theory of
symmetric polynomials. For any integer $r$, let $e_r(Y_{(d)})$ denote the
$r$-th elementary symmetric polynomial in the commuting variables
$Y_{(d)}=(y_1,\ldots,y_d)$. Given a partition $\mu$ with at most $d$
nonzero parts, consider
\begin{equation}
\label{def1}
s_\mu(Y_{(d)}) =
\left.\det(y_i^{\mu_i+d-j})_{1\leq i,j \leq d}
\right\slash\det(y_i^{d-j})_{1\leq i,j \leq d}.
\end{equation}
The $s_{\mu}(Y_{(d)})$ for varying $\mu$ may be identified with the
polynomial characters of the general linear group $\GL_d(\C)$; this
had been established a few years earlier by Schur in his 1901 thesis
\cite{S1} (in fact, equation (\ref{def1}) is a special case of the
Weyl character formula).  For any partition $\la$, let $\wt{\la}$ be
the conjugate partition, whose Young diagram is the transpose of the
diagram of $\la$. Then Jacobi \cite{J} and Trudi proved that the {\em
  Schur polynomial} $s_{\wt{\la}}(Y_{(d)})$ satisfies
\begin{equation}
\label{JT}
s_{\wt{\la}}(Y_{(d)})=R^0\, e_{\la}(Y_{(d)})=\det(e_{\la_i +
  j-i}(Y_{(d)}))_{i,j}
\end{equation}
for any $\la\subset (d^m)$, where $e_\la :=\prod_ie_{\la_i}$. We may
thus consider $s_{\wt{\la}}$ as a polynomial in the algebraically
independent variables $e_r$, for $1\leq r \leq d$. In the theory of
characteristic classes, the variables $y_1,\ldots, y_d$ represent the
Chern roots of the quotient vector bundle $E''$, and $e_r(Y_{(d)})$ is
identified with the $r$-th Chern class $c_r(E'')$. Using $c(E'')$ to
denote the total Chern class $1+c_1(E'')+ \cdots +c_d(E'')$ of
$E''$, we obtain the following restatement of equations
(\ref{giambelli}) and (\ref{giambelli2}).

\begin{thm}[Classical Giambelli, \cite{G1}]
\label{Gthm}
For any partition $\la$ whose diagram fits inside an
$m\times (n-m)$ rectangle, we have
\begin{equation}
\label{giam3}
[X_\la] = s_{\wt{\la}}(c(E''))
\end{equation}
in the cohomology ring of $\G(m,n)$.
\end{thm}

For more on the connection between the representation
theory of the general linear group and the classical Schubert
calculus, see \cite{Be, BK, Tcon}.

\section{Giambelli formulas for isotropic Grassmannians}
\label{isogiam}

The study of homogeneous spaces of Lie groups was extended further
during the first half of the twentieth century by the work of \'Elie
Cartan \cite{Ca1, Ca2} and Ehresmann \cite{Eh}. They considered the
irreducible compact hermitian symmetric spaces, which generalize the
Grassmannian $\G(m,n)$, and began exploring their cohomology
rings. Rather than proceeding along the lines of the classical
Schubert calculus, this work used Cartan's theory of invariant
differential forms. It was only in the 1980s that analogues of Pieri's
rule and Giambelli's formula were obtained for all hermitian
symmetric Grassmannians, in the work of Hiller and Boe \cite{HB} and
Pragacz \cite{P2}. More recently, Pragacz and Ratajski \cite{PR1, PR3}
proved Pieri type rules and Buch, Kresch, and the author \cite{BKT1,
  BKT2, BKT4} generalized both the Pieri and Giambelli formulas of
\cite{HB, P2} to arbitrary symplectic and orthogonal Grassmannians,
using different notions of special Schubert classes.  We will follow
the references \cite{BKT2, BKT4} in this section.

Let $V = \C^N$ and equip $V$ with a nondegenerate skew-symmetric or
symmetric bilinear form $(\ ,\,)$.  A subspace $\Sigma$ of $V$ is
called {\em isotropic} if the restriction of $(\ ,\,)$ to $\Sigma$
vanishes identically. Since the form is nondegenerate, the dimension
of any isotropic subspace is at most $N/2$.  Given a nonnegative
integer $m \leq N/2$, we let $X$ denote the complex manifold which
parametrizes all the isotropic subspaces of dimension $m$ in $V$. This
space has a transitive action of the group $G =\Sp(V)$ or $G=\SO(V)$
of linear automorphisms preserving the form on $V$, unless $m=N/2$ and
the form is symmetric. In the latter case the space of isotropic
subspaces has two isomorphic connected components, each a single
$\SO(V)$ orbit.

An {\em isotropic flag} $F_\bull$ is a complete flag $$0 = F_0
\subsetneq F_1 \subsetneq \dots \subsetneq F_N = V$$ of subspaces of
$V$ such that $F_i = F_j^\perp$ whenever $i+j=N$; in particular, $F_i$ 
is an isotropic subspace for all $i\leq N/2$. Let $B\subset G$
denote the Borel subgroup which is the stabilizer of the flag
$F_\bull$. The Schubert cells in $X$ relative to the flag $F_\bull$
are the orbit closures for the natural action of $B$ on $X$. We call a
subset $\cP$ of $[1,N]$ of cardinality $m$ an {\em index set} if for
all $i,j\in \cP$ we have $i+j\neq N+1$. A point $\Sigma$ in $X$
defines an index set $\cP(\Sigma)$ by the prescription
\[
\cP(\Sigma) := \{ p \in [1,N] \mid \Sigma \cap F_p \supsetneq \Sigma \cap
F_{p-1} \} \,,
\]
since no vector in $F_j\ssm F_{j-1}$ is orthogonal to a vector in
$F_{N+1-j}\ssm F_{N-j}$, for each $j$. In the same manner as in \S
\ref{class}, equation (\ref{Ptocell}) establishes a one to one
correspondence between Schubert cells $X^\circ_\cP(F_\bull)$ relative
to $F_\bull$ and index sets $\cP$.

The closures of the Schubert cells are the Schubert varieties
$X_\cP(F_\bull)$, and their classes $[X_\cP]$ in $\HH^*(X,\Z)$ are the
Schubert classes, which form an additive basis of $\HH^*(X,\Z)$.
Moreover, there are special Schubert varieties, consisting of the
locus of subspaces $\Sigma$ in $X$ which meet a given subspace $F_j$
non-trivially, and corresponding special Schubert classes, which
generate the cohomology ring of $X$. In the following sections, we
will see that the Schubert varieties and classes may equivalently be
indexed by {\em $k$-strict partitions} and {\em typed $k$-strict
  partitions}. As in \S \ref{class}, this is a convention which makes
their codimension (or cohomological degree) apparent, and we will
require in order to state the Giambelli formulas of this section.

\subsection{Symplectic Grassmannians}
\label{sgs}

Suppose that $N=2n$ is even and the form $(\ ,\,)$ is skew-symmetric,
so that $X=\IG(m,2n)$ is a {\em symplectic Grassmannian}. Write
$m=n-k$ for some $k$ with $0 \leq k \leq n-1$. If $k=0$, then $X$ is
the Lagrangian Grassmannian $\LG(n,2n)$, and if $k=n-1$, then $X$ is
projective $2n-1$ space $\bP^{2n-1}$, since every line through the
origin in $V$ is isotropic. These are the only hermitian symmetric
examples. The space $\IG(n-k,2n)$ may be identified with a quotient
$\Sp_{2n}(\C)/P_k$ of the symplectic group $\Sp_{2n}(\C)$ by a maximal
parabolic subgroup $P_k$.  For instance, the subgroup $P_0$ is known
as the {\em Siegel parabolic} and consists of those matrices of the
symplectic group whose lower left quadrant is an $n\times n$ zero
matrix (when we choose the standard symplectic basis for $\C^{2n}$, as
in \S \ref{ddgeom}).

 Consider the infinite set of pairs 
 $$\Delta^\circ = \{(i,j) \in \N \times \N \mid 1\leq i<j \}$$ and
 define a partial order on $\Delta^\circ$ by agreeing that
 $(i',j')\leq (i,j)$ if $i'\leq i$ and $j'\leq j$. A subset $D$ of
 $\Delta^\circ$ is an order ideal if $(i,j)\in D$ implies $(i',j')\in
 D$ for all $(i',j')\in \Delta^\circ$ with $(i',j') \leq (i,j)$. In
 the next figure, the pairs $(i,j)$ in a typical finite order ideal
 are displayed as positions in a matrix above the main diagonal.

\[
\includegraphics[scale=0.5]{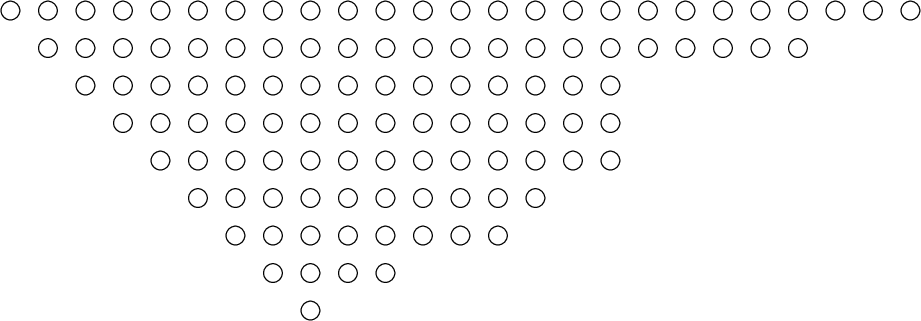}
\]

A partition $\la$ is {\em $k$-strict} if no part $\la_j$ greater than
$k$ is repeated; if $k=0$ this means that $\la$ is a strict partition,
i.e., has distinct nonzero parts.  To any $k$-strict partition $\la$
we associate the order ideal
\[
\cC(\la) := \{ (i,j)\in \Delta^\circ \ |\ \la_i + \la_j > 2k + j-i \}.
\]
The set $\cC(\la)$ is an important invariant of the partition $\la$
which appears in the definitions of both the index set $\cP(\la)$ and
the raising operator expression $R^\la$ associated to $\la$ (equations
(\ref{Pequ}) and (\ref{ReqC}) below; compare also with (\ref{Cweq})).

The Schubert varieties on $\IG(n-k,2n)$ are indexed by $k$-strict
partitions whose diagrams fit in an $(n-k)\times (n+k)$ rectangle.
Any such $\la$ corresponds to an index set $$\cP(\la)=\{p_1(\la)<\cdots
< p_m(\la)\}$$ given by the prescription
\begin{equation}
\label{Pequ}
p_j(\lambda) := n+k+j-\lambda_j - \#\{i<j \ |\  (i,j)\in \cC(\la) \}.
\end{equation}
The reader is invited to show that equation (\ref{Pequ}) gives a
bijection between index sets and $k$-strict partitions as claimed (a
proof is provided in \cite[\S 4.1]{BKT1}).  If $F_\bull$ is a fixed
isotropic flag of subspaces in $V$, we obtain the Schubert cell
\begin{equation}
\label{IGX1}
X^\circ_\la(F_\bull) := \{ \Sigma \in X \mid \cP(\Sigma) = \cP(\la) \} 
\end{equation}
and the Schubert variety $X_\lambda(F_\bull)$ is the closure of this cell.
One can show that
\begin{equation}
\label{IGX2}
   X_\lambda(F_\bull) = \{ \Sigma \in X \mid \dim(\Sigma \cap
   F_{p_j(\lambda)}) \geq j \ \ \ \forall\, 1 \leq j \leq
   \ell(\lambda) \} \,,
\end{equation}
where the {\em length} $\ell(\la)$ is the number of (nonzero) parts of
$\la$; see for example \cite[\S 4.2]{BKT1} and \cite[App.\ A]{BKT4}.
This variety has codimension $|\la|$ and defines, using Poincar\'e
duality, a Schubert class $[X_{\la}]$ in $\HH^{2|\la|}(\IG,\Z)$.

The {\em special Schubert varieties} are given by
\[
X_r(F_\bull)= \{\Sigma\in X\ |\ \Sigma\cap F_{n+k+1-r}\neq 0 \}
\] 
for $1\leq r \leq n+k$, and their classes $[X_r]$ are the {\em special
  Schubert classes}. Let $E'$ denote the tautological rank $(n-k)$
vector bundle over $\IG(n-k,2n)$, $E$ the trivial rank $2n$ vector
bundle, and $E''=E/E'$ the quotient bundle. As in the example of the
type A Grassmannian in \S \ref{class}, we have $[X_r]=c_r(E'')$ for
$1\leq r \leq n+k$.

For any $k$-strict partition $\la$, we define the operator
\begin{equation}
\label{ReqC}
R^{\la} := \prod_{i<j} (1-R_{ij})\prod_{(i,j) \in \cC(\la)}
(1+R_{ij})^{-1},
\end{equation}
where the first product is over all pairs $i<j$ and second product is
over pairs $i<j$ such that $\la_i+\la_j > 2k+j-i$. Note that equation
(\ref{ReqC}) is a multiplicative analogue of equation (\ref{Pequ}). If
$c= 1 + c_1t+c_2t^2+\cdots$ is any formal power series in commuting
variables $c_r$, we define the {\em theta polynomial} $\Ti_\la$ by
\begin{equation}
\label{Thet}
\Ti_\la(c) := R^{\la}\, c_\la.
\end{equation}

\begin{thm}[Giambelli for $\IG$, \cite{BKT2}]
\label{IGthm}
For any $k$-strict partition $\la$ whose diagram fits inside an
$(n-k)\times (n+k)$ rectangle, we have
\begin{equation}
\label{Theq}
[X_\la] =\Ti_\la(c(E''))
\end{equation}
in the cohomology ring of $\IG(n-k,2n)$.
\end{thm}

\begin{example}
Let $k=1$ and $\la$ be the $1$-strict partition $(3,1,1)$, with 
$\cC(\la)=\{(1,2)\}$. Then the following computation holds in the
ring $\HH^*(\IG(4,10),\Z)$:
\[
[X_{(3,1,1)}] =
\frac{1-R_{12}}{1+R_{12}}(1-R_{13})(1-R_{23})\, c_{(3,1,1)} =
(1-2R_{12}+2R_{12}^2)(1-R_{13}-R_{23}) \, c_{(3,1,1)}
\]
\[
= c_{(3,1,1)} - 2 c_{(4,0,1)} - c_{(4,1)} +2 c_5 -c_{(3,2)} +2 c_{(4,1)} -2 c_5=
c_3c_1^2 - c_4 c_1 - c_3 c_2.
\]
\end{example}

Comparing (\ref{Theq}) with (\ref{giambelli2}) and (\ref{giam3}), we
see that the polynomials $\Ti_\la$ play the role of the Schur
polynomials for the Giambelli problem on $\IG(n-k,2n)$. An important
difference with the story for the type A Grassmannian is that there
are {\em relations} among the $c_r$ which persist even as
$n\to\infty$, namely:
\begin{equation}
\label{crels}
\frac{1-R_{12}}{1+R_{12}}\, c_{(r,r)} = 
c_r^2 + 2\sum_{i=1}^r(-1)^i c_{r+i}c_{r-i}= 0,
\ \  \text{for all} \ \,  r > k.
\end{equation}
Another difference is that the raising operator expressions $R^\la$
which enter in (\ref{Thet}) depend on the partition $\la$ (see also
Example \ref{interpolate} below).

We next specialize the above to the Lagrangian Grassmannian
$\LG(n,2n)$, which is the case where $k=0$. 
The Schubert classes in $\HH^*(\LG(n,2n),\Z)$
are indexed by strict partitions $\la$ whose diagrams fit inside a 
square of side $n$. The theta polynomial (\ref{Thet}) specializes to
a {\em $Q$-polynomial}
\begin{equation}
\label{Qdef}
Q_\la(c) := \prod_{i<j}\frac{1-R_{ij}}{1+R_{ij}}\,c_{\la}.
\end{equation}
The Giambelli formula of Theorem \ref{IGthm} for $\LG$ becomes
\begin{equation}
\label{LGgiam}
[X_\la] = Q_\la(c(E'')).
\end{equation}
Following Pragacz \cite{P2}, formula (\ref{LGgiam}) may
be expressed using a Schur Pfaffian, as follows. For partitions $\la=(a,b)$
with only two parts, we have
\begin{equation}
\label{giamQ0}
[X_{(a,b)}]=\frac{1-R_{12}}{1+R_{12}}\, c_{(a,b)} = 
c_ac_b-2c_{a+1}c_{b-1}+2c_{a+2}c_{b-2}-\cdots
\end{equation}
while for $\la$ with $3$ or more parts,
\begin{equation}
\label{giamQ}
[X_\la] = \Pf([X_{(\la_i,\la_j)}])_{1\leq i<j \leq 2\ell'}
\end{equation}
where $\ell'$ is the least positive integer such that $2\ell' \geq
\ell(\la)$.

As we alluded to above, the identities (\ref{giamQ0}) and
(\ref{giamQ}) go back to the work of Schur on the projective
representations of symmetric groups \cite{HH, Jo, S2, St1}, where he
introduced a family of symmetric functions $\{Q_\la(X)\}$ known as
Schur $Q$-functions.  We let $X=(x_1,x_2,\ldots)$ be a list of
variables, define $q_r(X)$ by the equation
\[
\prod_{i=1}^{\infty}\frac{1+x_it}{1-x_it} = \sum_{r=0}^{\infty}q_r(X)t^r
\]
and then use the same relations (\ref{giamQ0}) and (\ref{giamQ}) with
$q_r(X)$ in place of $c_r$ to define $Q_{(a,b)}(X)$ and then
$Q_\la(X)$, for each strict partition $\la$. Once more we emphasize
that there are relations among the $q_r$, the simplest being
$q_1^2=2q_2$; hence, the above polynomials which define $Q_\la(X)$ are
not uniquely determined. However, the equivalence of the raising
operator and Pfaffian definitions of $Q_\la$ is a formal consequence
of the following Pfaffian identity from \cite{S2}:
\[
\prod_{1\leq i<j \leq 2\ell'}\frac{\x_i-\x_j}{\x_i+\x_j} =
\Pf\left(\frac{\x_i-\x_j}{\x_i+\x_j}\right)_{1\leq i,j \leq 2\ell'},
\]
which holds in the quotient field of $\Z[\x_1,\ldots,\x_{2\ell'}]$.

\begin{remark}
\label{rmk1}
Although the definition (\ref{Qdef}) is not standard, it is in direct
analogy with the usage of the term `Schur polynomial' in type A. We
reserve the name `$Q$-polynomial' for the polynomial in the variables
$c_r$ given in (\ref{Qdef}), and also for its principal
specialization, when $c_r$ is replaced by $q_r(X)$ for each integer
$r$. This nomenclature extends to the theta polynomials; compare
(\ref{Thet}) with (\ref{Tidefn}) in \S \ref{ste}.
\end{remark}

\begin{example}
\label{interpolate}
Let $\la=(\la_1,\la_2,\ldots)$ be a $k$-strict partition. If
$\la_i\leq k$ for each $i$, then
\begin{equation}
\label{Schurdet}
\Theta_\la(c) = \prod_{i<j}(1-R_{ij})\, c_\la  = 
 \det(c_{\la_i+j-i})_{i,j},
\end{equation}
while if $\la_i>k$ for all nonzero parts $\la_i$, then
\begin{equation}
\label{SchurPf}
\Theta_\la(c) = \prod_{i<j}\frac{1-R_{ij}}{1+R_{ij}}\,c_{\la} = 
\Pf\left(\frac{1-R_{12}}{1+R_{12}}\,c_{\la_i,\la_j}\right)_{i<j}.
\end{equation}
We deduce that as $\la$ varies, the polynomial $\Theta_\la(c)$
interpolates between the Jacobi-Trudi determinant (\ref{Schurdet}) and
the Schur Pfaffian (\ref{SchurPf}). In general, the inclusion of a
pair $(i,j)$ in the set $\cC(\la)$ which specifies the denominators in
(\ref{ReqC}) depends not only on the size of $\la_i$ and $\la_j$, but
also on their relative position in the sequence $\la$.
\end{example}

\subsection{Orthogonal Grassmannians}
\label{ogs}

Consider the case where $m < N/2$ and the form $(\ ,\,)$ is symmetric,
so that $X=\OG(m,N)$ is an {\em orthogonal Grassmannian}. If $N=2n+1$
is odd, then the Schubert varieties in $\OG(m,2n+1)$ are indexed by
the same set of $k$-strict partitions that index the Schubert
varieties in $\IG(m,2n)$. For any $k$-strict partition $\la$, let
$\ell_k(\la)$ denote the number of parts $\la_i$ of $\la$ which are
strictly greater than $k$. Let $E''_{\IG}$ and $E''_{\OG}$ be the
universal quotient vector bundles over $\IG(n-k,2n)$ and
$\OG(n-k,2n+1)$, respectively. One then knows (see for example
\cite[\S 3.1]{BS}) that the map which sends $c_r(E''_{\IG})$ to
$c_r(E''_{\OG})$ for all $r$ extends to an isomorphism of graded rings
\[
\HH^*(\IG(n-k,2n),\Q) \stackrel{\sim}\longrightarrow \HH^*(\OG(n-k,2n+1),\Q),
\]
which sends a Schubert class
$[X_\la]$ on $\IG$ to $2^{\ell_k(\la)}$ times the corresponding
Schubert class on $\OG$. This isomorphism shows that the Schubert
calculus on symplectic and odd orthogonal Grassmannians coincides, up
to well determined powers of two. In particular, one can easily
transfer the Giambelli formula of \S \ref{sgs} to $\OG(m,2n+1)$.

We next assume that $N=2n$ is even, so that $m=n-k$ with $k> 0$. A
{\em typed $k$-strict partition} is a pair consisting of a $k$-strict
partition $\la$ together with an integer in $\{0,1,2\}$ called the
{\em type} of $\la$, and denoted $\type(\la)$, such that
$\type(\la)>0$ if and only if $\la_i=k$ for some $i\geq 1$. We usually
omit the type from the notation for the pair $(\la,\type(\la))$. To
any typed $k$-strict partition $\la$ we associate the order ideal
\[
\cC'(\la) := \{ (i,j)\in \Delta^\circ \ |\ \la_i + \la_j \geq 2k + j-i \}
\]
in $\Delta^\circ$. The Schubert cells in the cohomology of the even
orthogonal Grassmannian $X=\OG(n-k,2n)$ are indexed by the typed
$k$-strict partitions $\la$ whose diagrams are contained in an
$(n-k)\times (n+k-1)$ rectangle.  For any such $\la$, define the index
function $p_j=p_j(\la)$ by
\begin{multline*} p_j(\lambda) := n+k+j-\lambda_j - 
   \#\{\,i<j\ |\ (i,j)\in \cC'(\la)\,\} \\
   {} - \begin{cases} 
      1 & \text{if $\lambda_j > k$, or $\lambda_j=k < \lambda_{j-1}$ and
        $n+j+\type(\lambda)$ is odd}, \\
      0 & \text{otherwise}.
   \end{cases}
\end{multline*}
We obtain an index set $\cP(\la)$ associated to any typed $k$-strict 
partition $\la$ as above, and a Schubert cell $X^\circ_\cP(F_\bull)$ 
defined by (\ref{IGX1}).

The Schubert variety $X_\cP(F_\bull)$ is best defined as the closure
of $X^\circ_\cP(F_\bull)$, since a geometric description of
$X_\cP(F_\bull)$ analogous to (\ref{IGX2}) involves subtle parity
conditions (see \cite[App.\ A]{BKT4}).  We say that two maximal
isotropic subspaces $E$ and $F$ of $V$ are in the same family if
$$\dim(E\cap F)\equiv n \,(\text{mod}\, 2).$$ Fix a maximal isotropic
subspace $L$ of $V$, so that $\dim(L)=n$. The cohomology classes
$[X_\cP]$ of the Schubert varieties $X_\cP(F_\bull)$ in $\OG$ are are
independent of the choice of isotropic flag $F_\bull$ as long as $F_n$
is in the same family as $L$. If $\cP$ corresponds to $\la$, then the
associated Schubert class $[X_\la]=[X_\cP]$ in $\HH^{2|\la|}(X,\Z)$ is
said to have a type which agrees with the type of $\la$.

The special Schubert varieties in $\OG(m,2n)$ can be defined as before
by a single Schubert condition, as the locus of $\Sigma\in X$ which
intersect a given isotropic subspace or its orthogonal complement
non-trivially (see \cite[\S 3.2]{BKT1}).  The corresponding special
Schubert classes
\begin{equation}
\label{spcls}
\ta_1,\ldots,\ta_{k-1},\ta_k,\ta'_k,\ta_{k+1},\ldots,\ta_{n+k-1}
\end{equation}
are indexed by the typed $k$-strict partitions with a single nonzero
part, and generate the cohomology ring $\HH^*(X,\Z)$.  Here
$\type(\tau_k)=1$, $\type(\tau'_k)=2$, and if 
\[
0 \to E' \to E \to E'' \to 0
\] 
denotes the universal sequence of vector bundles over $X$, then we have 
\begin{equation}
\label{ctotau}
c_r(E'')=
\begin{cases}
\ta_r &\text{if $r< k$},\\
\ta_k+\ta_k' &\text{if $r=k$},\\
2\ta_r &\text{if $r> k$}.
\end{cases}
\end{equation}

We set $c_\alpha = \prod_i c_{\alpha_i}$. Given any typed $k$-strict
partition $\la$, we define the operator
\begin{equation}
\label{ReqD}
R^{\la} := \prod (1-R_{ij})\prod_{(i,j)\in \cC'(\la)}(1+R_{ij})^{-1}
\end{equation}
where the first product is over all pairs $i<j$ and the second product
is over pairs $i<j$ such that $\la_i+\la_j \geq 2k+j-i$. Let $R$ be
any finite monomial in the operators $R_{ij}$ which appears in the
expansion of the formal power series $R^\la$ in (\ref{ReqD}).  If
$\type(\la)=0$, then set $R \star c_{\la} := c_{R \,\la}$. Suppose
that $\type(\la)>0$, let $d$ be the least index such that $\la_d=k$,
and set $$\wh{\al} :=
(\al_1,\ldots,\al_{d-1},\al_{d+1},\ldots,\al_\ell)$$ for any integer
sequence $\al$ of length $\ell$. If $R$ involves any factors $R_{ij}$
with $i=d$ or $j=d$, then let $R \star c_{\la} := \frac{1}{2}\,c_{R
  \,\la}$. If $R$ has no such factors, then let
\[
R \star c_{\la} := \begin{cases}
\ta_k \,c_{\wh{R \,\la}} & \text{if  $\,\type(\la) = 1$}, \\
\ta'_k \, c_{\wh{R \,\la}} & \text{if  $\,\type(\la) = 2$}.
\end{cases}
\]
We define the {\em eta polynomial} $\Eta_\la$ by
\begin{equation}
\label{Etapoly}
\Eta_\la(c) := 2^{-\ell_k(\la)} R^{\la} \star c_\la.
\end{equation}
Note that $\Eta_\la(c)$ is really a polynomial in the variables
$\ta_r$ for $r\geq 1$ and $\ta_k'$, which are related to the variables $c_r$ 
by the formal equations (\ref{ctotau}).

\begin{thm}[Giambelli for $\OG$, \cite{BKT4}]
\label{OGthmD}
For every typed $k$-strict partition $\la$ whose diagram fits inside an
$(n-k)\times (n+k-1)$ rectangle, we have 
\begin{equation}
\label{OGgiamb}
[X_\la] = \Eta_\la(c(E''))
\end{equation}
in the cohomology ring of $\OG(n-k,2n)$. 
\end{thm}

\begin{example}
Consider the typed $1$-strict partition $\la=(3,1,1)$ with
$\type(\la)=2$.  The corresponding Schubert class in
$\HH^*(\OG(4,10),\Z)$ satisfies
\begin{align*}
[X_{(3,1,1)}] & = \frac{1}{2}
\frac{1-R_{12}}{1+R_{12}}\,\frac{1-R_{13}}{1+R_{13}}\,(1-R_{23}) \star
c_{(3,1,1)} \\ &= \frac{1}{2}(1-2R_{12}+2R_{12}^2)(1-2R_{13}-R_{23})
\star c_{(3,1,1)} \\ &= \ta_3\ta'_1(\ta_1+\ta_1') - 2\ta_4\ta'_1 -
\ta_3\ta_2 +\ta_5.
\end{align*}
In general, the Giambelli formula (\ref{OGgiamb}) expresses
the Schubert class $[X_\la]$ as a polynomial in the special Schubert
classes (\ref{spcls}) with {\em integer} coefficients.
\end{example}

To conclude this section, we consider the space of all maximal
isotropic subspaces of an even dimensional orthogonal vector space
$V$. This set is a disjoint union of two connected components, each
giving one family (or $\SO(V)$-orbit) of such subspaces. The
(irreducible) algebraic variety $\OG=\OG(n,2n)$ is defined by choosing
one of these two components. If $H$ is a hyperplane in $V$ on which
the restriction of the symmetric form is nondegenerate, then the map
$\Sigma \mapsto \Sigma\cap H$ gives an isomorphism between $\OG(n,2n)$
and $\OG(n-1,2n-1)$. The analysis for the (type B) odd orthogonal
Grassmannian in the maximal isotropic case therefore applies to
$\OG$. In particular, the Schubert classes $X_\la$ may be indexed by
strict partitions $\la$ with $\la_1\leq n-1$, and the Giambelli
formula for $\OG$ reads
\[
[X_\la] = P_\la(c(E'')),
\]
where $E''\to \OG$ is the universal quotient bundle and the $P$-polynomial
$P_\la$ is related to $Q_\la$ by the equation $P_\la = 2^{-\ell(\la)}Q_\la$.

\section{Cohomology of  $G/P$ spaces}
\label{cohgp}

In this section $G$ will denote a connected complex reductive Lie
group.  The main examples we will consider are the classical groups:
the general linear group $\GL_n(\C)$, the symplectic group
$\Sp_{2n}(\C)$, and the (odd and even) orthogonal groups
$\SO_N(\C)$. A closed algebraic subgroup $P$ of $G$ is called a
parabolic subgroup if the quotient $X=G/P$ is compact, or
equivalently, a projective algebraic variety. We proceed to study the
topology of the complex manifold $X$ by generalizing the
constructions found in the previous sections.

The Borel subgroups $B$ of $G$ are the maximal connected solvable
subgroups; a subgroup $P$ of $G$ is parabolic if and only if it
contains a Borel subgroup. Fix a Borel subgroup $B$, let $T\cong
(\C^*)^r$ be a maximal torus in $B$, and $W=N_G(T)/T$ be the Weyl
group of $G$. The simple reflections $s_\al$, which generate $W$, are
indexed by the set $\Delta$ of positive simple roots $\al$, and are
also in one-to-one correspondence with the vertices of the {\em Dynkin
  diagram} $D$ associated to the root system.  The parabolic subgroups
$P$ containing $B$ are in bijection with the subsets $\Delta_P$ of
$\Delta$, or the subsets of the vertices of $D$, as shown in the
figures of \S \ref{dloc}. In particular each root $\al\in \Delta$
corresponds to a maximal parabolic subgroup, which is associated to
the subset $\Delta\ssm\{\al\}$.  We let $W_P$ denote the subgroup of
$W$ generated by all the simple reflections in $\Delta_P$.

The {\em length} $\ell(w)$ of an element $w$ in $W$ is equal to the
least number of simple reflections whose product is $w$. We denote by
$w_0$ the element of longest length in $W$. It is known that every
coset in $W/W_P$ has a unique representative $w$ of minimal length; we
denote the set of all minimal length $W_P$-coset representatives by
$W^P$. The manifold $X$ has complex dimension equal to the length of
the longest element in $W^P$. The set $W^P$ (or the coset space
$W/W_P$) indexes the Schubert cells, varieties, and classes on $X$ as
follows. First, the Bruhat decomposition $G= \coprod_{w\in W} BwB$ of
the group $G$ induces a cell decomposition
\[
G/P = \coprod_{w\in W^P} BwP/P
\]
of the homogeneous space $X=G/P$. The cell $BwP/P$ is isomorphic to
the affine space $\C^{\ell(w)}$. Since we prefer the length of the
indexing element to equal the codimension of the cell in $X$, we
define the Schubert cell $X^\circ_w$ to be $Bw_0wP/P$. The Schubert
variety $X_w$ is defined as the closure of the Schubert cell
$X_w^\circ$, and its cohomology class $[X_w]$ lies in
$\HH^{2\ell(w)}(X,\Z)$.  This completes the additive description of
the cohomology of $X$, as a free abelian group on the basis of
Schubert classes $[X_w]$:
\[
\HH^*(X,\Z) = \bigoplus_{w\in W^P}\Z\, [X_w].
\]

\subsection{The Giambelli problem}
\label{giamprob}
We seek to generalize the classical Schubert
calculus of \S \ref{class} to the $G/P$ spaces in their natural cell
decompositions described above. By this we mean to understand as
explicitly as possible the multiplicative structure of the cohomology
ring $\HH^*(G/P,\Z)$, expressed in the basis of Schubert classes.
However, it is clear that the extension of the classical Pieri and
Giambelli formulas to $G/P$ depends on a choice of {\em special
Schubert classes} which generate the cohomology ring. For arbitrary
$G$, it is fortunate that there is one example of a parabolic $P$
where the choice of generating set is clear: the case when $P=B$ is a
Borel subgroup of $G$, and $X=G/B$ is the (complete) flag variety of
$G$. In this case, the cohomology ring $\HH^*(X,\Q)$ is generated by
the classes of the Schubert divisors $X_{s_\al}$, one for each simple
root $\al$.  Equivalently, we can form a generating set by taking any
$\Z$-basis of $\HH^2(X,\Z)$, which is related to the Schubert divisor
basis by a linear change of variables. In a seminal paper, Borel
\cite{Bo}, using the theory of group characters and associated
characteristic classes of line bundles over $G/B$, gave an invariant,
group theoretic approach to this question, which we recall below.

A {\em character} of the group $B$ is a homomorphism of algebraic
groups $B \to \C^*$. We denote the abelian group of characters of $B$
by $\wh{B}$. Observe that any character $\chi$ of $B$ is uniquely
determined by its restriction to $T$, since $B=T\ltimes U$ is the
semidirect product of $T$ and the unipotent subgroup $U$ of $B$, and
regular invertible functions on $U$ are constant. It follows that the
character group $\wh{T}$ of $T$ is isomorphic to $\wh{B}$.

If $\chi$ is a character of $B$, then we get an induced free action of
$B$ on the product $G\times \C$ by $b\cdot (g,z) =
(gb^{-1},\chi(b)z)$. The quotient space $L_\chi=(G\times \C)/B$, also
denoted by $G\times^B\C$, projects to the flag manifold $G/B$ by
sending the orbit of $(g,z)$ to $gB$. This makes $L_\chi$ into the
total space of a holomorphic line bundle over $G/B$, which is the
homogeneous line bundle associated to the weight $\chi$.  Let
$\Pic(G/B)$ be the Picard group of isomorphism classes of line bundles
on $G/B$, with the group operation given by the tensor product. Then
the map $\chi\mapsto L_\chi$ is a group homomorphism $\wh{B}\to
\Pic(G/B)$.  Composing this with the first Chern class map
$c_1:\Pic(G/B)\to \HH^2(G/B,\Z)$ gives a group homomorphism
\begin{gather*}
\wh{B} \to \HH^2(G/B,\Z), \\
\chi \mapsto c_1(L_\chi).
\end{gather*}
If $S(\wh{B})$ denotes the symmetric algebra of the $\Z$-module
$\wh{B}$, then the above map extends to a homomorphism of graded rings
\[
c: S(\wh{B}) \to \HH^*(G/B,\Z)
\]
called the {\em characteristic homomorphism}. 

Let $\Sh:=S(\wh{B})\otimes_\Z\Q$.  Borel \cite{Bo} proved that the
morphism $c$ is surjective, after tensoring with $\Q$, and that the
kernel of $c$ is the ideal generated by the $W$-invariants of positive
degree in $\Sh$, denoted $\langle\Sh^W_+\rangle$. We thus obtain the
classical Borel presentation of the cohomology ring of $G/B$,
\begin{equation}
\label{Bor1}
\HH^*(G/B,\Q) \cong \Sh/\langle\Sh^W_+\rangle.
\end{equation}
Furthermore, for any parabolic subgroup $P\supset B$, the projection map
$G/B\to G/P$ induces an injection $\HH^*(G/P)\hra \HH^*(G/B)$, and 
this inclusion is realized in the presentation (\ref{Bor1}) by taking 
$W_P$ invariants on the right hand side:
\begin{equation}
\label{Bor2}
\HH^*(G/P,\Q) \cong \Sh^{W_P}/\langle\Sh^W_+\rangle.
\end{equation}
One knows that if the algebraic group $G$ is {\em special} in the
sense of \cite{SeCh}, then the isomorphisms (\ref{Bor1}) and
(\ref{Bor2}) hold with $\Z$-coefficients.  For the classical groups,
this is the case if $G$ is the general linear group $\GL_n$ or the
symplectic group $\Sp_{2n}$.

\begin{example}
Suppose that $G=\GL_n(\C)$, and choose the Borel subgroup $B$ of 
upper triangular matrices, as in \S \ref{class}. The flag manifold
$G/B$ then parametrizes complete flags of subspaces 
\[
0\subsetneq E_1\subsetneq E_2\subsetneq \cdots\subsetneq E_n=\C^n
\]
in $\C^n$. The maximal torus $T\subset B$ is the group of invertible
diagonal matrices, and the characters of $T$ are the maps
\[ 
\text{diag}(t_1,\ldots,t_n) \mapsto t_1^{\al_1}\cdots t_n^{\al_n},
\]
where $\al_1,\ldots,\al_n$ are integers. In this way, the multiplicative
group of characters $\wh{T}$ (and $\wh{B}$) is identified with the
additive group $\Z^n$, and we have 
$$S(\wh{T})=S(\wh{B})=\Z[\x_1,\ldots,\x_n].$$ 
The Weyl group $W$ is the symmetric group $S_n$, and 
(\ref{Bor1}) implies that $\HH^*(G/B,\Z)$ is isomorphic to
$\Z[\x_1,\ldots,\x_n]/\I$, where $\I$ is the ideal generated by the 
non-constant elementary symmetric polynomials in $\x_1,\ldots,\x_n$.

In the case of the Grassmannian $\G(m,n)=\GL_n(\C)/P$, the associated
parabolic subgroup of $W=S_n$ is $W_P=S_m\times S_{n-m}$, embedded in
$S_n$ in the obvious way. The $W_P$-coset representatives of minimal
length are the permutations $\om\in S_n$ such that $\om_1<\cdots < \om_m$
and $\om_{m+1}<\cdots < \om_n$. If $\la=(\la_1,\ldots,\la_m)$ is a
partition which indexes a Schubert variety $X_\la$ in $\G(m,n)$, then
the minimal length coset representative $\om$ corresponding to
$\la$ is determined by the equations 
\[
\om_j = \la_{m+1-j}+j, \ \ \text{for $1\leq j \leq m$}.
\]
The ring presentation (\ref{Bor2}) assumes the form
\[
\HH^*(\G(m,n),\Z) \cong \left(\Z[\x_1,\ldots,\x_m]^{S_m}\otimes
 \Z[\x_{m+1},\ldots,\x_n]^{S_{n-m}}\right)/\I
\]
and the two groups of variables $\x_1,\ldots,\x_m$ and 
$\x_{m+1},\ldots,\x_n$ are the Chern roots of the vector bundles
$E'$ and $E''$ in the universal exact sequence (\ref{ses}), respectively.
The relations generating the ideal $\I$ translate into the Whitney sum formula
\[
c(E')c(E'') = c(E)= 1.
\]
\end{example}

Although the presentation (\ref{Bor2}) is very natural from a
Lie-theoretic point of view, in general there will be more than one
way to identify special Schubert classes which generate the cohomology
ring of $G/P$ among the $W_P$-invariants in $\Sh$. However, when $G$
is a {\em classical group}, there is a good uniform choice of special
Schubert class generators of $\HH^*(G/P)$, as explained in the
introduction.  Initially, the special Schubert classes on any
Grassmannian are defined exactly as in \S \ref{class} and \S
\ref{isogiam}. If $P$ is arbitrary, then the partial flag variety
$G/P$ admits projection maps to various Grassmannians $G/P_r$, defined
by omitting all but one of the subspaces in the flags parametrized by
$G/P$. The special Schubert classes on $G/P$ are defined to be the
pullbacks of the special Schubert classes from the Grassmannians
$G/P_r$. The {\em Giambelli problem} then asks for an explicit
combinatorial formula which writes a general Schubert class in
$\HH^*(G/P)$ as a polynomial in these special classes.

\subsection{Equivariant cohomology and degeneracy loci}
\label{ecdl}

The definition of equivariant cohomology begins with the construction
of a contractible space $EG$ on which $G$ acts freely. If $BG=EG/G$
denotes the quotient space, which is a classifying space for the group
$G$, then the map $EG\to BG$ is a universal principal $G$-bundle. In
the topological category, this means that if ${\mathcal U}\to M$ is
any principal $G$-bundle, then there is a morphism $f:M\to BG$, which
is unique up to homotopy, such that ${\mathcal U}\cong f^*EG$. If $X$
is any topological space endowed with a $G$-action, then $G$ acts
diagonally on $EG\times X$ (that is, $g(e,x):= (ge,gx)$) and the
quotient $$EG\times^GX:=(EG\times X)/G$$ exists. The $G$-equivariant
cohomology ring $\HH^*_G(X)$ of $X$ is then defined by
\[
\HH^*_G(X):=\HH^*(EG\times^GX),
\]
where we usually take cohomology with $\Q$-coefficients. Note that any
subgroup of $G$ also acts freely on $EG$, therefore we can similarly
define classifying spaces $BB:=EG/B$, $BT:=EG/T$, etc., and
equivariant cohomology rings $\HH^*_B(X)$, $\HH^*_T(X)$, respectively.
The natural map $EG\times^T X\to EG\times^B X$ induces an isomorphism
\begin{equation}
\label{btiso}
\HH^*_T(X)\cong \HH^*_B(X).
\end{equation} 
For more on equivariant cohomology, see for example \cite{Hs, Br2}.

A {\em flag bundle} is a fibration with fibers isomorphic to the flag
variety $X=G/B$.  Motivated by the theory of degeneracy loci of vector
bundles, Fulton, Pragacz, and Ratajski \cite{Fu2, Fu3, PR2}
studied the question of obtaining explicit Chern class formulas for
the classes of universal Schubert varieties in flag bundles for the
classical Lie groups. In effect, this is the Giambelli problem of \S
\ref{giamprob} when the homogeneous space $X$ varies in a family. We
give a description of the problem here assuming that we are in Lie
type A, B, or C for simplicity. Suppose that $E\to M$ is a vector
bundle over a variety $M$, which in type B or C comes equipped with a
nondegenerate symmetric or skew-symmetric bilinear form $E\otimes
E\to \C$, respectively. Assume that $E_\bull$ and $F_\bull$ are two
complete flags of subbundles of $E$, taken to be isotropic in types B
and C. For any $w$ in the Weyl group, we have the degeneracy locus
\begin{equation}
\label{degdef}
\X_w:=\{b\in M\ |\ \dim(E_r(b)\cap F_s(b)) \geq d_w(r,s) \ \, \forall \, r,s\},
\end{equation}
where $d_w(r,s)$ is a function taking values in the nonnegative
integers. The inequalities in (\ref{degdef}) are exactly those which
define the Schubert variety $X_w(F_\bull)$ in the flag variety $X$
(they are given explicitly in \S \ref{dloc}). Assuming that $M$
is smooth and that the locus $\X_w$ has pure codimension $\ell(w)$ in
$M$ (hypotheses which can both be relaxed), one seeks a formula for
the cohomology class $[\X_w]\in \HH^*(M)$ in terms of the Chern
classes of the vector bundles which appear in (\ref{degdef}).

Graham \cite{Gra} studied the above problem by placing it in a more
general Lie-theoretic framework, as follows. The morphism $BB\to BG$
is a flag bundle, and he showed that the fiber product space
$BB\times_{BG}BB$ is a classifying space for the question posed by
Fulton et.\ al., because any other example will pull back from
it. Therefore, all the desired formulas for $[\X_w]$ occur in
$\HH^*(BB\times_{BG}BB)$.  Furthermore, it is known (see for example
\cite[\S 1]{Br2}) that $\HH^*(BB)=\Sh$, $\HH^*(BG)=\Sh^W$, and the
Leray-Hirsch theorem implies that there is a natural isomorphism
\begin{equation}
\label{natiso}
\HH^*(BB\times_{BG}BB) \cong \Sh\otimes_{\Sh^W}\Sh.
\end{equation}
This point of view on degeneracy loci and the polynomials which 
represent their cohomology classes will be used in \S \ref{ddgeom}, and
we refer there for more details.

Another key observation in \cite{Gra} is that the degeneracy locus
problem is equivalent to the question of obtaining an {\em equivariant
  Giambelli formula} in $\HH^*_T(G/B)$. Indeed, there is a $(B\times
B)$-equivariant isomorphism $$EG\times G \to EG\times_{BG}EG$$ sending
$(e,g)$ to $(e,ge)$, and passing to the quotient spaces gives a
natural isomorphism
\begin{equation}
\label{BMiso}
EG\times^B (G/B) \cong BB\times_{BG}BB
\end{equation}
which maps $EG\times^B X_w$ to $\X_w$. In view of the isomorphism
(\ref{BMiso}), we will call $BB\times_{BG}BB$ the {\em Borel mixing
  space} associated to $G/B$.  Taking the cohomology of both sides of
(\ref{BMiso}) and using (\ref{btiso}), we obtain an isomorphism
\[
\HH^*_T(G/B) = \HH^*(EG\times^B (G/B)) \cong \HH^*(BB\times_{BG}BB)
\]
which sends the {\em equivariant Schubert class}
$[X_w]^T:=[EG\times^BX_w]$ to $[\X_w]$.  
Finally, we remark that an equivariant Giambelli formula for $[X_w]^T$
in $\HH^*_T(G/B)$ specializes to a Giambelli formula for $[X_w]$ in
$\HH^*(G/B)$ when we set all the variables coming from the linear 
$T$-action equal to zero.

\begin{remark}
The interpretation of the Giambelli problem for the cohomology of
$G/P$ given here is our own; compare with the discussion in
\cite[Chp.\ III \S 3]{Hi}. We note that there are other kinds of
`Giambelli formulas' which appear in the literature. The subject of
degeneracy loci also has a long history (see for example \cite{FP})
and includes varieties such as those discussed in the introduction,
defined by imposing rank conditions on generic {\em morphisms} between
vector bundles. The degeneracy loci considered here, which are defined
by intersecting subbundles of a fixed vector bundle as in
(\ref{degdef}), were introduced and studied in \cite{Fu3, PR2}.
\end{remark}

\section{Weyl groups, Grassmannian elements, and transition trees}
\label{wgtts}

In this section we explain the combinatorial objects which enter into
the algebraic and geometric formulas of later sections. We begin with
a discussion of the Weyl groups for the classical root systems of type
A, B, C, and D.

\subsection{Weyl groups}
Let $W_n$ denote the {\em hyperoctahedral group} of signed
permutations on the set $\{1,\ldots,n\}$, which is the semidirect
product $S_n \ltimes \Z_2^n$ of the symmetric group $S_n$ with
$\Z_2^n$. We adopt the notation where a bar is written over an entry
with a negative sign; thus $w=(3,\ov{1},2)$ maps $(1,2,3)$ to
$(3,-1,2)$. The group $W_n$ is the Weyl group for the root system
$\text{B}_n$ or $\text{C}_n$, and is generated by the simple
transpositions $s_i=(i,i+1)$ for $1\leq i \leq n-1$ and the sign
change $s_0(1)=\ov{1}$. The {\em symmetric group} $S_n$ is the
subgroup of $W_n$ generated by the $s_i$ for $1\leq i \leq n-1$, and
is the Weyl group for the root system $\text{A}_{n-1}$. The 
Weyl group $\wt{W}_n$ for the root system $\text{D}_n$ is the
subgroup of $W_n$ consisting of all signed permutations with an even
number of sign changes.  The group $\wt{W}_n$ is an extension of $S_n$
by the element $s_\Box=s_0s_1s_0$, which acts on the right by
\[
(w_1,w_2,\ldots,w_n)s_\Box=(\ov{w}_2,\ov{w}_1,w_3,\ldots,w_n).
\]
There are natural embeddings $W_n\hookrightarrow W_{n+1}$ and 
$\wt{W}_n\hookrightarrow \wt{W}_{n+1}$ defined by
adjoining the fixed point $n+1$. We let $S_\infty := \cup_nS_n$,  
$W_\infty :=\cup_n W_n$, and $\wt{W}_\infty := \cup_n \wt{W}_n$.

Consider the two sets $\N_0 :=\{0,1,\ldots\}$ and $\N_\Box
:=\{\Box,1,\ldots\}$ whose members index the simple reflections. A
{\em reduced word} of an element $w$ in $W_\infty$ (respectively
$\wt{W}_\infty$) is a sequence $a_1\cdots a_\ell$ of elements of
$\N_0$ (respectively $\N_\Box$) such that $w=s_{a_1}\cdots s_{a_\ell}$
and $\ell$ is minimal, so (by definition) equal to the length
$\ell(w)$ of $w$. We say that $w$ has {\em descent} at position $r$ if
$\ell(ws_r)<\ell(w)$, where $s_r$ is the simple reflection indexed by
$r$. For $r\in \N_0$, this is equivalent to the condition
$w_r>w_{r+1}$, where we set $w_0=0$ (so $w$ has a descent at position
$0$ if and only if $w_1<0$).

\subsection{Grassmannian elements}

A permutation $\om\in S_\infty$ is {\em Grassmannian} if there exists
an $m\geq 1$ such that $\om_i<\om_{i+1}$ for all $i\neq m$.  The {\em
  shape} of such a Grassmannian permutation $\om$ is the partition
$\la=(\la_1,\ldots, \la_m)$ with $\la_{m+1-j}=\om_j-j$ for $1\leq j
\leq m$. Notice that there are infinitely many permutations of a given
shape $\la$. However, for each fixed $m$ and $n>m$, we obtain a
bijection between the set of permutations in $S_n$ with at most one
descent at position $m$ and the set of partitions $\la$ whose diagram
fits inside an $m \times (n-m)$ rectangle.

Fix a nonnegative integer $k$. An element $w=(w_1,w_2,\ldots)$ in
$W_\infty$ is called {\em $k$-Grassmannian} if and only if we have
$\ell(ws_i)=\ell(w)+1$ for all $i\neq k$.  When $k=0$, this says that
$w$ is increasing: $w_1<w_2<\cdots$, while when $k>0$, then $w$ is 
$k$-Grassmannian if and only if
\[
0< w_1 < \cdots < w_k \quad \text{and} \quad w_{k+1}<w_{k+2}<\cdots .
\]

There is an explicit bijection between $k$-Grassmannian elements $w$ of
$W_\infty$ and $k$-strict partitions $\la$, under which the elements in
$W_n$ correspond to those partitions whose diagram fits inside
an $(n-k)\times (n+k)$ rectangle. This bijection is obtained as
follows.  The absolute value of the negative entries in $w$ form a
(possibly empty) strict partition $\mu$ whose diagram gives the part
of the diagram of $\la$ which lies in columns $k+1$ and higher. The
boxes in these columns which lie outside of the diagram of $\mu$ are
partitioned into south-west to north-east {\em diagonals}. For each
$i$ between $1$ and $k$, let $d_i$ be the diagonal among these which
contains $w_i$ boxes. Then the bottom box $B_i$ of $\la$ in column
$k+1-i$ is {\em $k$-related} to $d_i$, for $1\leq i \leq k$.  Here
`$k$-related' means that the directed line segment joining $B_i$ with
the lowest box of the diagonal $d_i$ is north-west to south-east (this
notion was used in \cite{BKT1, BKT2}). We will denote the Weyl group
element associated to $\la$ by $w_\la$.

\begin{example}
(a) The $3$-Grassmannian element $w = (3,5,8,\ov{4},\ov{1},2,6,7)$
  corresponds to the $3$-strict partition $\la= (7,4,3,1,1)$.
\[
\lambda \ = \ \ \raisebox{-53pt}{\pic{.6}{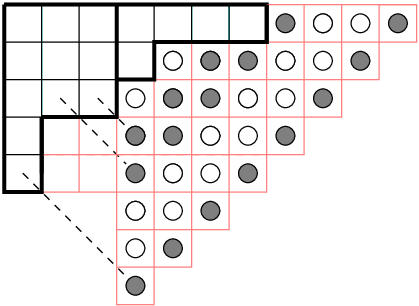}}
\]
\noin 
(b) A permutation $\om\in S_\infty$ with a unique descent at
position $m$ may be viewed as an $m$-Grassmannian element of
$W_\infty$, via the natural inclusion of $S_\infty$ in $W_\infty$. In
this case, the $m$-strict partition associated to $\om$ is the
conjugate of the shape of $\om$ as a Grassmannian permutation.
\end{example}

Fix a $k\in \N_\Box$ with $k\neq 1$. An element $w\in\wt{W}_\infty$ is
called $k$-Grassmannian if $\ell(ws_i)=\ell(w)+1$ for all $i\neq
k$. We say that $w$ is $1$-Grassmannian if $\ell(ws_i)=\ell(w)+1$ for
all $i\geq 2$.  The $\Box$-Grassmannian elements are therefore the
ones whose values strictly increase, while if $k>0$, then $w$ is
$k$-Grassmannian if and only if
\[
|w_1| < w_2 <\cdots < w_k \quad \text{and} \quad w_{k+1}<w_{k+2}<\cdots .
\]

There is a bijection between the $k$-Grassmannian elements $w$ of
$\wt{W}_\infty$ and typed $k$-strict partitions $\la$, under which the
elements in $\wt{W}_n$ correspond to typed partitions whose
diagram fits inside an $(n-k)\times (n+k-1)$ rectangle, obtained as
follows. Subtract one from the absolute value of each negative entry
$w_{k+j}$, for $j\geq 1$, to obtain the parts of a strict partition
$\mu$. As above, the partition $\mu$ gives the part of $\la$ which
lies in columns $k+1$ and higher, and the boxes in these columns
outside of $\mu$ are partitioned into south-west to north-east
diagonals. Choose $d_1,\ldots, d_k$ among these diagonals such that
the number of boxes in $d_i$ is given by $|w_i|-1$, for $1\leq i \leq
k$. The bottom boxes of $\la$ in the first $k$ columns are then {\em
  $k'$-related} to these $k$ diagonals, with bottom box $B_i$ in
column $k+1-i$ related to diagonal $d_i$.  Here `$k'$-related' differs
from `$k$-related' by a shift up by one unit, as shown in the next
figure. Finally, we have $\type(\la)>0$ if and only if $|w_1|>1$, and
in this case $\type(\la)$ is equal to $1$ or $2$ depending on whether
$w_1>0$ or $w_1<0$, respectively. We will denote the Weyl group
element associated to $\la$ by $w_\la$, as above.

\begin{example}
The $3$-Grassmannian element $w =
(\ov{2},6,7,\ov{5},\ov{3},\ov{1},4,8)$ corresponds to the $3$-strict
partition $\la= (7,5,3,2)$ of type $2$.
\[
\lambda \ = \ \ \raisebox{-53pt}{\pic{.6}{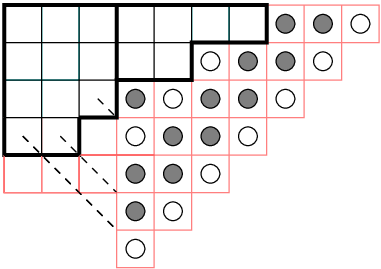}}
\]
\end{example}

To each $k$-Grassmannian element $w$ in $W_\infty$ or $\wt{W}_\infty$,
we attach a finite order ideal $\cC(w)$ in $\Delta^\circ$ by the
prescription
\begin{equation}
\label{Cweq}
\cC(w) := \{ (i,j)\in \Delta^\circ \ |\ w_{k+i}+ w_{k+j} < 0 \}.
\end{equation}
If $\la$ is the $k$-strict partition or typed $k$-strict partition
corresponding to $w$, then $w=w_\la$ and we say that $\la$ is the {\em
  shape} of $w$. In either case, it is easy to check that $\cC(w)$
is equal to the ideal $\cC(\la)$ or $\cC'(\la)$ defined in \S
\ref{isogiam}, respectively.

\subsection{Transition trees and Stanley coefficients}
\label{tts}

The notion of a transition was introduced by Lascoux and
Sch\"utzenberger in \cite{LS1, LS3}, where it was applied to achieve
efficient, recursive computations of type A Schubert polynomials and a
new form of the Littlewood-Richardson rule. Analogues of their
transition equations for the other classical Lie types were obtained
by Billey \cite{Bi1}, using the Schubert polynomials found in
\cite{BH}. We will require the extension of this theory given in
\cite{T6}, which includes the $k$-Grassmannian elements for all $k>0$.
For positive integers $i<j$ we define reflections $t_{ij}\in S_\infty$
and $\ov{t}_{ij},\ov{t}_{ii} \in W_\infty$ by their right actions
\begin{align*}
(\ldots,w_i,\ldots,w_j,\ldots)\,t_{ij} &= 
(\ldots,w_j,\ldots,w_i,\ldots), \\
(\ldots,w_i,\ldots,w_j,\ldots)\,\ov{t}_{ij} &= 
(\ldots,\ov{w}_j,\ldots,\ov{w}_i,\ldots), \ \ \mathrm{and} \\
(\ldots,w_i,\ldots)\,\ov{t}_{ii} &= 
(\ldots,\ov{w}_i,\ldots),
\end{align*} 
and let $\ov{t}_{ji} := \ov{t}_{ij}$.

Following Lascoux and Sch\"utzenberger \cite{LS3}, for any permutation
$\om\in S_\infty$, we construct a rooted tree $T(\om)$ with nodes
given by permutations of the same length $\ell(\om)$, and root $\om$,
as follows. If $\om=1$ or $\om$ is Grassmannian, then set
$T(\om):=\{\om\}$. Otherwise, let $r$ be the largest descent of $\om$,
and set $s := \max(i>r\ |\ \om_i < \om_r)$. The definitions of $r$ and
$s$ imply that we have $\ell(\om t_{rs}) =\ell(\om)-1$. Let
$$I(\om):=\{i \ |\ 1\leq i < r \ \ \mathrm{and} \ \ \ell(\om
t_{rs}t_{ir}) = \ell(\om) \}.$$ If $I(\om)\neq \emptyset$, then let
$\Psi(\om):=\{ \om t_{rs}t_{ir}\ |\ i\in I(\om)\}$; otherwise, let
$\Psi(\om):=\Psi(1\times \om)$.  To recursively define $T(\om)$, we
join $\om$ by an edge to each $v\in \Psi(\om)$, and attach to each
$v\in \Psi(\om)$ its tree $T(v)$. One can show that $T(\om)$ is a
finite tree whose leaves are all Grassmannian permutations. The tree
$T(\om)$ is the {\em Lascoux-Sch\"utzenberger transition tree} of
$\om$.\footnote{The trees in \cite{LS3} differ slightly from
the ones here; their leaves are vexillary permutations.} We
define the {\em Stanley coefficient} $c^\om_\la$ to be the number of
leaves of shape $\la$ in the transition tree $T(\om)$ associated to
$\om$.

\begin{example}
\label{treeA}
The transition tree of $\om=143265$ is displayed below. At all 
nodes where the branching rule is applied, the positions $r$ and $s$ 
are shown in boldface.
\medskip
\[
\includegraphics[scale=0.31]{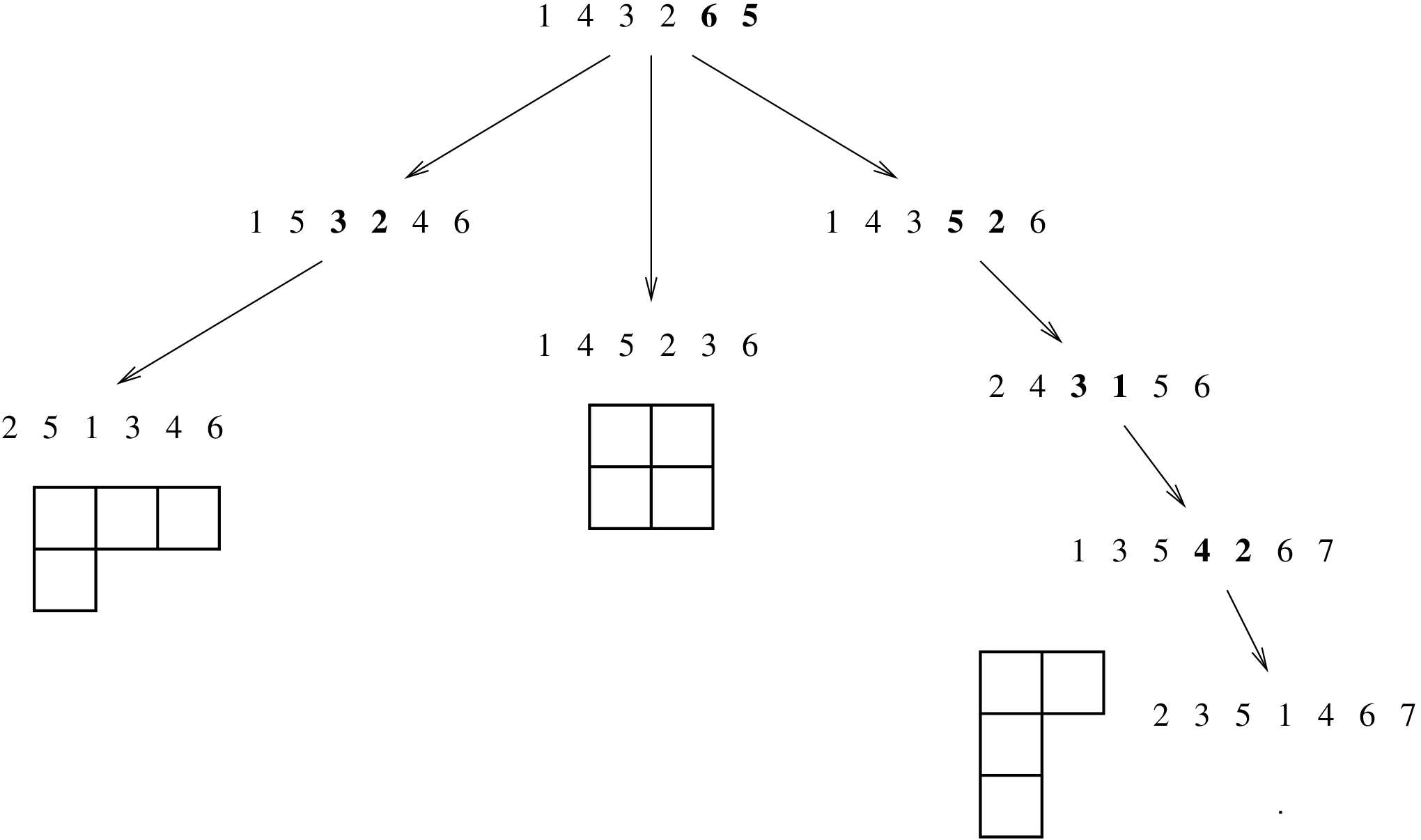}
\]
\end{example}

There exist alternative combinatorial formulas for the coefficients
$c_\la^\om$ in terms of Young tableaux \cite{EG, FG, RS}; see also
\cite{Li}. We note here the result of Fomin and Greene \cite{FG} that
$c_\la^\om$ is equal to the number of semistandard tableaux $T$ of
shape $\wt{\la}$ such that the {\em column word} of $T$, obtained by
reading the entries of $T$ from bottom to top and left to right, is a
reduced word for $w$. We display below the three tableaux associated
to the leaves of the tree for $143265=s_3s_2s_3s_5$ in Example
\ref{treeA}.
\smallskip
\[
\includegraphics[scale=0.45]{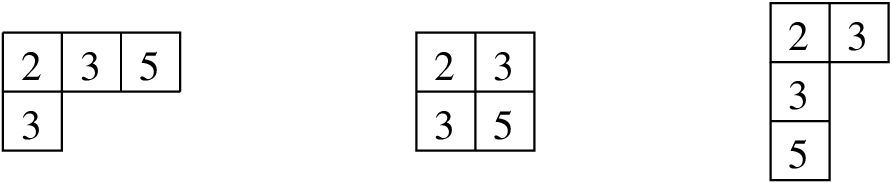}
\]

A signed permutation $w\in W_\infty$ is said to be {\em increasing up
to $k$} if it has no descents less than $k$.  This condition is
automatically satisfied if $k=0$, and for positive $k$ it means that
$0 < w_1 < w_2 < \cdots < w_k$. If $k\geq 2$, we say that an element
$w\in \wt{W}_\infty$ is {\em increasing up to $k$} if it has no
descents less than $k$; this means that $|w_1|< w_2<\cdots < w_k$. By
convention we agree that every element of $\wt{W}_\infty$ is
increasing up to $\Box$ and also increasing up to $1$.

For any $w\in W_\infty$ which is
increasing up to $k$, we construct a rooted tree $T^k(w)$ with nodes
given by elements of $W_\infty$ and root $w$ as follows. Let $r$ be
the largest descent of $w$. If $w=1$ or $r=k$, then set
$T^k(w):=\{w\}$. Otherwise, let $s := \max(i>r\ |\ w_i < w_r)$ and
$\Phi(w) := \Phi_1(w)\cup \Phi_2(w)$, where
\begin{gather*}
\Phi_1(w) := \{wt_{rs}t_{ir}\ |\ 1\leq i < r \ \ \mathrm{and} \ \ 
\ell(wt_{rs}t_{ir}) = \ell(w) \}, \\
\Phi_2(w) := 
\{wt_{rs}\ov{t}_{ir}\ |\ i\geq 1 \ \ \mathrm{and} \ \
\ell(wt_{rs}\ov{t}_{ir}) = \ell(w) \}.
\end{gather*}
To recursively define $T^k(w)$, we join $w$ by an edge to each $v\in
\Phi(w)$, and attach to each $v\in \Phi(w)$ its tree $T^k(v)$.  We
call $T^k(w)$ the {\em $k$-transition tree} of $w$. 

\begin{example}
\label{treeC}
The $1$-transition tree of $w=3\ov{1}2645$ is displayed below, with 
positions $r$ and $s$ shown in boldface.
\medskip
\[
\includegraphics[scale=0.31]{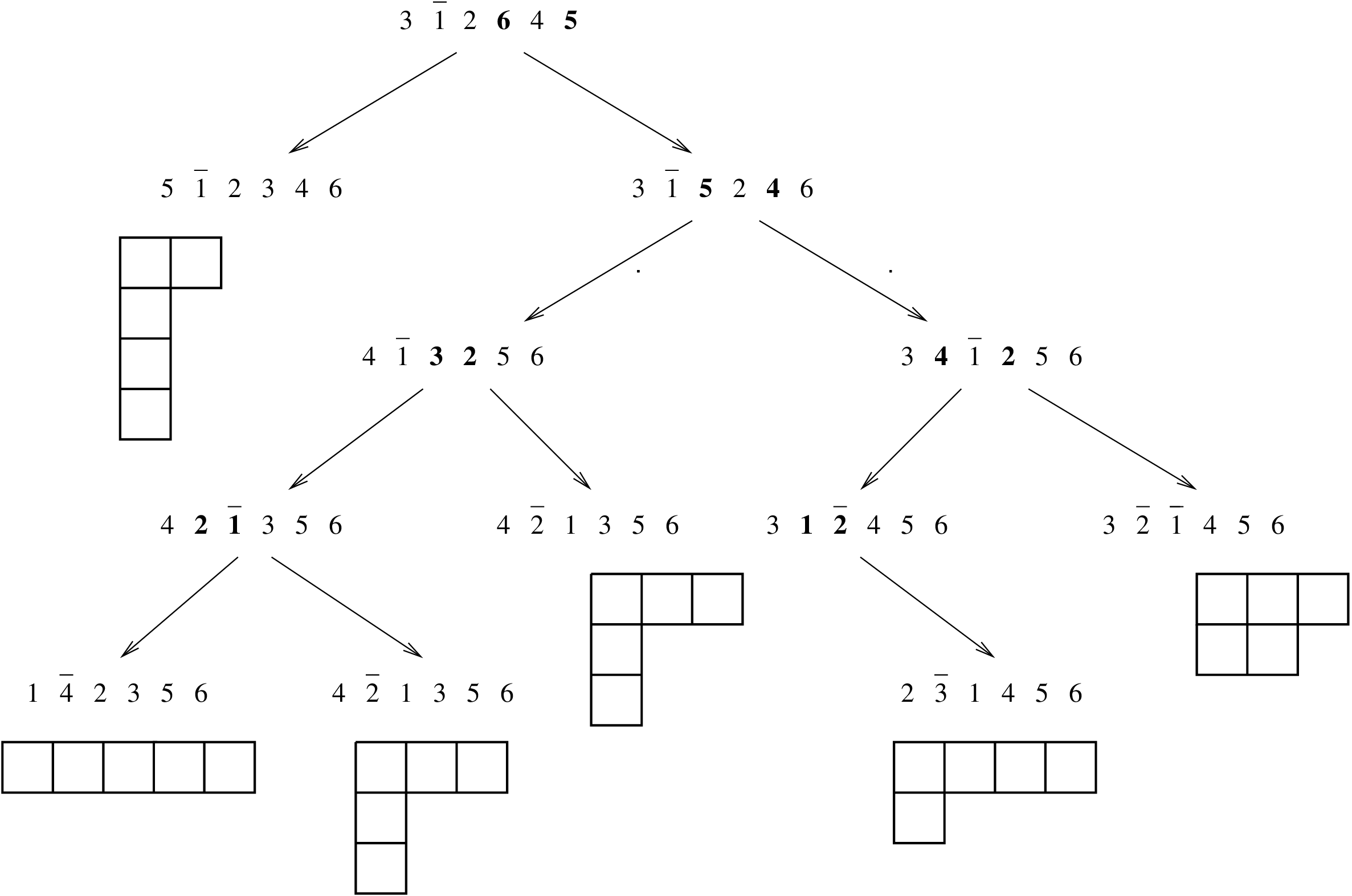}
\]
\end{example}

For any $w\in \wt{W}_\infty$ which is increasing up to $k$, we
construct the $k$-transition tree $\wt{T}^k(w)$ with nodes given by
elements of $\wt{W}_\infty$ and root $w$ in a manner parallel to the
above. Let $r$ be the largest descent of $w$. If $w=1$, or $k\neq 1$
and $r=k$, or $k=1$ and $r\in \{\Box,1\}$, then set $\wt{T}^k(w)
:=\{w\}$. Otherwise, let $s := \max(i>r\ |\ w_i < w_r)$ and
$\wt{\Phi}(w) := \wt{\Phi}_1(w)\cup \wt{\Phi}_2(w)$, where
\begin{gather*}
\wt{\Phi}_1(w) := \{wt_{rs}t_{ir}\ |\ 1\leq i < r \ \ \mathrm{and} \ \ 
\ell(wt_{rs}t_{ir}) = \ell(w) \}, \\
\wt{\Phi}_2(w) := 
\{wt_{rs}\ov{t}_{ir}\ |\ i\neq r \ \ \mathrm{and} \ \
\ell(wt_{rs}\ov{t}_{ir}) = \ell(w) \}.
\end{gather*}
To define $\wt{T}^k(w)$, we join $w$ by an edge to each $v\in
\wt{\Phi}(w)$, and attach to each $v\in \wt{\Phi}(w)$ its tree
$\wt{T}^k(v)$. 

According to \cite[Lemma 3 and \S 6]{T6}, the $k$-transition trees
$T^k(w)$ and $\wt{T}^k(w)$ are finite trees, all of whose nodes are
increasing up to $k$. Moreover, the leaves of these trees are
$k$-Grassmannian elements.  For any $k$-strict (respectively, typed
$k$-strict) partition $\la$, the {\em mixed Stanley coefficient}
$e^w_{\la}$ (respectively, $d^w_{\la}$) is defined to be the number of
leaves of $T^k(w)$ (respectively, $\wt{T}^k(w)$) of shape $\la$. For
instance, from Example \ref{treeC} we deduce that
\[
e^{3\ov{1}2645}_{(4,1)} = e^{3\ov{1}2645}_{(3,2)} = 1 \quad \mathrm{and} \quad
e^{3\ov{1}2645}_{(3,1,1)}=2.
\]

\begin{example}
When $k=0$, there exist alternative combinatorial formulas for the
integers $e^w_\la$ and $d^w_{\la}$ involving certain tableaux \cite{H,
  Kr, L1, L2}. For instance, choose any $w$ in $W_\infty$ and let
$\la$ be a Young diagram with $r$ rows such that $|\la| = \ell(w)$. We
call a sequence $b=(b_1,\ldots,b_m)$ {\em unimodal} if for some index
$j$ we have
\[
b_1 > b_2 > \cdots > b_{j-1} > b_j < b_{j+1} < \cdots < b_m .
\]
A {\em subsequence} of $b$ is any sequence $(b_{i_1}, \dots, b_{i_p})$
with $1 \leq i_1 < \dots < i_p \leq m$. A {\em Kra\'skiewicz tableau}
(also known as a {\em standard decomposition tableau}) for $w$ of
shape $\la$ is a filling $T$ of the boxes of $\lambda$ with
nonnegative integers in such a way that (i) if $t_i$ is the sequence
of entries in the $i$-th row of $T$, reading from left to right, then
the row word $t_r\ldots t_1$ is a reduced word for $w$, and (ii) for
each $i$, $t_i$ is a unimodal subsequence of maximum length in $t_r
\ldots t_{i+1} t_i$. Lam \cite{L2} proved that $e^w_{\la}$ is equal to
the number of Kra\'skiewicz tableaux for $w$ of shape $\la$.
\end{example}

\section{Schubert polynomials and symmetric functions}
\label{ssss}

\subsection{The nilCoxeter algebra and Schubert polynomials}
\label{nilSch}

Bernstein-Gelfand-Gelfand \cite{BGG} and Demazure \cite{D1, D2} used
{\em divided difference operators} to construct an algorithm that
produces Giambelli polynomials which represent the Schubert classes on
the flag variety $G/B$ in the Borel presentation (\ref{Bor1}) of the
cohomology ring, starting from the choice of a representative for the
class of a point. For the general linear group, Lascoux and
Sch\"utzenberger \cite{LS1, LS3} applied this algorithm with a very
natural choice of top degree polynomial to define {\em Schubert
  polynomials}, a theory which included both single and double
versions.

In a series of papers \cite{FS, FK1, FK2, L1}, Fomin, Stanley,
Kirillov, and Lam developed an alternative approach to the theory of
Schubert polynomials, which extended to all the classical types. The
main idea was to use the nilCoxeter algebra of the Weyl group, an
abstract algebra which is isomorphic to the algebra of divided
difference operators. The work was complicated by the fact that in Lie
types B, C, and D, there are many candidate theories of Schubert
polynomials (see \cite{BH, Fu2, LP1, LP2, T2, T3} for examples), and
it was not clear why any one of them should be preferred from the
others. However, the theory that was best understood from a
combinatorial point of view was that of Billey and Haiman \cite{BH},
which incorporated both the Lascoux-Sch\"utzenberger type A Schubert
polynomials and the classical Schur $Q$-functions. Double versions of the
Billey-Haiman polynomials were recently introduced and studied by
Ikeda, Mihalcea, and Naruse \cite{IMN}. Their type B, C, and D double
Schubert polynomials will be key ingredients in our story, and we
discuss their construction using the nilCoxeter algebra below.

The {\em nilCoxeter algebra} $\cW_n$ of the hyperoctahedral group $W_n$ 
is the free associative algebra with unity
generated by the elements $u_0,u_1,\ldots,u_{n-1}$ modulo the
relations
\[
\begin{array}{rclr}
u_i^2 & = & 0, & i\geq 0\ ; \\
u_iu_j & = & u_ju_i, & |i-j|\geq 2\ ; \\
u_iu_{i+1}u_i & = & u_{i+1}u_iu_{i+1}, & i>0\ ; \\
u_0u_1u_0u_1 & = & u_1u_0u_1u_0.
\end{array}
\]
For any $w\in W_n$, choose a reduced word $a_1\cdots a_\ell$ for $w$
and define $u_w := u_{a_1}\ldots u_{a_\ell}$. Since the last three
relations listed are the Coxeter relations for the Weyl group $W_n$,
it follows that $u_w$ is well defined. Moreover, the $u_w$ for $w\in
W_n$ form a free $\Z$-basis of $\cW_n$. We denote the coefficient of
$u_w\in \cW_n$ in the expansion of the element $h\in \cW_n$ by
$\langle h,w\rangle$; thus $h = \sum_{w\in W_n}\langle
h,w\rangle\,u_w$ for all $h\in \cW_n$.

Let $t$ be an indeterminate and define
\begin{gather*}
A_i(t) := (1+t u_{n-1})(1+t u_{n-2})\cdots 
(1+t u_i) \ ; \\
\tilde{A}_i(t) := (1-t u_i)(1-t u_{i+1})\cdots (1-t u_{n-1}) \ ; \\
C(t) := (1+t u_{n-1})\cdots(1+t u_1)(1+tu_0)
(1+t u_0)(1+t u_1)\cdots (1+t u_{n-1}).
\end{gather*}
Suppose that $X=(x_1,x_2,\ldots)$, 
$Y=(y_1,y_2,\ldots)$, and $Z=(z_1,z_2,\ldots)$ are three
infinite sequences of commuting independent variables. 
Let $$C(X):=C(x_1)C(x_2)\cdots$$ and for $w\in W_n$, define
\begin{equation}
\label{dbleC}
\CS_w(X\,;Y,Z) := \left\langle 
\tilde{A}_{n-1}(z_{n-1})\cdots \tilde{A}_1(z_1)C(X) A_1(y_1)\cdots 
A_{n-1}(y_{n-1}), w\right\rangle.
\end{equation}
Set $\CS_w(X\,;Y):=\CS_w(X\,;Y,0)$. The polynomials $\CS_w(X\,;Y)$ are
the type C Schubert polynomials of Billey and Haiman and the
$\CS_w(X\,;Y,Z)$ are their double versions due to Ikeda, Mihalcea, and
Naruse.  Note that $\CS_w$ is really a polynomial in the $Y$ and $Z$
variables, with coefficients which are formal power series in $X$. In
fact, these formal power series are symmetric in the $X$ variables;
this follows immediately from Fomin and Kirillov's result
\cite[Prop.\ 4.2]{FK2} that $C(s)C(t) = C(t)C(s)$, for any two
commuting variables $s$ and $t$. We set 
$$F_w(X):=\CS_w(X\,;0,0) = \left\langle C(X), w\right\rangle$$ and
call $F_w$ the {\em type C Stanley symmetric function} indexed by
$w\in W_n$.  For any $\om \in S_n$, the double Schubert polynomial
$\AS_\om(Y,Z)$ of Lascoux and Sch\"utzenberger is given by
\[
\AS_\om(Y,Z) := \CS_\om(0\,; Y,Z)=
\left\langle 
\tilde{A}_{n-1}(z_{n-1})\cdots \tilde{A}_1(z_1)
A_1(y_1)\cdots A_{n-1}(y_{n-1}), \om\right\rangle.
\]
The polynomial $\AS_\om(Y):=\AS_\om(Y,0)$ is the single Schubert
polynomial. 

It is transparent from the definition that $\CS_w$ is stable under the
natural inclusion of $W_n$ in $W_{n+1}$; it follows that $\AS_\om$ and
$\CS_w$ are well defined for $\om\in S_\infty$ and $w\in W_\infty$,
respectively. This {\em stability property} of Schubert polynomials is
special to the classical groups, and will be used in \S \ref{ddgeom}
in order to characterize them via their compatibility with divided
difference operators (Theorem \ref{uniq}).

Given any Weyl group elements $u_1,\ldots,u_p,w$, we will write
$u_1\cdots u_p=w$ if $\ell(u_1)+\cdots + \ell(u_p)=\ell(w)$ and the
product of $u_1,\ldots,u_p$ is equal to $w$. In this case we say that
$u_1\cdots u_p$ is a {\em reduced factorization} of $w$. Equation
(\ref{dbleC}) is equivalent to the relations
\begin{equation}
\label{dbleC2}
\CS_w(X\,;Y,Z) = \sum_{uv=w}\AS_{u^{-1}}(-Z)\CS_v(X\,;Y) =
\sum_{uvu'=w}\AS_{u^{-1}}(-Z)F_v(X)\AS_{u'}(Y)
\end{equation}
summed over all reduced factorizations $uv=w$ and $uvu'=w$
(respectively) with $u,u'\in S_\infty$ (compare with \cite[Cor.\
8.10]{IMN}). In particular, we have
\begin{equation}
\label{dbleA2}
\AS_\om(Y,Z) = \sum_{uu'=\om}\AS_{u^{-1}}(-Z)\AS_{u'}(Y)
\end{equation}
summed over all reduced factorizations $uu'=\om$ in $S_\infty$.

For the orthogonal types B and D we will work with coefficients in the
ring $\Z[\frac{1}{2}]$. For $w\in W_\infty$, let $s(w)$ denote the
number of $i$ such that $w_i<0$.  The type B double Schubert
polynomials $\BS_w$ are related to the type C polynomials by the
equation $\BS_w=2^{-s(w)}\CS_w$. We will
use the nilCoxeter algebra $\wt{\cW}_n$ of $\wt{W}_n$ to define type D
Schubert polynomials.  $\wt{\cW}_n$ is the free associative algebra
with unity generated by the elements $u_\Box,u_1,\ldots,u_{n-1}$
modulo the relations
\[
\begin{array}{rclr}
u_i^2 & = & 0, & i\in \N_\Box\ ; \\
u_\Box u_1 & = & u_1 u_\Box, \\
u_\Box u_2 u_\Box & = & u_2 u_\Box u_2, \\
u_iu_{i+1}u_i & = & u_{i+1}u_iu_{i+1}, & i>0\ ; \\
u_iu_j & = & u_ju_i, & j> i+1, \ \text{and} \ (i,j) \neq (\Box,2).
\end{array}
\]
For any $w\in \wt{W}_n$, choose a reduced word $a_1\cdots a_\ell$ for
$w$ and define $u_w := u_{a_1}\ldots u_{a_\ell}$.  We denote the
coefficient of $u_w\in \wt{\cW}_n$ in the expansion of the element
$h\in \wt{\cW}_n$ by $\langle h,w\rangle$.  Let $t$ be an
indeterminate and, following Lam \cite{L1}, define
\[
D(t) := (1+t u_{n-1})\cdots (1+t u_2)(1+t u_1)(1+t u_\Box)
(1+t u_2)\cdots (1+t u_{n-1}).
\]
Let $D(X)=D(x_1)D(x_2)\cdots$, and define
\[
\DS_w(X\,;Y,Z) := \left\langle 
\tilde{A}_{n-1}(z_{n-1})\cdots \tilde{A}_1(z_1) D(X) A_1(y_1)\cdots 
A_{n-1}(y_{n-1}), w\right\rangle.
\]
The polynomials $\DS_w(X\,;Y):=\DS_w(X\,;Y,0)$ are the type D
Billey-Haiman Schubert polynomials, and the $\DS_w(X\,;Y,Z)$ are their
double versions studied in \cite{IMN}.

\begin{example}
For $\om=321\in S_3$ we have that
\begin{align*}
\AS_{321}(Y,Z) &= \left\langle (1-z_2u_2)(1-z_1u_1)(1-z_1u_2)
(1+y_1u_2)(1+y_1u_1)(1+y_2u_2), 321 \right\rangle \\ &=
-z_2z_1^2+z_2z_1y_1+z_2z_1y_2-z_2y_1y_2
+z_1^2y_1-z_1y_1^2-z_1y_1y_2+y_1^2y_2 \\ 
&= (y_1-z_1)(y_1-z_2)(y_2-z_1).
\end{align*}
Let $\om_0=(n,n-1,\ldots,1)$ be the longest element of $S_n$.
Then one can show (see \cite[Cor.\ 4.4]{FS}) that 
\[
\AS_{\om_0}(Y,Z) = \prod_{i+j \leq n} (y_i-z_j).
\] 
There are more complicated formulas (see \cite[Thm.\ 1.2]{IMN}) for
the double Schubert polynomials $\CS_{w_0}$ and $\DS_{w_0}$, where
$w_0$ denotes the longest element in the respective Weyl group. We 
do not require these formulas in this article.
\end{example}

\subsection{Schur, theta, and eta polynomials}
\label{ste}

In this section, we use raising operators to construct the polynomials
from \cite{BKT2, BKT4} which will appear later in splitting formulas
for the double Schubert polynomials.

For $m,n\geq 1$, set $Y_{(m)}=(y_1,\ldots,y_m)$ and $Z_{(n)}=(z_1,\ldots,z_n)$.
Define the complete supersymmetric polynomials $h_r=h_r(Y_{(m)}/Z_{(n)})$ by 
their generating function
\[
 \prod_{i=1}^m (1-y_it)^{-1}\prod_{j=1}^n (1-z_jt) =
\sum_{r=0}^\infty h_r t^r.
\]
The {\em supersymmetric Schur polynomial} $s_\la(Y_{(m)}/Z_{(n)})$ is
obtained by setting 
\[
s_\la(Y_{(m)}/Z_{(n)}) := R^0\, h_\la = 
\det\left(h_{\la_i + j-i}(Y_{(m)}/Z_{(n)})\right)_{i,j}
\]
for any partition $\la$, where $h_\la:=\prod_ih_{\la_i}$. The usual
Schur polynomials satisfy the identities 
\[
s_\la(Y_{(m)}) = s_\la(Y_{(m)}/Z_{(n)})\vert_{Z_{(n)}=0}
\] 
and 
\[
s_\la(0/Z_{(n)}) = s_\la(Y_{(m)}/Z_{(n)})\vert_{Y_{(m)}=0} =
(-1)^{|\la|}s_{\wt{\la}}(Z_{(n)}).
\]

Fix an integer $k\geq 0$, and define formal power series
$\ti_r=\ti_r(X\,;Y_{(k)})$ for $r\in \Z$ by the equation
\[
\prod_{i=1}^{\infty}\frac{1+x_it}{1-x_it} \prod_{j=1}^k
(1+y_jt)= \sum_{r=0}^{\infty} \ti_r t^r.
\]
Set $\ti_\la:= \prod_i\ti_{\la_i}$ and  
\begin{equation}
\label{Tidefn}
\Ti_\la(X\,;Y_{(k)}) := R^\la\, \ti_\la
\end{equation}
for any $k$-strict partition $\la$, where $R^\la$ is the raising
operator in (\ref{ReqC}). Following \cite{BKT2}, we call
$\Ti_\la(X\,;Y_{(k)})$ a {\em theta polynomial}, although this is a
slight abuse of language (compare with Remark \ref{rmk1}).  Define
\begin{equation}
\label{etaequ}
\eta_k := \frac{1}{2}\ti_k + \frac{1}{2}e_k(Y_{(k)}) \ \ 
\mathrm{and} \ \  \eta'_k :=
\frac{1}{2}\ti_k - \frac{1}{2}e_k(Y_{(k)}) =
\frac{1}{2}\sum_{i=0}^{k-1}q_{k-i}(X)e_i(Y_{(k)}).
\end{equation}
For any typed $k$-strict partition $\la$, consider the {\em eta polynomial}
\begin{equation}
\label{Heq}
\Eta_\la(X\,;Y_{(k)}) := 2^{-\ell_k(\la)}\,R^\la \star \ti_\la. 
\end{equation}
The raising operator expression in (\ref{Heq}) is defined in the same
way as the analogous one in equation (\ref{Etapoly}), but using
$\ti_r$ and $\eta_k,\eta'_k$ in place of $c_r$ and $\tau_k, \tau'_k$,
respectively. 

When $k=0$, the indexing partitions $\la$ are strict, we have that
$\ti_r(X)=q_r(X)$ and $\Ti_\la(X\,;Y_{(0)})=Q_\la(X)$ are the
classical Schur $Q$-functions discussed in \S \ref{sgs}, and
$H_\la(X\,;Y_{(0)})=P_\la(X)$ is a Schur $P$-function. The ring
$\Gamma:=\Z[q_1,q_2,\ldots]$ is the ring of Schur $Q$-functions and
$\Gamma':= \Z[P_1,P_2,\ldots]$ is the ring of Schur $P$-functions (see
\cite{HH} and \cite[III.8]{M}).  For our purposes here we need to know
that $\CS_w(X\,;Y,Z)$ lies in $\Gamma[Y,Z]$ and $\DS_w(X\,;Y,Z)$ lies
in $\Gamma'[Y,Z]$. These assertions follow from Example
\ref{zerosplit} below; see also equation (\ref{dbleCeq}) for a quick
proof of the first one.

\begin{thm}[\cite{LS1, BKT2, BKT4}]
\label{xtoy}
{\em (a)} If $\om\in S_\infty$ is a Grassmannian permutation with a unique
descent at $m$ and $\la$ is the corresponding partition of length at 
most $m$, then we have
\begin{equation}
\label{StoS}
\AS_\om(Y)=s_\la(Y_{(m)}). 
\end{equation}

\smallskip
\noin
{\em (b)} If $w_\la\in W_\infty$ is the
$k$-Grassmannian element corresponding to the $k$-strict partition
$\la$, then we have
\begin{equation}
\label{CtoT}
\CS_{w_\la}(X\,;Y) = \Ti_\la(X\,;Y_{(k)}).
\end{equation}

\smallskip
\noin
{\em (c)} If $w_\la\in\wt{W}_\infty$ is the $k$-Grassmannian element 
corresponding to the typed $k$-strict partition $\la$, then we have
\begin{equation}
\label{DtoT}
\DS_{w_\la}(X\,;Y)=H_\la(X\,;Y_{(k)}).
\end{equation}
\end{thm}

Although the equality of polynomials (\ref{StoS}) in $\Z[Y]$ is
directly analogous to (\ref{CtoT}) and (\ref{DtoT}), the latter two
equations are much harder to prove. One reason for this is that
e.g.\ (\ref{CtoT}) is an equality in the ring
$$\Gamma[Y]=\Z[q_1(X),q_2(X),\ldots\, ;\, y_1,y_2,\ldots],$$ where there
are relations among the $q_r$, and these relations play a crucial role
in the proof. In fact, the proofs of (\ref{CtoT}) and (\ref{DtoT}) in
\cite{BKT2, BKT4} rely on Theorems \ref{IGthm} and \ref{OGthmD} (see
\S \ref{pf45}).  We will generalize Theorem \ref{xtoy} below to a
result (Theorem \ref{dbleACD}) which gives analogous formulas for
arbitrary double Schubert polynomials.

\begin{remark}
Each of the Schur, theta, and eta polynomials in Theorem \ref{xtoy}
may be expressed as a sum of monomials indexed by certain tableaux on
$\la$. The latter are fillings of the boxes of the Young diagram of
$\la$ with (marked) positive integers satisfying natural
conditions. These {\em tableau formulas} have the form
\[
s_\la(Y_{(m)}) = \sum_Ty^T,\
\Ti_\la(X\,;Y_{(k)}) = \sum_U2^{n(U)}(xy)^U,\
H_\la(X\,;Y_{(k)}) = \sum_{V}2^{n(V)}(xy)^V
\]
summed over all semistandard tableaux $T$ with entries at most $m$,
all $k$-bitableaux $U$, and all typed $k'$-bitableaux $V$,
respectively, on the diagram of $\la$. The precise definitions of
these objects are given in \cite{Fu4, T4, T7}.
\end{remark}

\begin{example}
(a) Let $\om$ be a Grassmannian permutation with unique descent at $m$
  and $\la=(\om_m-m, \ldots, \om_1-1)$ be the corresponding
  partition. The double Schubert polynomial $\AS_{\om}(Y,Z)$ is a
  {\em double Schur polynomial} $s_\la(Y_{(m)},Z)$, also known as a
  factorial Schur function. In this case, formula (\ref{dbleA2}) can
  be made more explicit, by working as follows. In any reduced
  factorization $\om=uu'$, the right factor $u'$ is also Grassmannian,
  and corresponds to a partition $\mu$ whose diagram is contained in
  the diagram of $\la$; moreover, we have
  $\AS_{u'}(Y)=s_\mu(Y_{(m)})$. The left factor $u$ is a fully
  commutative element of $S_\infty$ (in the sense of \cite{St2}), and
  it follows from \cite[\S 2]{BJS} that the Schubert polynomial
  $\AS_{u^{-1}}(Z)$ in (\ref{dbleA2}) is a flagged skew Schur
  polynomial. Setting $Z_{(r)}=(z_1,\ldots,z_r)$ for each $r$, we
  conclude (see also \cite[Prop.\ 4.1]{Kre}) that
\[
s_{\la}(Y_{(m)},Z)= \sum_{\mu\subset \la} 
s_\mu(Y_{(m)})\det\left(e_{\la_i-\mu_j-i+j}(-Z_{(\la_i+m-i)})\right)_{1\leq
  i,j\leq m}.
\]

\noindent
(b) Let $\la$ be a $k$-strict partition. In any reduced factorization
$w_\la=uv$ of the $k$-Grassmannian element $w_\la$ in $W_\infty$, the
right factor $v$ is also $k$-Grassmannian, and hence $v=w_\mu$ for
some $k$-strict partition $\mu$. Assuming that $u\in S_\infty$, we
note that $u$ may not be fully commutative (for example, let $k=1$ and
consider the factorization $2\ov{3}1=uv$ with $u=321$ and
$v=2\ov{1}3$), although it is a {\em skew} element of $W_\infty$, in
the sense of \cite[Def.\ 4]{T4}.  Equations (\ref{dbleC2}) and
(\ref{CtoT}) give
\begin{equation}
\label{dbleC3}
\CS_{w_\la}(X\,;Y,Z) = \sum_{uw_\mu=w_\la}\Ti_\mu(X\,;Y_{(k)})\AS_{u^{-1}}(-Z)
\end{equation}
summed over all reduced factorizations $uw_\mu=w_\la$ with $u\in
S_\infty$. It follows from \cite[Prop.\ 4]{T4} that the $k$-strict
partitions $\mu$ in (\ref{dbleC3}) satisfy $\mu\subset \la$ and
$\ell_k(\mu)=\ell_k(\la)$.  Suppose next that $k=0$ and $\la$ is a
strict partition of length $\ell$.  Then $w_\la$ is a fully
commutative element and $\CS_{w_\la}(X\,;Y,Z) = Q_\la(X\,;Z)$ is a
double analogue of Schur's $Q$-function introduced by Ivanov
\cite{I}. Equation (\ref{dbleC3}) becomes
\[
Q_\la(X\,;Z) = \sum_{\mu\subset\la}
Q_\mu(X)\det\left(e_{\la_i - \mu_j}(-Z_{(\la_i-1)})\right)_{1\leq i,j\leq \ell}
\]
summed over all strict partitions $\mu\subset\la$ with
$\ell(\mu)=\ell(\la)=\ell$.
\end{example}

\subsection{Splitting formulas for Schubert polynomials}
\label{splitsps}

Following \cite{BKTY1, T6}, we say that an element $w\in W_\infty$ is
{\em compatible} with the sequence $\fraka \, :\, a_1 < \cdots < a_p$
of elements of $\N_0$ if all descent positions of $w$ are contained in
$\fraka$.  We say that an element $w\in \wt{W}_\infty$ is {\em
  compatible} with the sequence $\fraka \, :\, a_1 < \cdots < a_p$ of
elements of $\N_\Box$ if all descent positions of $w$ are listed among
$\Box, a_1,\ldots, a_p$, if $a_1=1$, or contained in $\fraka$,
otherwise. Such sequences $\fraka$ parametrize parabolic subgroups
of the respective Weyl groups, and $w$ is compatible with $\fraka$ if
and only if it indexes a Schubert variety in the associated
homogeneous space, or a corresponding degeneracy locus (see \S
\ref{cohgp} and \S \ref{dloc}).

Let $w$ be an element of $W_\infty$ (respectively $\wt{W}_\infty$)
compatible with $\fraka$ as above and let $\frakb \, :\, b_1 < \cdots
<b_q$ be a second sequence of elements of $\N_0$ (respectively
$\N_\Box$) such that $w^{-1}$ is compatible with $\frakb$. Suppose
that $w=\om\in S_\infty$ and $a_1>0$. Given any sequence of partitions
$\underline{\la}=(\la^1,\ldots,\la^{p+q-1})$, we define
\begin{equation}
\label{dbfdef00}
c^\om_{\underline{\la}} := \sum_{u_1\cdots u_{p+q-1} = \om}
c_{\la^1}^{u_1}\cdots c_{\la^{p+q-1}}^{u_{p+q-1}}.
\end{equation}
Next suppose that $w \in W_\infty$ and $b_1=0$. Given any sequence of
partitions $\underline{\la}=(\la^1,\ldots,\la^{p+q-1})$ with $\la^q$
$a_1$-strict, we define
\begin{equation}
\label{dbfdef0}
f^w_{\underline{\la}} := \sum_{u_1\cdots u_{p+q-1} = w}
c_{\la^1}^{u_1}\cdots c_{\la^{q-1}}^{u_{q-1}}
e_{\la^q}^{u_q}c_{\la^{q+1}}^{u_{q+1}}\cdots c_{\la^{p+q-1}}^{u_{p+q-1}}.
\end{equation}
Finally, suppose that $w \in \wt{W}_\infty$ and $b_1=\Box$. 
Given any sequence of partitions
$\underline{\la}=(\la^1,\ldots,\la^{p+q-1})$ with $\la^q$ $a_1$-strict
and typed, define 
\begin{equation}
\label{dbfd}
g^w_{\underline{\la}} := \sum_{u_1\cdots u_{p+q-1} = w}
c_{\la^1}^{u_1}\cdots c_{\la^{q-1}}^{u_{q-1}}
d_{\la^q}^{u_q}c_{\la^{q+1}}^{u_{q+1}}\cdots c_{\la^{p+q-1}}^{u_{p+q-1}}.
\end{equation}
The sums in equations (\ref{dbfdef00}), (\ref{dbfdef0}), and
(\ref{dbfd}) are over all reduced factorizations $u_1\cdots u_{p+q-1}
= w$ such that $u_i\in S_\infty$ for all $i\neq q$, $u_j(i)=i$
whenever $j<q$ and $i\leq b_{q-j}$ or whenever $j>q$ and $i \leq
a_{j-q}$ (and where we set $u_j(0)=0$). Moreover, the nonnegative
integers $c^u_\la$, $e^u_\la$, and $d^u_\la$ in these formulas are the
(mixed) Stanley coefficients which were defined in \S \ref{tts}.

It is shown in \cite[\S 4]{BKTY1} that the $c^\om_{\underline{\la}}$
in (\ref{dbfdef00}) may be interpreted as special cases of type A {\em
  quiver coefficients}. These are certain integers which appear in a
formula of Buch and Fulton \cite{BF} for the cohomology class of a
quiver variety (a degeneracy locus associated to a sequence of vector
bundles and vector bundle morphisms over a fixed base, in the
form of an oriented quiver of type A). Equations (\ref{dbfdef0}) and
(\ref{dbfd}) therefore provide analogues $f^w_{\underline{\la}}$ and
$g^w_{\underline{\la}}$ of these coefficients in the other classical
Lie types, for the degeneracy loci considered in \S \ref{dloc}.

 Set  $Y_i := \{y_{a_{i-1}+1},\ldots,y_{a_i}\}$ for each
$i\geq 1$ and $Z_j := \{z_{b_{j-1}+1},\ldots,z_{b_j}\}$ for each
$j\geq 1$. Notice in particular that $Y_1=\emptyset$ if $a_1=0$
or $a_1=\Box$.

\begin{thm}[Splitting Schubert polynomials, \cite{BKTY1, T6}]
\label{dbleACD} 
{\em (a)} Suppose that $a_1>0$ and that $\om$ and $\om^{-1}$ in $S_\infty$
are compatible with the sequences $\fraka$ and $\frakb$,
respectively. Then the Schubert polynomial $\AS_\om(Y,Z)$ satisfies
\begin{equation}
\label{dbASsplitting}
\AS_\om = \sum_{\underline{\la}} c^\om_{\underline{\la}}\,
s_{\la^1}(0/Z_q)\cdots s_{\la^{q-1}}(0/Z_2)
s_{\la^q}(Y_1/Z_1)s_{\la^{q+1}}(Y_2) \cdots s_{\la^{p+q-1}}(Y_p)
\end{equation}
summed over all sequences of partitions
$\underline{\la}=(\la^1,\ldots,\la^{p+q-1})$.

\medskip
\noin
{\em (b)} Suppose that $b_1=0$ and that $w$ and $w^{-1}$ in $W_\infty$
are compatible with the sequences $\fraka$ and $\frakb$,
respectively. Then the Schubert polynomial $\CS_w(X\,;Y,Z)$ satisfies
\begin{equation}
\label{dbCSsplitting}
\CS_w = \sum_{\underline{\la}} 
f^w_{\underline{\la}}\,
s_{\la^1}(0/Z_q)\cdots s_{\la^{q-1}}(0/Z_2)
\Ti_{\la^q}(X\,;Y_1)s_{\la^{q+1}}(Y_2)\cdots s_{\la^{p+q-1}}(Y_p)
\end{equation}
summed over all sequences of partitions 
$\underline{\la}=(\la^1,\ldots,\la^{p+q-1})$
with $\la^q$ $a_1$-strict.

\medskip
\noin
{\em (c)} Suppose that $b_1=\Box$ and that $w$ and $w^{-1}$ in $\wt{W}_\infty$
are compatible with the sequences $\fraka$ and $\frakb$,
respectively. Then the Schubert polynomial $\DS_w(X\,;Y,Z)$ satisfies
\begin{equation}
\label{dbDSsplitting}
\DS_w = \sum_{\underline{\la}} 
g^w_{\underline{\la}}\,
s_{\la^1}(0/Z_q)\cdots s_{\la^{q-1}}(0/Z_2)
H_{\la^q}(X\,;Y_1)s_{\la^{q+1}}(Y_2)\cdots s_{\la^{p+q-1}}(Y_p)
\end{equation}
summed over all sequences of partitions 
$\underline{\la}=(\la^1,\ldots,\la^{p+q-1})$
with $\la^q$ $a_1$-strict and typed.
\end{thm}

Theorem \ref{dbleACD} includes multiple formulas for the double
Schubert polynomials in each Lie type, depending on the choice of the
two sequences $\fraka$ and $\frakb$. For each fixed Weyl group element
$w$, there is a basic formula which involves the descents of $w$ and
$w^{-1}$, defined by choosing $\fraka$ and $\frakb$ to be minimal, and
all the other formulas are obtained from it by splitting the groups of
variables further.

On the other hand, suppose that the sequences $\fraka$ and $\frakb$
are fixed; then the expressions for the Schubert polynomials
displayed in equations (\ref{dbASsplitting}), (\ref{dbCSsplitting}),
and (\ref{dbDSsplitting}) are {\em uniquely determined}. This follows
from the fact that the Schur, theta, and eta polynomials in the
respective groups of variables appearing in Theorem \ref{dbleACD} are
linearly independent over the ring of integers (see also Remark
\ref{rmksix}).  It would be desirable to eliminate the hypothesis on
$b_1$ from parts (b) and (c) of the theorem; for some obstacles in the
way of achieving this, see \cite[Example 3]{T6}.

\begin{example}
\label{zerosplit}
Suppose that $a_1=0$ and $a_1=\Box$ in parts (b) and (c) of Theorem
\ref{dbleACD}. Then we obtain the splitting formulas
\begin{equation}
\label{dbCSsplita10}
\CS_w = \sum_{\underline{\la}} 
f^w_{\underline{\la}}\,
s_{\la^1}(0/Z_q)\cdots s_{\la^{q-1}}(0/Z_2)
Q_{\la^q}(X) s_{\la^{q+1}}(Y_2)\cdots s_{\la^{p+q-1}}(Y_p)
\end{equation}
and 
\begin{equation}
\label{dbDSsplita10}
\DS_w = \sum_{\underline{\la}} 
g^w_{\underline{\la}}\,
s_{\la^1}(0/Z_q)\cdots s_{\la^{q-1}}(0/Z_2)
P_{\la^q}(X) s_{\la^{q+1}}(Y_2) \cdots s_{\la^{p+q-1}}(Y_p),
\end{equation}
respectively, where the sums are over all sequences of partitions
$\underline{\la}=(\la^1,\ldots,\la^{p+q-1})$ with $\la^q$ strict. Thus
we see that in all cases, the Schubert polynomials can be expressed as
a {\em positive} linear combination of products of Jacobi-Trudi
determinants times (at most) a single Schur Pfaffian. Recall also that
in this situation, there exist tableau-based combinatorial expressions
for the coefficients $f^w_{\underline{\la}}$ and $g^w_{\underline{\la}}$. 
Single Schubert polynomial versions of equations (\ref{dbCSsplita10}) 
and (\ref{dbDSsplita10}) were given in \cite{Yo}.
\end{example}

\begin{example}
To simplify the notation, write $\CS_w$ for $\CS_w(X\,;Y,Z)$, $Q_\la$
for $Q_\la(X)$, and $\Ti^{(k)}_\la$ for $\Ti_\la(X\,; Y_{(k)})$. 

\medskip
\noin (a) Let $w=3\ov{1}\ov{2}=s_0s_1s_0s_1s_2s_1\in W_3$ and set
$\frakb= (0<1)$. The following equalities correspond to $\fraka=(1<2)$
and $\fraka=(0<1<2)$, respectively, in Theorem \ref{dbleACD}(b).
\begin{align*}
\CS_{3\ov{1}\ov{2}} &= \Ti^{(1)}_{(4,2)} +\Ti^{(1)}_{(3,2)}\,y_2-z_1\,
\Ti^{(1)}_{(3,2)} \\ &= Q_{(4,2)} +
Q_{(4,1)}\,y_1+Q_{(3,2)}\,y_1+Q_{(3,1)}\,y_1^2+
\left(Q_{(3,2)}+Q_{(3,1)}\,y_1+Q_{(2,1)}\,y_1^2\right)\,y_2 \\ &\quad
-z_1\left(Q_{(3,2)}+Q_{(3,1)}\,y_1+Q_{(2,1)}\,y_1^2\right).
\end{align*}

\medskip
\noin (b) Let $w=12\ov{3}=s_2s_1s_0s_1s_2\in W_3$, set $\frakb=(0<2)$,
$e_i^y=e_i(y_1,y_2)$ and $e_i^z=e_i(z_1,z_2)$ for $i=1,2$. The
following equalities correspond to taking $\fraka$ to be $(2)$, 
$(1<2)$, and $(0<2)$, respectively, in Theorem \ref{dbleACD}(b).
\begin{align*}
\CS_{12\ov{3}} &= \Ti^{(2)}_5-e_1^z\,\Ti^{(2)}_4+e^z_2\,\Ti^{(2)}_3 \\
&=\Ti^{(1)}_5+\Ti^{(1)}_4\,y_2-e_1^z\left(\Ti^{(1)}_4+\Ti^{(1)}_3\,y_2\right)
+e_2^z\left(\Ti^{(1)}_3+\Ti^{(1)}_2\,y_2\right) \\
&= Q_5 + Q_4\,e_1^y+Q_3\,e_2^y-e_1^z\left(Q_4 + Q_3\,e_1^y+Q_2\,e_2^y\right)
+e_2^z\left(Q_3+Q_2\,e_1^y+Q_1\,e_2^y\right).
\end{align*}
\end{example}

\section{Degeneracy loci}
\label{dloc}

In this section, we show how the splitting formulas of Theorem
\ref{dbleACD} translate directly into Chern class formulas for
degeneracy loci in the sense of \cite{Fu3} and \S \ref{ecdl}, with the
symmetries native to the corresponding $G/P$ space.

Let $c=(c_0,c_1,c_2,\ldots)$ and $d=(d_0,d_1,d_2,\ldots)$ be two
families of commuting variables, with $c_0=d_0=1$ as usual. We first
extend the definitions of the polynomials $s_\la(c)$, $\Ti_\la(c)$,
and $\Eta_\la(c)$ from \S \ref{class} and \S \ref{isogiam} to obtain
polynomials in the formal difference $c-d$ (this notation stems from
the theory of $\lambda$-rings). Define elements $g_r$ and $h_r$ for
$r\in \Z$ by the identities of formal power series
\begin{gather*}
\sum_{r=-\infty}^{+\infty}g_rt^r := 
\left(\sum_{i=0}^\infty c_it^i\right)
\left(\sum_{i=0}^\infty d_it^i\right)^{-1} \ ; \\
\sum_{r=-\infty}^{+\infty}h_rt^r := 
\left(\sum_{i=0}^\infty(-1)^ic_it^i\right)^{-1}
\left(\sum_{i=0}^\infty(-1)^id_it^i\right).
\end{gather*}
For any partition $\la$, define the {\em Schur
polynomial} $s_\la(c-d)$ by
\[
s_\la(c-d) := R^{0}\, h_\la = \det(h_{\la_i+j-i})_{i,j}.
\]
For any $k$-strict partition $\la$, define the 
theta polynomial $\Ti_\la(c-d)$ by
\[
\Theta_\la(c-d) := R^\la\, g_\la, 
\]
where $R^\la$ denotes the raising operator in (\ref{ReqC}). Finally,
for any typed $k$-strict partition $\la$, define the eta polynomial
$\Eta_\la(c-d)$ by
\[
\Eta_\la(c-d) := 2^{-\ell_k(\la)} R^\la \star g_\la,
\]
where $R^\la$ denotes the raising operator in (\ref{ReqD}), and the
action $\star$ is defined as in \S \ref{ogs}, replacing $c_\la$ by
$g_\la$ throughout. If $V$ and $V'$ are two vector bundles on an
algebraic variety $M$ with total Chern classes $c(V)$ and $c(V')$,
respectively, we denote the class $s_\la(c(V)-c(V'))$ by
$s_\la(V-V')$, the class $\Ti_\la(c(V)-c(V'))$ by $\Ti_\la(V-V')$, and
the class $\Eta_\la(c(V)-c(V'))$ by $\Eta_\la(V-V')$.

\subsection{Type A degeneracy loci}
\label{dlA}

Fix a sequence $\fraka \,
:\, a_1 < \cdots < a_p$ of positive integers with $a_p < n$, which
gives a subset of the nodes of the Dynkin diagram for the root system
of type $\text{A}_{n-1}$:
\medskip
\[
\includegraphics[scale=0.35]{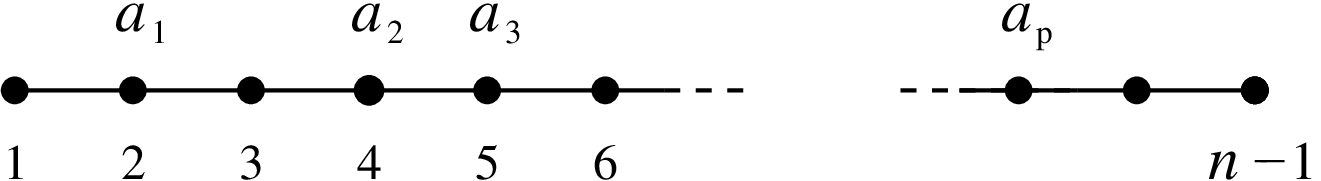}
\]
Let $S^\fraka$ be the set of permutations $\om\in S_n$ whose descent
positions are listed among the integers $a_1,\ldots, a_p$, i.e., that
are compatible with $\fraka$. These elements are the minimal length
coset representatives in $S_n/S_\fraka$, where $S_\fraka$ denotes the
parabolic subgroup of $S_n$ generated by the simple transpositions
$s_i$ for $i\notin \{a_1,\ldots, a_p\}$.

Let $E\to M$ be a vector bundle of rank $n$ on an algebraic variety
$M$, assumed to be smooth for simplicity. 
Consider a partial flag of subbundles of $E$
\[
0 \subset E_1 \subset \cdots \subset E_p \subset E
\]
with $\rank E_r = a_r$ for each $r$, and a complete flag $$0 \subset
F_1\subset \cdots \subset F_n = E$$ of subbundles of $E$ (with $\rank
F_s = s$ for each $s$).  For every $\om\in S^{\fraka}$, we define
the {\em degeneracy locus} $\X_\om\subset M$ as the locus of $b \in M$
such that
\[
\dim(E_r(b)\cap F_s(b))\geq \#\,\{\,i \leq a_r \ |\ \om_i> n-s\,\}
\ \, \forall \, r,s.
\]
A precise definition of $\X_\om$ as a subscheme of $\X$ can be obtained
by pulling back from the universal case, which occurs on the partial
flag bundle $\F^\fraka(E)$ (see \cite{Fu3} and \cite[\S 6.2 and
  App.\ A]{FP} for more details). Assume further that $\X_\om$ has
pure codimension $\ell(\om)$ in $M$, which is the case when the vector
bundles are in general position. The next result will be a formula for
the cohomology class $[\X_\om]$ in $\HH^{2\ell(\om)}(M)$ in terms of the
Chern classes of the bundles $E_r$ and $F_s$.

Consider a second sequence $\frakb \, :\, 0 \leq b_1 < \cdots <b_q$
with $b_q <n$.

\begin{thm}[\cite{BKTY1}]
\label{dbleAloci}
Suppose that $\om\in S^{\fraka}$ and that $\om^{-1}$ is compatible with 
$\frakb$. Then we have
\begin{align*}
[\X_\om] 
&= \sum_{\underline{\la}} c^\om_{\underline{\la}}\,
s_{\wt{\la}^1}(F_{n-b_{q-1}} - F_{n-b_q})\cdots
s_{\wt{\la}^q}(E-E_1-F_{n-b_1})
\cdots s_{\wt{\la}^{p+q-1}}(E_{p-1}-E_p)
\end{align*}
in $\HH^*(M)$, where the sum is over all sequences of
partitions $\underline{\la}=(\la^1,\ldots,\la^{p+q-1})$ and the 
coefficients $c^w_{\underline{\la}}$ are given
by {\em (\ref{dbfdef00})}.
\end{thm}

\subsection{Symplectic degeneracy loci}
\label{sdl}

Fix a sequence $\fraka \, :\, a_1 < \cdots < a_p$
of nonnegative integers with $a_p <n$, which is a subset of the nodes
of the Dynkin diagram for the root system of type $\text{C}_n$:
\medskip
\[
\includegraphics[scale=0.35]{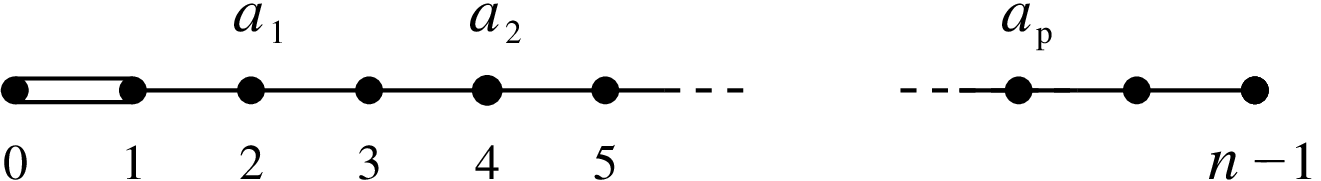}
\]
Denote by $W^\fraka$ the set of signed permutations $w\in W_n$ whose
descent positions are listed among the integers $a_1,\ldots,
a_p$.

Let $E\to M$ be a vector bundle of rank $2n$ on a smooth algebraic
variety $M$. Assume that $E$ is a {\em symplectic} bundle, so that $E$
is equipped with an everywhere nondegenerate skew-symmetric form
$E\otimes E\to \C$.  Consider a partial flag of subbundles of $E$
\[
0 \subset E_p \subset \cdots \subset E_1 \subset E
\]
with $\rank E_r = n-a_r$ and $E_1$ isotropic, and a (complete)
isotropic flag $$0 \subset F_1\subset \cdots \subset F_{2n}=E$$ of
subbundles of $E$ with $\rank F_s = s$ for each $s$ (and
$F_{n+s}=F_{n-s}^{\perp}$ for $0\leq s < n$). Fix a second sequence
$\frakb \, :\, 0=b_1 < \cdots <b_q$ with $b_q <n$, and define
quotient bundles
\begin{equation}
\label{quoteq}
\begin{split}
Q_1 & := E/E_1,\  Q_2 :=
E_1/E_2,\ \ldots, \ Q_p := E_{p-1}/E_p \\  
&\wh{Q}_2 := F_n/F_{n-b_2}, \ 
\ldots, \ \wh{Q}_q := F_{n-b_{q-1}}/F_{n-b_q}.
\end{split}
\end{equation}

There is a group monomorphism $\phi:W_n\hra S_{2n}$ with image
\[
\phi(W_n)=\{\,\om\in S_{2n} \ |\ \om_i+\om_{2n+1-i} = 2n+1,
 \ \ \text{for all}  \ i\,\}.
\]
The map $\phi$ is determined by setting,
for each $w=(w_1,\ldots,w_n)\in W_n$ and $1\leq i \leq n$,
\[
\phi(w)_i :=\left\{ \begin{array}{cl}
             n+1-w_{n+1-i} & \mathrm{ if } \ w_{n+1-i} \ \mathrm{is} \
             \mathrm{unbarred}, \\
             n+\ov{w}_{n+1-i} & \mathrm{otherwise}.
             \end{array} \right.
\]
For every $w\in W^{\fraka}$ we have the {\em degeneracy locus}
$\X_w\subset M$, which is the locus of $b \in M$ such that
\[
\dim(E_r(b)\cap F_s(b))\geq \#\,\{\,i \leq n-a_r
\ |\ \phi(w)_i > 2n-s\,\} \ \, \forall \, r,s.
\]
As in the type A case, assuming that $\X_w$ has
pure codimension $\ell(w)$ in $M$, we give a formula for the 
class $[\X_w]$ in $\HH^{2\ell(w)}(M)$.

\begin{thm}[\cite{T6}]
\label{dbleCloci}
Suppose that $w\in W^{\fraka}$ and that $w^{-1}$ is compatible with 
$\frakb$. Then we have
\begin{align*}
[\X_w] &= \sum_{\underline{\la}} f^w_{\underline{\la}}\,
s_{\wt{\la}^1}(\wh{Q}_q)\cdots s_{\wt{\la}^{q-1}}(\wh{Q}_2)
\Ti_{\la^q}(Q_1-F_n) s_{\la^{q+1}}(Q_2)\cdots s_{\la^{p+q-1}}(Q_p) \\
&= \sum_{\underline{\la}} f^w_{\underline{\la}}\,
s_{\la^1}(F_{n+b_{q-1}} - F_{n+b_q})\cdots
\Ti_{\la^q}(E-E_1-F_n)
\cdots s_{\la^{p+q-1}}(E_{p-1}-E_p)
\end{align*}
in $\HH^*(M)$, where the sum is over all sequences of
partitions $\underline{\la}=(\la^1,\ldots,\la^{p+q-1})$ with $\la^q$
$a_1$-strict, and the coefficients $f^w_{\underline{\la}}$ are given
by {\em (\ref{dbfdef0})}.
\end{thm}

\subsection{Orthogonal degeneracy loci}
\label{odl}

\subsubsection{The odd orthogonal case}

Let $E\to M$ be a vector bundle of rank $2n+1$ on a smooth algebraic
variety $M$. Assume that $E$ is an {\em orthogonal} bundle, i.e.\ $E$
is equipped with an everywhere nondegenerate symmetric form $E\otimes
E\to \C$. Fix a sequence $\fraka \, :\, a_1 < \cdots < a_p$ of
nonnegative integers with $a_p <n$, as in the symplectic case. Consider a
partial flag of subbundles of $E$
\[
0 \subset E_p \subset \cdots \subset E_1 \subset E
\]
with $\rank E_r = n-a_r$ and $E_1$ isotropic, and a complete isotropic
flag $$0\subset F_1\subset \cdots \subset F_{2n+1}=E$$ of subbundles
of $E$.  Let $\frakb \, :\, 0=b_1 < \cdots <b_q$ be a sequence of
integers with $b_q <n$ and define the quotient bundles $Q_r$ and
$\wh{Q}_s$ using the same equations (\ref{quoteq}) as in type C.

There is a group monomorphism $\phi':W_n\hra S_{2n+1}$ with image
\[
\phi'(W_n)=\{\,\om\in S_{2n+1} \ |\ \om_i +\om_{2n+2-i} = 2n+2,
 \ \ \text{for all}  \ i\,\}.
\]
The map $\phi$ is determined by setting,
for each $w=(w_1,\ldots,w_n)\in W_n$ and $1\leq i \leq n$,
\[
\phi'(w)_i :=\left\{ \begin{array}{cl}
             n+1-w_{n+1-i} & \mathrm{ if } \ w_{n+1-i} \ \mathrm{is} \
             \mathrm{unbarred}, \\
             n+1+\ov{w}_{n+1-i} & \mathrm{otherwise}.
             \end{array} \right.
\]

Fix an element $w$ in $W^{\fraka}$. The {\em degeneracy locus}
$\X_w\subset M$ is the locus of $b \in M$ such that
\[
\dim(E_r(b)\cap F_s(b))\geq \#\,\{\,i \leq n-a_r
\ |\ \phi'(w)_i > 2n+1-s\,\} \ \, \forall \, r,s,
\]
and we assume as before that $\X_w$ has pure codimension
$\ell(w)$ in $M$.  For any $a_1$-strict partition $\la$, let
$\Ti'_\la= 2^{-s(w)}\Ti_\la$, where $s(w)=\#\{i\ |\ w_i <0\}$.  We
then have the following analogue of Theorem \ref{dbleCloci}.

\begin{thm}[\cite{T6}]
\label{dbleBloci}
Suppose that $w\in W^{\fraka}$ and that $w^{-1}$ is compatible with 
$\frakb$. Then we have
\begin{gather*}
[\X_w] = \sum_{\underline{\la}} f^w_{\underline{\la}}\,
s_{\wt{\la}^1}(\wh{Q}_q)\cdots s_{\wt{\la}^{q-1}}(\wh{Q}_2)
\Ti'_{\la^q}(Q_1-F_{n+1}) s_{\la^{q+1}}(Q_2)\cdots s_{\la^{p+q-1}}(Q_p) \\
= \sum_{\underline{\la}} f^w_{\underline{\la}}\,
s_{\la^1}(F_{n+1+b_{q-1}} - F_{n+1+b_q})\cdots
%s_{\la^{q-1}}(F_n-F_{n+b_2}) 
\Ti'_{\la^q}(E-E_1-F_{n+1})
%s_{\la^{q+1}}(E_p-E_{p-1})
\cdots s_{\la^{p+q-1}}(E_{p-1}-E_p)
\end{gather*}
in $\HH^*(M)$, where the sum is over all sequences of
partitions $\underline{\la}=(\la^1,\ldots,\la^{p+q-1})$ with $\la^q$
$a_1$-strict, and the coefficients $f^w_{\underline{\la}}$ are given
by {\em (\ref{dbfdef0})}.
\end{thm}

\subsubsection{The even orthogonal case}
\label{eodl}

Let $E\to M$ be an orthogonal vector bundle of rank $2n$ on a smooth
algebraic variety $M$. Fix a complete isotropic flag $$0 \subset
F_1\subset \cdots \subset F_{2n}=E$$ of subbundles of $E$, and a
sequence $\fraka \, :\, a_1 < \cdots < a_p$ of elements of $\N_\Box$
with $a_p <n$, which gives a subset of the vertices of the Dynkin
diagram for the root system of type $\text{D}_n$:
\medskip
\[
\includegraphics[scale=0.35]{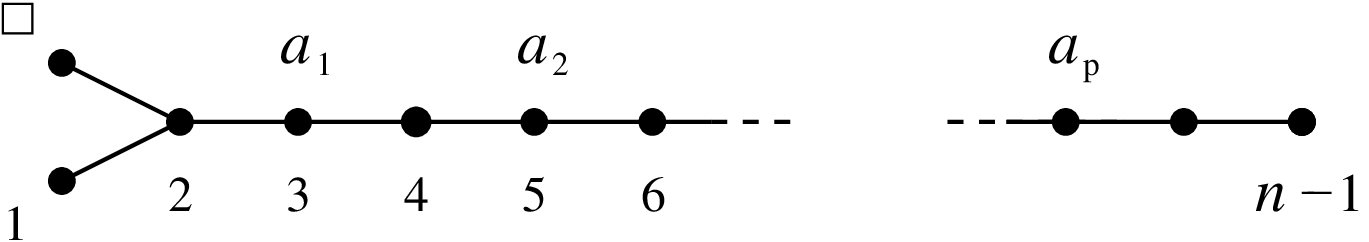}
\]
Consider a partial flag of subbundles of $E$
\[
0 \subset E_p \subset \cdots \subset E_1 \subset E
\]
with $\rank E_r = n-a_r$ and $E_1$ isotropic and in the same family as
$F_n$ if $a_1=\Box$. Fix a sequence $\frakb \, :\, \Box =b_1 < \cdots
<b_q$ with $b_q <n$, and define the quotient bundles $Q_r$ and
$\wh{Q}_s$ as in type C, using the equations (\ref{quoteq}).

Recall that $\wt{W}_n$ is a subgroup of $W_n$, and hence we have a
group monomorphism $\phi:\wt{W}_n\hra S_{2n}$, defined by restricting
the map $\phi$ of \S \ref{sdl} to $\wt{W}_n$.  Let $\wt{W}^\fraka$ be
the set of signed permutations $w\in \wt{W}_n$ which are compatible with
$\fraka$. For every $w\in \wt{W}^{\fraka}$, we define the {\em degeneracy
locus} $\X_w\subset M$ as the closure of the
locus of $b \in M$ such that
\[
\dim(E_r(b)\cap F_s(b)) = \#\,\{\,i \leq n-a_r
\ |\ \phi(w_0ww_0)_i > 2n-s\,\} \ \, \forall \, r,s,
\]
with the reduced scheme structure. Assume further that $\X_w$ has
pure codimension $\ell(w)$ in $M$, and consider its cohomology class
$[\X_w]$ in $\HH^{2\ell(w)}(M)$.

\begin{thm}[\cite{T6}]
\label{dbleDloci}
Suppose that $w\in \wt{W}^{\fraka}$ and that $w^{-1}$ is compatible with 
$\frakb$. Then we have
\begin{align*}
[\X_w] &= \sum_{\underline{\la}} g^w_{\underline{\la}}\,
s_{\wt{\la}^1}(\wh{Q}_q)\cdots s_{\wt{\la}^{q-1}}(\wh{Q}_2)
\Eta_{\la^q}(Q_1-F_n) s_{\la^{q+1}}(Q_2)\cdots s_{\la^{p+q-1}}(Q_p) \\
&= \sum_{\underline{\la}} g^w_{\underline{\la}}\,
s_{\la^1}(F_{n+b_{q-1}} - F_{n+b_q})\cdots
%s_{\la^{q-1}}(F_n-F_{n+b_2}) 
\Eta_{\la^q}(E-E_1-F_n)
%s_{\la^{q+1}}(E_p-E_{p-1})
\cdots s_{\la^{p+q-1}}(E_{p-1}-E_p),
\end{align*}
in $\HH^*(M)$, where the sum is over all sequences of
partitions $\underline{\la}=(\la^1,\ldots,\la^{p+q-1})$ with $\la^q$
$a_1$-strict and typed, and the coefficients $g^w_{\underline{\la}}$
are given by {\em (\ref{dbfd})}.
\end{thm}

\begin{example}
Suppose that $E=\C^N$ is a complex vector space, which in types B, C,
and D is equipped with a nondegenerate symmetric or skew-symmetric
bilinear form. The homogeneous space $M=F^{\fraka}(E)$ which
parametrizes all partial (isotropic) flags $E_\bull$ in $E$ is a $G/P$
space for the appropriate classical group $G$ and parabolic subgroup
$P$ corresponding to the sequence $\fraka$. In this case $E_\bull$ is
also used to denote the tautological filtration of the trivial rank
$N$ vector bundle $E$ over $M$, the vector bundles $F_s$ are all trivial,
and the formulas of this section solve the Giambelli problem for
$\HH^*(G/P)$ (see \cite[\S 5]{T6}). If we furthermore assume that $p=1$,
so that $\fraka=\{a_1\}$ and $M$ is an (isotropic) Grassmannian, then
Theorems \ref{dbleAloci}, \ref{dbleCloci}, and \ref{dbleDloci}
specialize to Theorems \ref{Gthm}, \ref{IGthm}, and \ref{OGthmD},
respectively.
\end{example}

Using the identification of degeneracy locus formulas with Giambelli
polynomials in equivariant cohomology (see \cite{Gra} and \ref{ecdl}
and \S \ref{ddgeom} of this paper), the theorems in this section
translate easily into corresponding results for the $T$-equivariant
cohomology ring of $G/P$. The Chern roots of the vector bundle $F_n$
(up to a sign) are identified with the weights of the linear
$T$-action on the standard representation $V$ of $G$; specifically,
the class $-c_1(F_{n+1-i}/F_{n-i})$ is mapped onto the variable $t_i$
in the notation of \cite[\S 10]{IMN} (and the variable $\y_i$ in \S
\ref{ddgeom} below). The results may also be used to deduce formulas
for the restriction of an equivariant Schubert class in $\HH^*_T(G/P)$
to a torus fixed point; see op.\ cit.\ for more details.

\section{Proofs of the main theorems}
\label{pfs}

We discuss here what is involved in proving the aforementioned
theorems. We pay particular attention to the proofs of Theorems
\ref{dbleAloci}--\ref{dbleDloci}, and establish the connection between
the Schubert polynomials of \cite{IMN} and the geometry of degeneracy
loci.

\subsection{Proofs of $\IG$ and $\OG$ Giambelli} 

Although briefly mentioned in \S \ref{class} and \S \ref{isogiam}, so
far there has been little discussion of {\em Pieri rules} for
isotropic Grassmannians. The first such rules were proved by Hiller
and Boe \cite{HB} for the maximal isotropic Grassmannians, and applied
by Pragacz \cite{P2} to prove his Giambelli formulas for these
spaces. Pragacz and Ratajski \cite{PR1, PR3} later obtained Pieri type
rules for general isotropic Grassmannians, but used a different notion
of `special Schubert classes' than ours (their special classes were
the Chern classes of the universal subbundles $E'$ in \S
\ref{isogiam}).

The proofs of Theorems \ref{IGthm} and \ref{OGthmD} from \cite{BKT2,
  BKT4} require the classical Pieri rules from \cite{BKT1}, which hold
for the special Schubert classes $c_r$ and $\tau_r$ defined in \S
\ref{isogiam}. For the symplectic Grassmannian $\IG(n-k,2n)$, the Pieri rule
has the form
\begin{equation}
\label{IGpieri}
 c_p \cdot [X_\lambda] = \sum_{\lambda \xrightarrow{p} \mu} 
 2^{N(\lambda,\mu)} \, [X_\mu] \,.
\end{equation}
When $k>0$, the Pieri relation $\lambda \xrightarrow{p} \mu$ and the
definition of the exponents $N(\la,\mu)$ is more complicated than in
the $k=0$ case of the Lagrangian Grassmannian. However, the Pieri rule
(\ref{IGpieri}) can still be used recursively to show that a general
Schubert class $[X_\la]$ may be written as a polynomial in the special
classes. Therefore, in order to prove Theorem \ref{IGthm}, it suffices
to show that the expression $R^\la\, c_\la$ satisfies the same Pieri
rule (\ref{IGpieri}), and this is the approach taken in \cite{BKT2}.
The complex argument that proves that $R^\la\,c_\la$ obeys
(\ref{IGpieri}) uses subtle alternating properties of $R^\la \, c_\la$
which depend on the order ideal $\cC(\la)$, and a substitution
algorithm which is a mathematical model of controlled evolution. The
reader will find an exposition of the main ideas in \cite[\S
  3]{T4}. The proof of Theorem \ref{OGthmD} is similar to that of
Theorem \ref{IGthm}. It exploits the weight space decomposition of
$\HH^*(\OG(m,2n),\Q)$ induced by the natural involution of the Dynkin
diagram of type $\text{D}_n$, and a surprising relation between the
Schubert calculus on even and odd orthogonal Grassmannians.

Currently, there is a second, simpler proof of Theorem \ref{IGthm},
using the {\em double theta polynomials} $\Ti_\la(c\, |\, t)$ defined
by Wilson \cite{TW, W}. These polynomials are analogues of the usual
double (or factorial) Schur polynomials, and give equivariant
Giambelli expressions for the Schubert classes in
$\HH^*_T(\IG(m,2n))$. This was proved by Ikeda and Matsumura \cite{IM}
and by the author and Wilson \cite{TW} in two different ways. Both
proofs rely on a key technical result from \cite{IM}, which is a
certain compatibility of the $\Ti_\la(c\, |\, t)$ with the action of
$t$-divided difference operators.

\subsection{Proofs of Theorems \ref{xtoy} and \ref{dbleACD}}
\label{pf45}

Following \cite{LS1, BH, FK2} or by specializing the arguments of \S
\ref{ddgeom}, one shows that the single Schubert polynomials represent
the Schubert classes in $\HH^*(G/B)$, and that the algebra that they
span is isomorphic to the stable cohomology ring $\IH(G/B)$ of the
complete flag manifold as $n\to \infty$. Therefore, to prove Theorem
\ref{xtoy} it suffices to show that (a) the Schur, theta, and eta
polynomials indexed by single row partitions agree with the
corresponding single Schubert polynomials, and (b) the Giambelli
formulas of \S \ref{class} and \S \ref{isogiam} hold in $\IH(G/B)$.
Assertion (a) is easy to check directly from the definition of the
Schubert polynomials, and (b) is an immediate consequence of Theorems
\ref{Gthm}, \ref{IGthm}, and \ref{OGthmD}.

The proof of Theorem \ref{dbleACD} is based on the transition
equations for Schubert polynomials mentioned in \S \ref{tts}, and
some important actors that have remained largely in the background 
until now, the {\em Stanley symmetric functions} of \cite{Sta, BH, FK2}
and the {\em mixed Stanley functions} of \cite{T6}. We will state
the key result here for the type C Billey-Haiman Schubert polynomials
$\CS_w(X\,;Y)$. Let $\CS_w(X\,;Y_{(k)})$ denote the power series
obtained from $\CS_w(X\,;Y)$ by setting $y_j=0$ for all $j>k$.

\begin{thm}[\cite{T6}]
\label{StanC}
Suppose that $w\in W_\infty$ is increasing up to $k$. Then we have
\begin{equation}
\label{CStan}
\CS_w(X\,;Y_{(k)}) = \sum_{\la} e^w_{\la}\,\Ti_{\la}(X\,;Y_{(k)}),
\end{equation}
where the sum is over all $k$-strict partitions $\la$ with $|\la|=\ell(w)$.
\end{thm}

\noindent
For instance, the $k=1$ transition tree in Example \ref{treeC} 
corresponds to the equation
\[
\CS_{3\ov{1}2645}(X\,;Y_{(1)}) = \Ti_{(2,1,1,1)} + \Ti_5 +
2\,\Ti_{(3,1,1)}+\Ti_{(4,1)}+ \Ti_{(3,2)}.
\]
Taking $k=0$ in Theorem \ref{StanC} gives the expansion of
the type C Stanley symmetric function $F_w$ as a sum of Schur 
$Q$-functions, which was first proved by Billey
\cite{Bi1}:
\begin{equation}
\label{FtoQ}
F_w(X) = \sum_{\la\, :\, |\la| = \ell(w)} e^w_{\la}\,Q_{\la}(X)
\end{equation}
with the sum over strict partitions $\la$ of $\ell(w)$.

When $w$ is increasing up to $k$, we can identify $\CS_w(X\,;Y_{(k)})$
in (\ref{CStan}) with the restriction $J_w(X\,;Y_{(k)})$ of a `mixed
Stanley function' $J_w(X\,;Y)$ (see \cite[\S 2]{T6} for the precise
definitions). Furthermore, one shows directly from their definition in
\S \ref{nilSch} that the Schubert polynomials satisfy splitting
formulas which express them as sums of products of (mixed) Stanley
functions in mutually disjoint groups of variables. Theorem
\ref{dbleACD} follows immediately by combining these two ingredients.

\subsection{Divided differences and geometrization maps}
\label{ddgeom}

We define an action of $W_\infty$ on $\Gamma[Y,Z]$ by ring
automorphisms as follows. The simple reflections $s_i$ for $i>0$ act
by interchanging $y_i$ and $y_{i+1}$ and leaving all the remaining
variables fixed. The reflection $s_0$ maps $y_1$ to $-y_1$, fixes the
$y_j$ for $j\geq 2$ and all the $z_j$, and satisfies
\begin{equation}
\label{s_zero}
s_0(q_r(X)) = q_r(y_1,x_1,x_2,\ldots) =
q_r(X)+2\sum_{j=1}^ry_1^jq_{r-j}(X).
\end{equation}
For each $i\geq 0$, define the {\em divided difference operator}
$\partial_i^y$ on $\Gamma[Y,Z]$ by
\[
\partial_0^yf := \frac{f-s_0f}{-2y_1}, \qquad
\partial_i^yf := \frac{f-s_if}{y_i-y_{i+1}} \ \ \ \text{for $i>0$}.
\]
Consider the ring involution $\omega:\Gamma[Y,Z]\to\Gamma[Y,Z]$
determined by
\[
\omega(y_j) = -z_j, \qquad
\omega(z_j) = -y_j, \qquad
\omega(q_r(X))=q_r(X)
\]
and set $\partial_i^z:=\omega\partial_i^y\omega$ for each $i\geq 0$.

\begin{thm}[\cite{IMN}]
\label{uniq}
The polynomials $\CS_w(X\,;Y,Z)$ for $w\in W_{\infty}$ are the unique
family of elements of $\Gamma[Y,Z]$ satisfying the equations
\begin{equation}
\label{ddeqs}
\partial_i^y\CS_w = \begin{cases}
\CS_{ws_i} & \text{if $\ell(ws_i)<\ell(w)$}, \\ 
0 & \text{otherwise},
\end{cases}
\quad
\partial_i^z\CS_w = \begin{cases}
\CS_{s_iw} & \text{if $\ell(s_iw)<\ell(w)$}, \\ 
0 & \text{otherwise},
\end{cases}
\end{equation}
for all $i\geq 0$, together with the condition that the constant term 
of $\CS_w$ is $1$ if $w=1$, and $0$ otherwise. 
\end{thm}
\begin{proof}
Define $\CS^{(n)}:= C(X)A_1(y_1)\cdots A_{n-1}(y_{n-1})$, so that we
have $\CS_w(X\,;Y)= \left\langle \CS^{(n)}, w \right\rangle$, for any
$w\in W_n$.  In \cite[Lemma 3.5]{FS} and \cite[Theorem 7.1]{FK2}
Fomin, Stanley, and Kirillov provide simple proofs (using the
nilCoxeter relations) that the single Schubert polynomials
$\CS_w(X\,;Y)$ satisfy the equations in (\ref{ddeqs}) which involve
the $\partial_i^y$ operators. We give the proof for $\partial_0^y$
below for the reader's convenience, since our setup differs from that
in \cite{FK2}. It follows from (\ref{FtoQ}) and (\ref{s_zero}) that
$s_0 C(X) = C(X)C(y_1)$, and we clearly have 
\[ 
C(y_1)A_1(-y_1)= A_1(y_1)(1+y_1 u_0)(1+y_1 u_0) = A_1(y_1)(1+2y_1 u_0).
\]
We deduce that
\begin{align*}
\partial_0^y\, \CS^{(n)} &=  
\frac{1}{-2y_1} \left( C(X)(A_1(y_1)-C(y_1)A_1(-y_1))A_2(y_2)\cdots
A_{n-1}(y_{n-1}) \right) \\
&= \frac{1}{-2y_1} \, C(X)A_1(y_1)(-2y_1 u_0)A_2(y_2)\cdots
A_{n-1}(y_{n-1}) = \CS^{(n)}u_0,
\end{align*}
which is equivalent the desired conditions
$\partial_0^y\,\CS_w(X\,;Y) = \CS_{ws_0}(X\,;Y)$, if
$\ell(ws_0)<\ell(w)$, and $\partial_0^y\,\CS_w(X\,;Y)=0$,
otherwise. Moreover, in view of equation (\ref{dbleC2}), the arguments
here and in loc.\ cit.\ extend easily to show that the double Schubert
polynomials $\CS_w(X\,;Y,Z)$ fulfill the entire list of conditions in
the theorem.

Following \cite[\S 7.4]{IMN}, the uniqueness is shown as follows. If
$\{\CS'_w\}$ for $w\in W_\infty$ is a second family of elements of
$\Gamma[Y,Z]$ satisfying the displayed conditions, then by inducting
on the length of $w$ one sees that $\partial_i^y(\CS'_w-\CS_w) =
\partial_i^z(\CS'_w-\CS_w)=0$ for all $i\geq 0$. We deduce that the
difference $\CS'_w-\CS_w$ is invariant under the action of $s_i$ and
$\omega s_i \omega$ for every $i$, and hence must be a
constant. Finally, the condition on the constant term implies that
$\CS'_w=\CS_w$, for all $w\in W_\infty$.
\end{proof}

The connection between Theorem \ref{dbleACD} and Theorems
\ref{dbleAloci}, \ref{dbleCloci}, \ref{dbleBloci}, and \ref{dbleDloci}
depends on an important ring homomorphism derived from \cite{T2, T3},
\cite[\S 10]{IMN}, and \cite{Gra}, which we call the {\em
  geometrization map}. We will discuss this homomorphism in detail in
the Lie types A, C, and D (leaving type B as an exercise for the
reader), and use it to give a complete proof of Theorem
\ref{dbleDloci}.

Let $G$ denote the group $\GL_n(\C)$, $\Sp_{2n}(\C)$, or
$\SO_{2n}(\C)$ with its standard representation $V=\C^n$ for $\GL_n$
or $V=\C^{2n}$ in the latter two cases. In type A, equip $V$ with the
zero form, and in types C and D equip $V$ with an antidiagonal
symplectic or orthogonal form $(\ ,\, )$, so that the standard basis
$\{e_1,\ldots,e_{2n}\}$ of $V$ satisfies $(e_i,e_j) = 0$ for $i+j \neq
2n+1$ and $(e_i,e_{2n+1-i}) = 1$, for $1\leq i \leq n$. We obtain an
induced vector bundle $E=EG\times^GV$ over $BG$ and bilinear form
$E\otimes E\to \C$. Let $V_\bull$ be the isotropic flag in $V$ with
$V_i = \langle e_1,\ldots,e_i\rangle$ for each $i\in [1,n]$, and $B$
denote the stabilizer of $V_\bull$. The pullback of the bundle $E$ to
$BB$ has an isotropic flag $\V'_\bull=\{EG\times^B V_i\}_i$ of
subbundles of $E$.  If $\mathrm{pr}_1$ and $\mathrm{pr}_2$ are the two
projection maps $BB\times_{BG} BB \to BB$, then we obtain the two
isotropic flags of subbundles $\V_\bull:=\mathrm{pr}_1^*\,\V'_\bull$ and
$F_\bull:=\mathrm{pr}_2^*\,\V'_\bull$ of the pullback of $E$ to
$BB\times_{BG}BB$. This is the universal case of the degeneracy locus
problems considered in \S \ref{dlA}, \S \ref{sdl}, and \S \ref{eodl},
for the parabolic subgroup $P=B$.

\medskip
\noin {\bf Type A.} Introduce two new sets of commuting variables
$\XX=(\x_1,\x_2,\ldots)$, $\YY=(\y_1,\y_2,\ldots)$ and let 
$\XX_n=(\x_1,\ldots,\x_n)$ and $\YY_n=(\y_1,\ldots,\y_n)$.
In this case $G=\GL_n$ and  it follows from the above
discussion and (\ref{natiso}) that there is a natural presentation
\begin{equation}
\label{presentationA}
\HH^*(BB\times_{BG} BB,\Z) \cong \Z[\XX_n,\YY_n]/\I_n,
\end{equation}
where $\I_n$ is the ideal generated by the differences
$e_i(\x_1,\ldots,\x_n)- e_i(\y_1,\ldots,\y_n)$ for $1\leq i
\leq n$. The inverse of the isomorphism (\ref{presentationA}) sends the
class of $\x_i$ to $-c_1(\V_i/\V_{i-1})$ and of $\y_i$ to
$-c_1(F_{n+1-i}/F_{n-i})$ for each $i$ with $1\leq i \leq n$.

The {\em geometrization map} is the ring homomorphism
\[
\rho_n : \Z[Y,Z] \to \Z[\XX_n,\YY_n]/\I_n
\]
defined by setting $\rho_n(y_i):=\x_i$ and $\rho_n(z_i):=\y_i$ for $1\leq
i \leq n$ and $\rho_n(y_i)=\rho_n(z_i)=0$ for $i>n$. Fulton \cite{Fu1,
  Fu3} showed that for $\om\in S_n$, the homomorphism $\rho_n$ sends
$\AS_\om(Y,Z)$ to a polynomial which represents the universal Schubert
class $[\X_\om]$ in the presentation (\ref{presentationA}). A different
way to establish this is obtained by arguing as in the proof of
Theorem \ref{gm} below.

\medskip
\noin {\bf Type C.} 
Here $G=\Sp_{2n}$ and we have a natural presentation
\begin{equation}
\label{presentationC}
\HH^*(BB\times_{BG} BB,\Z) \cong \Z[\XX_n,\YY_n]/\J_n,
\end{equation}
where $\J_n$ is the ideal generated by the differences
$e_i(\x_1^2,\ldots,\x_n^2)- e_i(\y_1^2,\ldots,\y_n^2)$ for $1\leq i
\leq n$. The inverse of the isomorphism (\ref{presentationC}) sends the
class of $\x_i$ to $-c_1(\V_{n+1-i}/\V_{n-i})$ and of $\y_i$ to
$-c_1(F_{n+1-i}/F_{n-i})$ for each $i$ with $1\leq i \leq n$.

Recall that $h_j(\YY_n)=s_j(\YY_n)$ denotes the $j$-th complete
symmetric polynomial, which is the sum of all monomials of total
degree $j$ in the variables $\YY_n$.  The {\em geometrization map}
is the ring homomorphism
\[
\pi_n : \Gamma[Y,Z] \to \Z[\XX_n,\YY_n]/\J_n
\]
determined by setting
\begin{gather*}
\pi_n(q_r(X)):=\sum_{i=0}^r e_i(\XX_n)h_{r-i}(\YY_n) \ \ \text{for all} \
r, \\ \pi_n(y_i):=\begin{cases} -\x_i & \text{if $1\leq i\leq n$}, \\
\ \ 0 & \text{if $i>n$}, \end{cases}
\ \ \ \text{and} \ \ \ \pi_n(z_j):= \begin{cases}
\y_j & \text{if $1\leq j\leq n$}, \\
0 & \text{if $j>n$}. \end{cases}
\end{gather*}
To show that $\pi_n$ is well defined, set
$\xi_r := \sum_{i=0}^r e_i(\XX_n)h_{r-i}(\YY_n)$ and observe that the 
generating function $\Xi(t)$ for the $\xi_r$ satisfies
\[
\Xi(t):=\sum_{r=0}^\infty \xi_rt^r = \left(\sum_{i=0}^\infty e_i(\XX_n)t^i\right)
 \left(\sum_{j=0}^\infty h_j(\YY_n)t^j\right) = \prod_{i=1}^n
\frac{1+\x_it}{1-\y_it}.
\]
Since the difference $\prod_i(1-\x^2_it^2)-
\prod_i(1-\y^2_it^2)$ lies in the ideal $\J_n[t^2]$, we have
\begin{equation}
\label{pfeq}
\Xi(t)\Xi(-t) =
\prod_{i=1}^n\frac{1-\x^2_it^2}{1-\y^2_it^2} \in 1+\J_n[[t^2]].
\end{equation}
Extracting the coefficient of $t^{2r}$ from both sides of equation
(\ref{pfeq}) gives
\begin{equation}
\label{pfeq2}
\xi_r^2 + 2\sum_{i=1}^r(-1)^i \xi_{r+i}\xi_{r-i} \in \J_n,  
\ \ \text{for all} \ \,  r > 0.
\end{equation}
The relations (\ref{pfeq2}), which agree with the $k=0$ case of
(\ref{crels}), generate the ideal of relations among the $q_r(X)$.
Therefore, $\pi_n$ is a well defined ring homomorphism.

\begin{thm}[Geometrization, \cite{IMN}]
\label{gm}
For any $w\in W_n$, the geometrization map $\pi_n$ sends
$\CS_w(X\,;Y,Z)$ to a polynomial which represents the universal
Schubert class $[\X_w]\in \HH^*(BB\times_{BG} BB,\Z)$ 
in the presentation {\em (\ref{presentationC})}. 
\end{thm}
\begin{proof}
The simple reflections $s_i$ in $W_\infty$ act on $\Z[\XX,\YY]$
as follows. The reflection $s_i$ interchanges $\x_i$ and $\x_{i+1}$
for $i>0$, while $s_0$ replaces $\x_1$ by $-\x_1$; all other variables
remain fixed. We have corresponding divided difference operators
$\partial^\x_i: \Z[\XX,\YY]\to \Z[\XX,\YY]$. For each $i\geq 0$ and 
$f\in \Z[\XX,\YY]$, they are defined by
\[
\partial_0^\x f := \frac{f-s_0f}{2\x_1}, \qquad
\partial_i^\x f := \frac{f-s_if}{\x_{i+1}-\x_i} \ \ \ \text{for $i>0$}.
\]
Let $\tilde{\omega}$ denote the involution of $\Z[\XX,\YY]$ obtained
by interchanging the variable $\x_j$ with $\y_j$ for all $j\geq
1$. Geometrically, this corresponds to the automorphism of
$BB\times_{BG} BB$ given by switching the two factors. Define the
$\y$-divided difference operators $\partial^\y_i$ on $\Z[\XX,\YY]$ by
$\partial^\y_i=\tilde{\omega}\partial^\x_i\tilde{\omega}$ for each
$i\geq 0$.

Restricting the above to $W_n$ and $0\leq i\leq n-1$, we obtain divided
differences $\partial^\x_i$, $\partial^\y_i$ acting on
$\Z[\XX_n,\YY_n]/\J_n$ and hence on $\HH^*(BB\times_{BG}BB,\Z)$.  In
the mid 1990s, building on the work of Bernstein-Gelfand-Gelfand
\cite{BGG} and Demazure \cite{D1, D2}, Fulton \cite{Fu2, Fu3} studied
the action of the operators $\partial^\x_i$ on the universal Schubert classes
$[\X_w]$, constructing them geometrically using correspondences by
$\bP^1$-bundles. It follows from this work that we have
\begin{equation}
\label{ddeqs2}
\partial_i^\x[\X_w] = \begin{cases}
[\X_{ws_i}] & \text{if $\ell(ws_i)<\ell(w)$}, \\ 
0 & \text{otherwise},
\end{cases}
\quad
\partial_i^\y[\X_w] = \begin{cases}
[\X_{s_iw}] & \text{if $\ell(s_iw)<\ell(w)$}, \\ 
0 & \text{otherwise},
\end{cases}
\end{equation}
for all $i$.

Write $BB_n$ and $BG_n$ for $BB$ and $BG$, respectively, to emphasize
the dependence on the rank $n$, and denote by $BM_n$ the Borel mixing
space $BB_n\times_{BG_n} BB_n$. The natural embedding of $W_n$ into
$W_{n+1}$ induces a morphism $\phi_n:BM_n \to BM_{n+1}$ and hence a
map of cohomology rings
\begin{equation}
\label{BBsys}
\phi_n^*:\HH^*(BM_{n+1},\Z)\to\HH^*(BM_n,\Z)
\end{equation}
which in terms of the presentation (\ref{presentationC}) is given by 
sending $\x_{n+1}$ and $\y_{n+1}$ to zero. Let 
\[
\IH(BM_\infty):=
\lim_{\longleftarrow}\HH^*(BM_n,\Z)
\]
be the stable cohomology ring of $BM_n$, which is the inverse limit in
the category of graded rings of the system of maps
$\{\phi_n^*\}_{n\geq 1}$ in (\ref{BBsys}).  For each $w\in W_\infty$,
we have a stable Schubert class $C_w$ in $\IH(BM_\infty)$, defined as
the element $\dis\lim_{\longleftarrow}[\X_w]$. It follows from the
equations (\ref{ddeqs2}) that
\begin{equation}
\label{ddeqs3}
\partial_i^\x C_w = \begin{cases}
C_{ws_i} & \text{if $\ell(ws_i)<\ell(w)$}, \\ 
0 & \text{otherwise},
\end{cases}
\quad
\partial_i^\y C_w = \begin{cases}
C_{s_iw} & \text{if $\ell(s_iw)<\ell(w)$}, \\ 
0 & \text{otherwise},
\end{cases}
\end{equation}
for all $i\geq 0$, while the degree zero component of $C_w$ 
is $1$ if $w=1$, and $0$ otherwise. Moreover, using the presentation
(\ref{presentationC}) and arguing as Theorem \ref{uniq}, it is easy 
to check that the family $\{C_w\}$ for $w\in W_\infty$ is uniquely
determined by these conditions.

The composite homomorphism
\[
\Gamma[Y,Z] \stackrel{\pi_n}\longrightarrow \Z[\XX_n,\YY_n]/\J_n
\stackrel{\simeq}\longrightarrow \HH^*(BM_n,\Z)
\]
is compatible with the maps $\phi^*_n$, therefore there is an induced 
ring homomorphism
\[
\pi_\infty:\Gamma[Y,Z]\to \IH(BM_\infty).
\]
One verifies that $\pi_\infty$ respects the actions of the divided
differences on its domain and codomain (compare with \cite{BH}). We
deduce the theorem, since both $\CS_w$ and $C_w$ are characterized by
the equations (\ref{ddeqs}) and (\ref{ddeqs3}), respectively, and the
same degree zero condition. One can show that, in fact, the map
$\pi_\infty$ is a canonical isomorphism of graded rings.
\end{proof}

For any $w\in W_n$, the image of $\CS_w$ under the geometrization map
$\pi_n$ may be computed as follows. Use (\ref{dbleC2}) and (\ref{FtoQ})
to write
\begin{equation}
\label{dbleCeq}
\CS_w(X\,;Y,Z) = 
\sum_{u,v,u',\la} e^v_\la \, \AS_{u^{-1}}(-Z)Q_\la(X) \AS_{u'}(Y),
\end{equation}
where the sum is over all reduced factorizations $uvu'=w$ and 
strict partitions $\la$ with $u,u'\in S_n$ and $|\la|=\ell(v)$. We 
then have 
\begin{equation}
\label{dbleC4}
\pi_n(\CS_w(X\,;Y,Z)) = \sum_{u,v,u',\la} e^v_\la \,
\AS_{u^{-1}}(-\YY_n)\wt{Q}_\la(\XX_n/\YY_n) \AS_{u'}(-\XX_n).
\end{equation}
We call the polynomial $\wt{Q}_\la(\XX_n/\YY_n)$ a {\em supersymmetric
  $\wt{Q}$-polynomial}; it is obtained from the $Q$-polynomial
$Q_\la(c)$ in (\ref{Qdef}) by specializing $c_r$ to $\sum_i
e_i(\XX_n)h_{r-i}(\YY_n)$ for each $r$. The supersymmetric
$\wt{Q}$-polynomials have properties directly analogous to those of
the $\wt{Q}$-polynomials $\wt{Q}_\la(\XX_n)$ of Pragacz and Ratajski
\cite{PR2}, which satisfy the identity 
\[
\wt{Q}_\la(\XX_n) = \wt{Q}_\la(\XX_n/\YY_n)\vert_{\YY_n=0}. 
\]
Moreover, setting $Z=0$ and
$\YY_n=0$ in (\ref{dbleC4}), one recovers the {\em symplectic Schubert
 polynomials} $\CS_w(\XX_n)$ of \cite{T2}.

\begin{remark}
\label{rmksix}
Let $\fraka$ be a sequence of integers as in \S \ref{sdl} and let
$P\subset G$ be the parabolic subgroup corresponding to
$\fraka$. Suppose further that we are concerned only with degeneracy
loci having the symmetries of $P$, or equivalently, with Giambelli
formulas in $\HH^*_T(G/P)$.  Let $k=a_1$ and $\Gamma^{(k)} :=
\Z[\ti_1,\ti_2,\ldots]$ denote the associated ring of theta
polynomials. Then it suffices to work with polynomials in the
ring $$\Gamma^{(k)}[y_{k+1},y_{k+2},\ldots,z_1,z_2,\ldots],$$ as these
elements map under $\pi_n$ to the cohomology classes that lie in
$\HH^*_T(G/P)$. This follows from equation (\ref{dbleC2}) and
\cite[Prop.\ 5]{T6}.
\end{remark}

\medskip
\noin {\bf Type D.} Here $G=\SO_{2n}$ and according to (\ref{natiso})
the ring $\HH^*(BB\times_{BG} BB,\Q)$ is presented as a quotient
\begin{equation}
\label{presentationD}
\HH^*(BB\times_{BG} BB,\Q) \cong \Q[\XX_n,\YY_n]/\J'_n,
\end{equation}
where $\J'_n$ is the ideal generated by the differences
$e_i(\x_1^2,\ldots,\x_n^2)- e_i(\y_1^2,\ldots,\y_n^2)$ for $1\leq i
\leq n-1$ and $\x_1\cdots\x_n- \y_1\cdots\y_n$. 
The inverse of the isomorphism (\ref{presentationD}) sends the
class of $\x_i$ to $-c_1(\V_{n+1-i}/\V_{n-i})$ and of $\y_i$ to
$-c_1(F_{n+1-i}/F_{n-i})$ for each $i$ with $1\leq i \leq n$.
The {\em geometrization map} is the ring homomorphism 
\[
\pi'_n : \Gamma'[Y,Z] \to \Q[\XX_n,\YY_n]/\J'_n
\] 
defined by setting
\begin{gather*}
\pi'_n(P_r(X)):=\frac{1}{2}\sum_{i=0}^r e_i(\XX_n)h_{r-i}(\YY_n)
\ \ \text{for all} \, r, \\ \pi'_n(y_i):=\begin{cases} -\x_i & \text{if
  $1\leq i\leq n$}, \\ \ \ 0 & \text{if $i>n$}, \end{cases}
\ \ \ \text{and} \ \ \ \pi'_n(z_j):= \begin{cases} \y_j & \text{if
    $1\leq j\leq n$}, \\ 0 & \text{if $j>n$}. \end{cases}
\end{gather*}
One can now prove a type D version of Theorem \ref{gm}: for any $w\in
\wt{W}_n$, the map $\pi'_n$ sends $\DS_w(X\,;Y,Z)$ to a polynomial
which represents the universal Schubert class $[\X_w]$ in the
presentation (\ref{presentationD}). Moreover, the system of maps
$\pi'_n$ induces a graded ring isomorphism $\pi'_\infty$ between
$\Gamma'[Y,Z]$ and the stable cohomology ring of the Borel mixing
space associated to $\SO_{2n}(\C)$. In the single (non-equivariant)
case when $Z=0$ and $\YY_n=0$, we observe that $\pi'_n(\DS_w(X\,;Y))$
is equal to the {\em orthogonal Schubert polynomial} $\DS_w(\XX_n)$ of
\cite{T3}.

\subsection{Proof of Theorems \ref{dbleAloci}--\ref{dbleDloci}}

We will give our proof of Theorem \ref{dbleDloci} below; the arguments
in types A--C are simpler, and the reader may also consult the
references \cite[\S 4.1]{BKTY1} (for Theorem \ref{dbleAloci}) and
\cite[Thm.\ 3]{T6} (for Theorems \ref{dbleCloci} and \ref{dbleBloci}).
All of the proofs rely on the {\em splitting principle}, which allows
one to express the Chern classes of the relevant vector bundles as
elementary symmetric polynomials in their Chern roots; for more
details on this, see for example \cite[\S 3.2]{Fu5} and \cite[\S
  3.5]{Ma}.

The variables $\x_i$ and $\y_i$ for $1\leq i \leq n$ give the Chern
roots of the various vector bundles over $BB\times_{BG}BB$, which pull
back to give the roots of the corresponding vector bundles over
$M$. In particular the Chern roots of $Q_1$ are (the pullbacks of)
$\x_1,\ldots,\x_n,-\x_1,\ldots,-\x_{a_1}$, while those of $Q_r$ for
$r\geq 2$ are $-\x_{a_{r-1}+1},\ldots,-\x_{a_r}$. Similarly the Chern
roots of $F_{n+1-r}$ are represented by $-\y_r,\ldots,-\y_n$ for each
$r$. With $k=a_1$, we have
$$\ti_r(X\,;Y_1) = \sum_{i=0}^r q_{r-i}(X)e_i(y_1,\ldots,y_{a_1})$$
for each $r\geq 0$. It follows that
\begin{align*}
\pi'_n(\ti_r(X\,;Y_1)) &= \sum_{i,j\geq 0} e_{r-i-j}(\XX_n) h_j(\YY_n)
e_i(-\x_1,\ldots,-\x_{a_1}) \\ 
& = \sum_{j\geq 0} e_{r-j}(\x_1,\ldots,\x_n, -\x_1,\ldots,-\x_{a_1}) h_j(\YY_n) \\
& = c_r(Q_1-F_n) = c_r(E-E_1-F_n).
\end{align*}
Moreover, we have $$\pi'_n(e_{a_1}(Y_1)) =
(-1)^{a_1}\x_1\cdots\x_{a_1} = c_{a_1}(E_0-E_1),$$ where $E_0$ denotes
the maximal isotropic subbundle of $E$. We deduce from the equations
(\ref{etaequ}) that
\begin{align*}
\pi'_n(\eta_{a_1}(X\,;Y_1)) &= \frac{1}{2}(c_{a_1}(E-E_1-F_n)+c_{a_1}(E_0-E_1))\quad
\text{and} \\
\pi'_n(\eta'_{a_1}(X\,;Y_1)) &= \frac{1}{2}(c_{a_1}(E-E_1-F_n)-c_{a_1}(E_0-E_1)).
\end{align*}
Furthermore, for any partition $\mu$ and $r\geq 2$, we have
\[
\pi'_n(s_\mu(Y_r)) = s_\mu(-\x_{a_{r-1}+1},\ldots,-\x_{a_r}) = 
s_\mu(Q_r) = s_\mu(E_{r-1}-E_r),
\] 
while
\[
\pi'_n(s_\mu(0/Z_r)) = s_{\wt{\mu}}(-\y_{b_{r-1}+1},\ldots,-\y_{b_r}) =
s_{\wt{\mu}}(\wh{Q}_r) = s_\mu(F_{n+b_{r-1}}-F_{n+b_r}).
\]
The equality in Theorem \ref{dbleDloci} is therefore obtained by applying
$\pi'_n$ to formula (\ref{dbDSsplitting}).

\section{Suggestions for future research.}
\label{fut}

In this section we propose some natural directions to follow in future 
work.

\subsection{Quantum cohomology of $G/P$} 

The last two decades have seen much exploration of the Gromov-Witten
theory and quantum cohomology rings of homogeneous spaces. In
particular, one seeks to extend the classical understanding of
Schubert calculus to the quantum setting, with analogues of the
theorems of Pieri, Giambelli, and computations of Schubert structure
constants for $\QH^*(G/P)$.  One of the motivations for \cite{T6} was
the fact that the known classical Giambelli formulas expressed in
terms of the special Schubert classes in this article have
straightforward extensions to the small quantum cohomology ring of
$G/P$; see \cite{Bertram, FGP, CF, KTlg, KTog, BKT3, BKT4}. We expect
that there should be quantum (double) Schubert polynomials and
(equivariant) Giambelli formulas in the Lie types B, C, and D which
restrict to the results found in the present paper when the quantum
parameters are set equal to zero.

\subsection{$K$-theory of $G/P$}

The classes of the structure sheaves $\cO_{X_w}$ of the Schubert
varieties $X_w$, $w\in W^P$ provide a natural $\Z$-basis for the
Grothendieck group $K(G/P)$ of vector bundles on $G/P$. When $G$ is
the general linear group $\GL_n$, one has the Lascoux-Sch\"utzenberger
theory of Grothendieck polynomials \cite{LS2, FLa}, and many of the
type A results for cohomology can be generalized to $K$-theory
\cite{La2, Bu, BKTY2}. Even for the Grassmannian $\G(m,n)$, however,
one does not yet have a Giambelli formula for $[\cO_{X_\la}]$ in terms
of special Schubert classes that clearly extends the raising operator
expression (\ref{giambelli2}) in \S \ref{class} (see
\cite[Thm.\ 1]{Bu} for a Jacobi-Trudi recursion). There should be
versions of the double Grothendieck polynomials for the other
classical Lie types, defined using the degenerate Hecke algebra in
place of the nilCoxeter algebra; the type A theory is worked out in
\cite{FK0}. One could then seek analogues of Lascoux's transition
equations \cite{La2} and the splitting formula of
\cite[Thm.\ 4]{BKTY2} for these objects. Some recent related work on
equivariant $K$-theory in the other types may be found in \cite{IN,
  GK}; see also \cite{KK, GR}.

\subsection{Combinatorial questions}

There is a large body of research on the combinatorial aspects of the
Schubert calculus, most all of it in the situation where the
underlying Weyl group elements are {\em fully commutative} in the
sense of \cite{St2}. By contrast, the combinatorial theory developed
in \cite{BKT1, BKT2, BKT4, T4, T6, T7} is in its infancy, with many
questions worth exploring further. In particular, we would like a
deeper understanding of the connections between the formulas that
appear in special cases, such as in the works cited in the
introduction, which differ from those given here (see \cite{TW} 
for recent progress on this).

We mentioned in \S \ref{tts} that the type A Stanley coefficients
$c^\om_\la$ enumerate Young tableaux as well as leaves in transition
trees. However the mixed Stanley coefficients $e^w_{\la}$ and
$d^w_{\la}$ are only known to be positive through transition when
$k>0$. Are there alternative combinatorial formulas for them, and is
there a geometric proof of their positivity? In type A, Little
\cite{Li} has studied the combinatorics of the
Lascoux-Sch\"utzenberger tree; is there an analogue of his bijection
for the $c^\om_\la$ which involves the $0$-transition trees of
\cite{Bi1} and Kra\'skiewicz-Lam tableaux \cite{Kr,
  L1}?\,\footnote{Such a bijection in Lie type B has recently been
  defined in \cite{BHRY}.}  Do the mixed Stanley coefficients include
a rule which computes the Schubert structure constants on isotropic
Grassmannians? Some partial results on this last question are obtained
in \cite[\S 2.4]{T6}.

\subsection{A theory for the exceptional types}

It natural to ask whether the uniform choice of special Schubert
classes on the Grassmannians for the classical groups from \cite{BKT1}
and \S \ref{class}-\S \ref{isogiam} extends to the exceptional Lie
types, and to look for canonical Giambelli expressions native to $G/P$
for arbitrary reductive groups $G$ and $P$. Anderson \cite{A} has
obtained degeneracy locus formulas for vector bundles with structure
group $\text{G}_2$, and there are parallels between the combinatorics
of Schubert calculus on the Grassmannians $\G(m,n)$ and the hermitian
symmetric (or cominuscule) quotients of $\text{E}_6$ and $\text{E}_7$
\cite{TY}. More remains to be understood however to obtain a
generalization of the results discussed in this article to any $G/P$
space.

\subsection{Connections with representation theory}

Do the raising operator formulas of this paper appear in other areas,
and in particular in representation theory? The Schur $S$- and
$Q$-polynomials were studied by Schur (and his advisor Frobenius) in
order to compute the characters and projective characters of the
symmetric group and the polynomial characters of the general linear
group. The discovery of formulas such as (\ref{JT}) and
(\ref{giamQ0}), (\ref{giamQ}) in the algebra of symmetric functions
and associated group representations long preceded their realizations
in the Schubert calculus. A representation-theoretic understanding of
the raising operator expressions (\ref{Tidefn}) and (\ref{Heq}),
parallel to the extensive theory of Schur polynomials, would certainly
be desirable.  Young's raising operators are well known in algebraic
combinatorics, but we suspect that their full potential has not yet
been exploited.  For a non-exhaustive list of references which apply
them in various settings, see \cite{BKT2, BKT3, BKT4, DLT, Ga, HH,
  LaS, Le, Lit, M, Mo, Ro, T4, T5, T7, TW, To, Y}.


\begin{thebibliography}{BKT4}


\bibitem[A]{A} D. Anderson : 
{\em Chern class formulas for $G_2$ Schubert loci}, 
Trans. Amer.  Math.  Soc.  {\bf 363} (2011), 6615--6646.


\bibitem[AF]{AF} D. Anderson and W. Fulton : 
{\em Degeneracy loci, Pfaffians, and vexillary signed permutations 
in types B, C, and D}, arXiv:1210.2066.


\bibitem[Ar]{Ar} A. Arabia : 
{\em Cycles de Schubert et cohomologie \'equivariante de $K/T$},
Invent. Math. {\bf 85} (1986), 39--52. 


\bibitem[Bel]{Be} P. Belkale : 
{\em Invariant theory of GL($n$) and intersection theory of Grassmannians},
Int. Math. Res. Not. {\bf 2004}, no. 69, 3709--3721. 

\bibitem[BK]{BK} P. Belkale and S. Kumar :
{\em Eigenvalue problem and a new product in cohomology of flag varieties},
Invent. Math. {\bf 166} (2006), 185--228. 


\bibitem[BS]{BS} N.  Bergeron and F.  Sottile :
{\em A Pieri-type formula for isotropic flag manifolds}, Trans.
Amer.  Math.  Soc.  {\bf 354} (2002), 4815--4829.


\bibitem[BGG]{BGG} I. N. Bernstein, I. M. Gelfand and S. I. Gelfand :
{\em Schubert cells and cohomology of the spaces $G/P$}, Russian
Math. Surveys {\bf 28} (1973), 1--26.



\bibitem[Be]{Bertram} A. Bertram :
{\em Quantum Schubert calculus},
Adv. Math. {\bf 128} (1997), 289--305.


\bibitem[Bi1]{Bi1} S. Billey : 
{\em Transition equations for isotropic flag manifolds}, 
Discrete Math. {\bf 193} (1998), 69--84.

\bibitem[Bi2]{Bi2} S. Billey :
{\em Kostant polynomials and the cohomology ring for $G/B$},
Duke Math. J. {\bf 96} (1999), 205--224.


\bibitem[BH]{BH} S. Billey and M. Haiman :
{\em Schubert polynomials for the classical groups},
J. Amer. Math. Soc. {\bf 8} (1995), 443--482.


\bibitem[BJS]{BJS} S. Billey, W. Jockusch and R. P. Stanley :
{\em Some combinatorial properties of Schubert polynomials},
J. Algebraic Combin. {\bf 2} (1993), 345--374.


\bibitem[BHRY]{BHRY} S. Billey, Z. Hamaker, A. Roberts and B. Young :
{\em Coxeter-Knuth graphs and a signed Little map for type B reduced words}, 
Electron. J. Combin. {\bf 21} (2014), Paper 4.6, 39 pp.


\bibitem[Bo]{Bo} A. Borel : {\em Sur la cohomologie des espaces fibr\'{e}s
  principaux et des espaces homog\`{e}nes de groupes de Lie compacts},
  Ann. of Math. {\bf 57} (1953), 115--207.


\bibitem[Br1]{Br} M. Brion : {\em Equivariant Chow groups for torus
actions}, Transformation Groups {\bf 2} (1997), 225--267.

\bibitem[Br2]{Br2} M. Brion : 
{\em Equivariant cohomology and equivariant intersection theory},
Notes by Alvaro Rittatore, NATO Adv. Sci. Inst. Ser. C Math. Phys. Sci., 
514, Representation theories and algebraic geometry (Montreal, PQ, 1997), 
1--37, Kluwer Acad. Publ., Dordrecht, 1998. 

\bibitem[Bu]{Bu} A.~S. Buch :
{\em Grothendieck classes of quiver varieties}, Duke Math. J.
{\bf 115} (2002), 75--103.

\bibitem[BF]{BF} 
A.~S. Buch and W. Fulton : 
{\em  Chern class formulas for quiver varieties},
Invent. Math. {\bf 135} (1999), 665--687.

\bibitem[BKT1]{BKT1} A. S. Buch, A. Kresch and H. Tamvakis :
{\em Quantum Pieri rules for isotropic Grassmannians},
Invent. Math. {\bf 178} (2009), 345--405.

\bibitem[BKT2]{BKT2} A. S. Buch, A. Kresch and H. Tamvakis :
{\em A Giambelli formula for isotropic Grassmannians},
Selecta Math. (N.S.), to appear.

\bibitem[BKT3]{BKT3} A. S. Buch, A. Kresch and H. Tamvakis :
{\em Quantum  Giambelli formulas for isotropic Grassmannians},
Math. Annalen {\bf 354} (2012), 801--812.

\bibitem[BKT4]{BKT4} A. S. Buch, A. Kresch and H. Tamvakis :
{\em A Giambelli formula for even orthogonal Grassmannians},
J. reine angew. Math. {\bf 708} (2015), 17--48.

\bibitem[BKTY1]{BKTY1} 
A.~S. Buch, A.~Kresch, H.~Tamvakis, and A.~Yong : 
{\em Schubert polynomials and quiver formulas}, Duke Math. J.
\textbf{122} (2004), 125--143.

\bibitem[BKTY2]{BKTY2}
A.~S. Buch, A.~Kresch, H.~Tamvakis, and A.~Yong : 
{\em Grothendieck polynomials and quiver formulas},
Amer. J. Math. \textbf{127} (2005), 551--567.


\bibitem[C1]{Ca1} \'E. Cartan : {\em Sur un classe remarquable d'espaces de
  Riemann}, Bull. Soc. Math. France {\bf 54} (1926), 214--264; {\bf
  55} (1927), 114--134; also, {\em Oeuvres Compl\`etes}, Partie I,
  vol. 2, 587--659, Gauthier-Villars, Paris, 1952.


\bibitem[C2]{Ca2} \'E. Cartan : 
{\em Sur les invariants int\'egraux de certains espaces homog\`enes clos et les
propri\'et\'es topologiques de ces espaces}, Annales de la Soci\'et\'e 
Polonaise de Math\'ematique {\bf 8} (1929), 181--225; also,  {\em Oeuvres 
Compl\`etes}, Partie I, vol. 2, 1081--1125, Gauthier-Villars, Paris, 1952.

\bibitem[CF]{CF} I. Ciocan-Fontanine :
{\em On quantum cohomology rings of partial flag
varieties}, Duke Math. J. {\bf 98} (1999), 485--524.


\bibitem[DP]{DP} C. De Concini and P. Pragacz : 
{\em On the class of Brill-Noether for Prym varieties}, 
Math. Ann. {\bf 302} (1995), 687--697.

\bibitem[D1]{D1} M. Demazure :
{\em Invariants sym\'{e}triques des groupes de Weyl et torsion},
Invent. Math. {\bf 21} (1973), 287--301.

\bibitem[D2]{D2} M. Demazure :
{\em D\'esingularisation des vari\'et\'es de Schubert 
g\'en\'eralis\'ees},  Ann. Sci. \'Ecole Norm. Sup. (4)
{\bf 7} (1974), 53--88.

\bibitem[DLT]{DLT} J. D\'esarm\'enien, B. Leclerc and J.-Y. Thibon :
{\em Hall-Littlewood functions and Kostka-Foulkes polynomials in representation
theory}, S\'em.\ Lothar.\ Combin.\ {\bf 32} (1994), Art. B32c, 38 pp.


\bibitem[EG]{EG} M. Edelman and C. Greene :
{\em Balanced tableaux}, Adv. Math. {\bf 63} (1987), 42--99.

\bibitem[EGr]{EGr} D. Edidin and W. Graham :
{\em Characteristic classes in the Chow ring}, 
 J. Algebraic Geom. {\bf 6} (1997), 431--443. 

\bibitem[E]{Eh} C. Ehresmann :
{\em Sur la topologie de certains espaces homog\`enes},
Ann. of Math. (2) {\bf 35} (1934), 396--443. 

\bibitem[EvG]{EvG} T. Ekedahl and G. van der Geer :
{\em Cycle classes of the E-O stratification of the moduli of 
abelian varieties}, in `Algebra, arithmetic, and geometry: in honor 
of Yu. I. Manin. Vol. I',  567--636, Progr. Math., 269, 
Birkh\"auser Boston, Inc., Boston, MA, 2009. 

\bibitem[FR1]{FR} L. Feh\'er and R. Rim\'anyi :
{\em Schur and Schubert polynomials as Thom polynomials -- 
cohomology of moduli spaces},  
Cent. Eur. J. Math. {\bf 1} (2003), 418--434.

\bibitem[FR2]{FR2} L. Feh\'er and R. Rim\'anyi :
{\em Calculation of Thom polynomials and other cohomological 
obstructions for group actions}, Real and complex singularities, 69--93,
Contemp. Math. {\bf 354}, Amer. Math. Soc., Providence, RI, 2004. 

\bibitem[FGP]{FGP} S. Fomin, S. Gelfand and A. Postnikov :
{\em Quantum Schubert polynomials}, J. Amer. Math. Soc. {\bf 10}
(1997), 565--596.


\bibitem[FG]{FG} S. Fomin and C. Greene :
{\em Noncommutative Schur functions and their applications},
Discrete Math. {\bf 193} (1998), 179--200.



\bibitem[FK1]{FK0} S.~Fomin and A.~N. Kirillov : {\em Grothendieck
  polynomials and the Yang-Baxter equation}, Proceedings of the 6th
  Intern.  Conference on Formal Power Series and Alg. Comb., DIMACS
  (1994), 183--190.


\bibitem[FK2]{FK1} S.~Fomin and A.~N.~Kirillov : 
{\em The Yang-Baxter equation, symmetric functions, and Schubert 
polynomials}, Discrete Math. {\bf 153} (1996), 123--143.

\bibitem[FK3]{FK2} S. Fomin and A. N. Kirillov :
{\em Combinatorial $B_n$-analogs of Schubert polynomials},
Trans. Amer. Math. Soc. {\bf 348} (1996), 3591--3620.


\bibitem[FS]{FS} S. Fomin and R. P. Stanley :
{\em Schubert polynomials and the nil-Coxeter algebra}, 
Adv. Math. {\bf 103} (1994), 196--207.


\bibitem[F1]{Fu1} W. Fulton : 
{\em Flags, Schubert polynomials, degeneracy loci, and determinantal formulas},
Duke Math. J. {\bf 65} (1992), 381--420.

\bibitem[F2]{Fu2} W. Fulton :
{\em Schubert varieties in flag bundles for the classical groups}, 
Proceedings of the Hirzebruch 65 Conference on
Algebraic Geometry (Ramat Gan, 1993), 241--262,
Israel Math. Conf. Proc. {\bf 9}, Ramat Gan, 1996.

\bibitem[F3]{Fu3} W. Fulton :
{\em Determinantal formulas for orthogonal and symplectic degeneracy
loci}, J. Differential Geom. {\bf 43} (1996), 276--290.


\bibitem[F4]{Fu4} W. Fulton :
{\em Young tableaux}, London Mathematical Society Student Texts {\bf 35},
Cambridge University Press, Cambridge, 1997.


\bibitem[F5]{Fu5} W. Fulton :
{\em Intersection theory}, Second edition, 
Ergebnisse der Math. {\bf 2}, Springer-Verlag, Berlin, 1998.


\bibitem[FLa]{FLa} W.~Fulton and A.~Lascoux : 
{\em A Pieri formula in the Grothendieck ring of a flag bundle}, 
Duke Math. J. {\bf 76} (1994), 711--729.


\bibitem[FL]{FL} W. Fulton and R. Lazarsfeld : 
{\em On the connectedness of degeneracy loci and special divisors},
Acta Math. {\bf 146} (1981), 271--283. 

\bibitem[FP]{FP} W. Fulton and P. Pragacz :
{\em Schubert varieties and degeneracy loci}, 
Lecture Notes in Math. {\bf 1689}, Springer-Verlag, Berlin, 1998.

\bibitem[Ga]{Ga} A. Garsia :
{\em Orthogonality of Milne's polynomials and raising operators}, 
Discrete Math. {\bf 99} (1992), 247--264. 


\bibitem[G1]{G1} G. Z. Giambelli :
{\em Risoluzione del problema degli spazi secanti}, Mem. R. Accad. Sci.
Torino (2) {\bf 52} (1902), 171--211.

\bibitem[G2]{G2} G. Z. Giambelli :
{\em Alcune propriet\`a delle funzioni simmetriche  caratteristiche},
Atti Torino {\bf 38} (1903), 823--844.


\bibitem[Gr]{Gra} W. Graham :
{\em The class of the diagonal in flag bundles},  J. Differential Geom.  
{\bf 45} (1997), 471--487.

\bibitem[GK]{GK} W. Graham and V. Kreiman :
{\em Excited Young diagrams, equivariant $K$-theory, and Schubert 
varieties}, Trans. Amer. Math. Soc. {\bf 367} (2015), 6597--6645.

\bibitem[GR]{GR} S. Griffeth and A. Ram : 
{\em Affine Hecke algebras and the Schubert calculus}, 
European J. Combin. {\bf 25} (2004), 1263--1283. 

\bibitem[H]{H} M. D. Haiman : 
{\em Dual equivalence with applications, including a conjecture of Proctor},
Discrete Math. {\bf 99} (1992), 79--113. 

\bibitem[HT]{HT} J. Harris and L. W. Tu :
{\em On symmetric and skew-symmetric determinantal varieties},
Topology {\bf 23} (1984), 71--84.

\bibitem[Hi]{Hi} H. Hiller : 
{\em Geometry of Coxeter groups},
Research Notes in Mathematics, 54. Pitman (Advanced Publishing Program), 
Boston, Mass.-London, 1982. 

\bibitem[HB]{HB} H. Hiller and B. Boe :
{\em Pieri formula for $SO_{2n+1}/U_n$ and $Sp_n/U_n$}, 
Adv. Math. {\bf 62} (1986), 49--67.

\bibitem[HH]{HH} P. N. Hoffman and J. F. Humphreys : {\em Projective
  representations of the symmetric groups}, Oxford Math. Monographs,
  The Claredon Press, Oxford University Press, New York, 1992.

\bibitem[Hs]{Hs} W.-Y. Hsiang : 
{\em Cohomology theory of topological transformation groups},
Ergebnisse der Mathematik und ihrer Grenzgebiete, Band 85. 
Springer-Verlag, New York-Heidelberg, 1975. 

\bibitem[I]{Ik} T. Ikeda : 
{\em Schubert classes in the equivariant cohomology of the Lagrangian 
Grassmannian}, Adv. Math. {\bf 215} (2007), 1--23. 


\bibitem[IM]{IM} T. Ikeda and T. Matsumura :
{\em Pfaffian sum formula for the symplectic Grassmannians}, 
Math. Z. {\bf 280} (2015), 269--306.


\bibitem[IMN]{IMN} T. Ikeda, L. C. Mihalcea, and H. Naruse :
{\em Double Schubert polynomials for the classical groups}, 
Adv. Math. {\bf 226} (2011), 840--886.


\bibitem[IN1]{IN0} T. Ikeda and H. Naruse :
{\em Excited Young diagrams and equivariant Schubert calculus},
Trans. Amer. Math. Soc. {\bf 361} (2009), 5193--5221. 


\bibitem[IN2]{IN} T. Ikeda and H. Naruse :
{\em $K$-theoretic analogues of factorial Schur $P$- and 
$Q$-functions}, Adv. Math. {\bf 243} (2013), 22--66.


\bibitem[Iv]{I} V. N. Ivanov : 
{\em Interpolation analogues of Schur $Q$-functions}, Zap. Nauchn. 
Sem. S.-Peterburg. Otdel. Mat. Inst. Steklov. (POMI) {\bf 307} (2004), 
Teor. Predst. Din. Sist. Komb. i Algoritm. Metody. 10, 99--119, 
281--282;  translation in J. Math. Sci. (N. Y.) {\bf 131} (2005),
5495--5507 .


\bibitem[J]{J} C. G. J. Jacobi :
{\em De functionibus alternantibus earumque divisione per productum 
e differentiis elementorum conflatum}, J. reine angew. Math. {\bf 22} 
(1841), 360--371. Reprinted in Gesammelte Werke {\bf 3}, 439--452,
Chelsea, New York, 1969.


\bibitem[Jo]{Jo} T. J\'ozefiak :
{\em Characters of projective representations of symmetric groups},
Exposition. Math. {\bf 7} (1989), 193--247.


\bibitem[JLP]{JLP} T. J\'ozefiak, A. Lascoux and P. Pragacz :
{\em Classes of determinantal varieties associated with
symmetric and skew-symmetric matrices}, Math USSR Izvestija 
{\bf 18} (1982), 575--586.


\bibitem[Ka1]{Ka1} M. Kazarian :
{\em On Lagrange and symmetric degeneracy loci}, preprint, 
Arnold Seminar (2000); available at 
http://www.newton.ac.uk/preprints/NI00028.pdf.

\bibitem[Ka2]{Ka} M. Kazarian :
{\em Thom polynomials}, Singularity theory and its applications, 85--135,
Adv. Stud. Pure Math., 43, Math. Soc. Japan, Tokyo, 2006. 


\bibitem[KL]{KL} G. Kempf and D. Laksov :
{\em The determinantal formula of Schubert calculus},
Acta Math. {\bf 132} (1974), 153--162.


\bibitem[KK]{KK} B. Kostant and S. Kumar : 
{\em $T$-equivariant $K$-theory of generalized flag varieties},
J. Differential Geom. {\bf 32} (1990), 549--603. 


\bibitem[Kr]{Kr} W. Kra\'skiewicz :
{\em Reduced decompositions in hyperoctahedral groups},
C. R. Acad. Sci. Paris S\'er. I Math. {\bf 309} (1989), 903--907.


\bibitem[Kre]{Kre} V. Kreiman : 
{\em Products of factorial Schur functions}, 
Electron. J. Combin. {\bf 15} (2008), Research Paper 84, 12 pp.



\bibitem[KT1]{KT} A. Kresch and H. Tamvakis :
{\em Double Schubert polynomials and degeneracy loci for the
classical groups}, Ann. Inst. Fourier {\bf 52} (2002), 1681--1727.


\bibitem[KT2]{KTlg} A. Kresch and H. Tamvakis :
{\em Quantum cohomology of the Lagrangian Grassmannian},
J. Algebraic Geom. {\bf 12} (2003), 777--810.


\bibitem[KT3]{KTog} A. Kresch and H. Tamvakis :
{\em Quantum cohomology of orthogonal Grassmannians},
Compos. Math. {\bf 140} (2004), 482--500.


\bibitem[LRS]{LaSa} V. Lakshmibai, K. N. Raghavan, and P. Sankaran, P. :
{\em Equivariant Giambelli and determinantal restriction formulas for 
the Grassmannian}, Pure Appl. Math. Q. {\bf 2} (2006), Special Issue: 
In honor of Robert D. MacPherson. Part 1, 699--717. 

\bibitem[La1]{L1} T. K. Lam : {\em B and D analogues of stable Schubert
polynomials and related insertion algorithms}, Ph.D.\ thesis, M.I.T., 1994; 
available at http://hdl.handle.net/1721.1/36537.



\bibitem[La2]{L2} T. K. Lam : {\em $B\sb n$ Stanley symmetric
functions}, Discrete Math. {\bf 157}  (1996),  241--270. 


\bibitem[L1]{La0} A. Lascoux : 
{\em Puissances ext\'erieures, d\'eterminants et cycles de Schubert},
Bull. Soc. Math. France {\bf 102} (1974), 161--179. 

\bibitem[L2]{La1} A. Lascoux : 
{\em Classes de Chern des vari\'et\'es de drapeaux}, 
C. R. Acad. Sci. Paris S\'er. I Math. {\bf 295} (1982), 393--398.

\bibitem[L3]{La2} A. Lascoux :
{\em Transition on Grothendieck polynomials}, Physics and Combinatorics, 
2000 (Nagoya), 164--179, World Sci. Publ., River Edge, NJ, 2001. 

\bibitem[L4]{La3} A. Lascoux : 
{\em Symmetric functions and combinatorial operators on polynomials},
CBMS Regional Conference Series in Mathematics {\bf 99},
American Mathematical Society, Providence, RI, 2003.


\bibitem[LP1]{LP1} A. Lascoux and P. Pragacz :
{\em Operator calculus for $\wt{Q}$-polynomials and
Schubert polynomials}, Adv. Math. {\bf 140} (1998), 1--43.


\bibitem[LP2]{LP2} A. Lascoux and P. Pragacz :
{\em Orthogonal divided differences and Schubert polynomials,
$\wt{P}$-functions, and vertex operators},
Michigan Math. J. {\bf 48} (2000), 417--441.


\bibitem[LS1]{LS1} A. Lascoux and M.-P. Sch\"{u}tzenberger :
{\em Polyn\^{o}mes de Schubert}, C. R. Acad. Sci. Paris S\'er. I
Math. {\bf 294} (1982), 447--450.


\bibitem[LS2]{LS2} A. Lascoux and M.-P. Sch\"{u}tzenberger :
{\em  Structure de Hopf de l'anneau de cohomologie et de
l'anneau de Grothendieck d'une vari\'et\'e de drapeaux},
C. R. Acad. Sci. Paris S\'er. I
Math. {\bf 295} (1982),  629--633.


\bibitem[LS3]{LS3} A. Lascoux and M.-P. Sch\"{u}tzenberger :
{\em Schubert polynomials and the Littlewood-Richardson rule}, 
Lett. Math. Phys. {\bf 10} (1985), 111--124.


\bibitem[LaSc]{LaS} M. Lassalle and M. Schlosser :
{\em Inversion of the Pieri formula for Macdonald polynomials}, 
Adv. Math. {\bf 202} (2006), 289--325. 


\bibitem[Le]{Le} C. Lecouvey :
{\em A duality between q-multiplicities in tensor products and 
q-multiplicities of weights for the root systems B, C or D}, 
J. Combin. Theory Ser. A {\bf 113} (2006), 739--761. 


\bibitem[Lit]{Li} D. P. Little :
{\em Combinatorial aspects of the Lascoux-Sch\"utzenberger tree}, 
Adv. Math. {\bf 174} (2003), 236--253.


\bibitem[Li]{Lit} D. E. Littlewood :
{\em On certain symmetric functions},
Proc. London Math. Soc. (3) {\bf 11} (1961), 485--498.


\bibitem[M]{M} I. G. Macdonald :
{\em Symmetric functions and Hall polynomials}, Second edition,
The Clarendon Press, Oxford University Press, New York, 1995.


\bibitem[Ma]{Ma} L. Manivel :
{\em Fonctions sym\'etriques, polyn\^omes de Schubert et 
lieux de d\'eg\'en\'erescence},  Cours Sp\'ecialis\'es {\bf 3},
Soci\'et\'e Math\'ematique de France, Paris, 1998. 

\bibitem[Mi]{Mi} L. C. Mihalcea : 
{\em Giambelli formulae for the equivariant quantum cohomology of the 
Grassmannian}, Trans. Amer. Math. Soc. {\bf 360} (2008), 2285--2301.


\bibitem[Mo]{Mo} A. O. Morris :
{\em The characters of the group $GL(n,q)$}, 
Math. Z. {\bf 81} (1963), 112--123. 

\bibitem[PP]{PP} A. Parusi\'nski and P. Pragacz :
{\em Chern-Schwartz-MacPherson classes and the Euler characteristic of 
degeneracy loci and special divisors},
J. Amer. Math. Soc. {\bf 8} (1995), 793--817. 


\bibitem[Pi]{Pi} M. Pieri :
{\em Sul problema degli spazi secanti. Nota 1${}^a$}, Rend. Ist. Lombardo
(2) {\bf 26} (1893), 534--546.

\bibitem[Po]{Po} I. R. Porteous : 
{\em Simple singularities of maps}, Proceedings of Liverpool Singularities 
Symposium, I (1969/70), pp. 286--307. Lecture Notes in Math. {\bf 192}, 
Springer, Berlin, 1971. 


\bibitem[P1]{P1} P. Pragacz :
{\em Enumerative geometry of degeneracy loci},
Ann. Sci. \'Ecole Norm. Sup. (4) {\bf 21} (1988), 413--454.


\bibitem[P2]{P2} P. Pragacz :
{\em Algebro-geometric applications of Schur $S$- and $Q$-polynomials},
S\'{e}minare d'Alg\`{e}bre Dubreil-Malliavin 1989-1990, Lecture
Notes in Math. {\bf 1478} (1991), 130--191, Springer-Verlag, Berlin,
1991.

\bibitem[P3]{P3} P. Pragacz :
{\em Symmetric polynomials and divided differences in formulas of 
intersection theory}, Parameter spaces (Warsaw, 1994), 125--177,
Banach Center Publ. {\bf 36}, Polish Acad. Sci., Warsaw, 1996. 


\bibitem[PR1]{PR1} P. Pragacz and J. Ratajski :
{\em A Pieri-type theorem for Lagrangian and odd orthogonal
Grassmannians}, J. Reine Angew. Math. {\bf 476} (1996), 143--189.


\bibitem[PR2]{PR2} P. Pragacz and J. Ratajski :
{\em Formulas for Lagrangian and orthogonal degeneracy loci;
$\wt{Q}$-polynomial approach}, Compos. Math. {\bf 107}
(1997), 11--87.


\bibitem[PR3]{PR3} P. Pragacz and J. Ratajski :
{\em A Pieri-type formula for even orthogonal Grassmannians},
Fund. Math. {\bf 178} (2003), 49--96.


\bibitem[RS]{RS} V. Reiner and M. Shimozono : 
{\em Plactification}, J. Alg. Combin. {\bf 4} (1995), 331--351.

\bibitem[Ro]{Ro} G. de B. Robinson : 
{\em Representation theory of the symmetric group},  
Mathematical Expositions {\bf 12}, University of Toronto Press, 
Toronto, 1961.


\bibitem[Sa]{Sa} S. Sam : 
{\em Schubert complexes and degeneracy loci},
J. Algebra {\bf 337} (2011), 103--125.

\bibitem[SdS]{SdS} F. Sancho de Salas : {\em Milnor number of a vector
  field along a subscheme: applications in desingularization},
  Adv. Math. {\bf 153} (2000), 299--324.

\bibitem[Sc1]{Sc1} H. Schubert :
{\em Kalk\"ul der abz\"ahlenden Geometrie}, Teubner, Leipzig, 1879.


\bibitem[Sc2]{Sc2} H. Schubert :
{\em Die $n$-dimensionalen Verallgemeinerungen der fundamentalen Anzhalen
unseres Raums}, Math. Annalen {\bf 26} (1886), 26--51.


\bibitem[S1]{S1} I. Schur : 
{\em \"Uber ein Klasse von Matrizen die sich einer gegebenen Matrix
zuordnen lassen}, Dissertation, Berlin, 1901. Reprinted in
Gesammelte Abhandlungen {\bf 1}, 1-72, Springer-Verlag,
Berlin-New York, 1973.



\bibitem[S2]{S2} I. Schur :
{\em \"{U}ber die Darstellung der symmetrischen und der alternierenden
Gruppe durch gebrochene lineare Substitutionen}, J. reine angew.
Math. {\bf 139} (1911), 155--250.


\bibitem[SeC]{SeCh} S\'eminaire C. Chevalley; 2e ann\'ee: 1958.
{\em Anneaux de Chow et applications}, Secr\'etariat math\'ematique, 11 rue
Pierre Curie, Paris, 1958.



\bibitem[Sta]{Sta} R. P. Stanley :
{\em On the number of reduced decompositions of elements of Coxeter 
groups}, European J. Combin. {\bf 5} (1984), 359--372. 


\bibitem[St1]{St1} J. R. Stembridge :
{\em Shifted tableaux and the projective representations of symmetric
groups}, Adv. Math. {\bf 74} (1989), 87--134.

\bibitem[St2]{St2} J. R. Stembridge :
{\em Some combinatorial aspects of reduced words in finite Coxeter 
groups}, Trans. Amer. Math. Soc. {\bf 349} (1997), 1285--1332.


\bibitem[T1]{T1} H. Tamvakis : 
{\em Arakelov theory of the Lagrangian Grassmannian},
J. reine angew. Math. {\bf 516} (1999), 207--223.

\bibitem[T2]{Tcon} H. Tamvakis : 
{\em The connection between representation theory and Schubert calculus},
Enseign. Math. {\bf 50} (2004), 267--286.


\bibitem[T3]{T2} H. Tamvakis : 
{\em Schubert polynomials and Arakelov theory of symplectic flag 
varieties}, J. London Math. Soc. {\bf 82} (2010), 89--109.

\bibitem[T4]{T3} H. Tamvakis : 
{\em Schubert polynomials and Arakelov theory of orthogonal flag 
varieties}, Math. Z. {\bf 268} (2011), 355--370.


\bibitem[T5]{T4} H. Tamvakis : 
{\em Giambelli, Pieri, and tableau formulas via raising operators}, 
J. reine angew. Math. {\bf 652} (2011), 207--244.


\bibitem[T6]{T5} H. Tamvakis : 
{\em The theory of Schur polynomials revisited}, 
Enseign. Math. {\bf 58} (2012), 147--163. 

\bibitem[T7]{T6} H. Tamvakis :
{\em A Giambelli formula for classical $G/P$ spaces}, 
J. Algebraic Geom. {\bf 23} (2014), 245--278. 


\bibitem[T8]{T7} H. Tamvakis : 
{\em A tableau formula for eta polynomials},
Math. Annalen {\bf 358} (2014), 1005--1029.


\bibitem[TW]{TW} H. Tamvakis and E. Wilson :
{\em Double theta polynomials and equivariant Giambelli formulas}, 
Math. Proc. Cambridge Philos. Soc., to appear.


\bibitem[Th]{Th} R. Thom : 
{\em Les singularit\'es des applications diff\'erentiables}, 
Ann. Inst. Fourier {\bf 6} (1955--1956), 43--87. 

\bibitem[To]{To} G. P. Thomas :
{\em A note on Young's raising operator}, 
Canad. J. Math. {\bf 33} (1981), 48--54.

\bibitem[TY]{TY} H. Thomas and A. Yong :
{\em A combinatorial rule for (co)minuscule Schubert calculus}, 
Adv. Math. {\bf 222} (2009), 596--620.


\bibitem[Tu]{Tu} L. W. Tu :
{\em Degeneracy loci}, Proc. conf. algebraic geom. (Berlin, 1985), 296--305,
Teubner-Texte Math. {\bf 92}, Teubner, Leipzig, 1986.

\bibitem[W]{W} E. Wilson : 
{\em Equivariant Giambelli formulae for Grassmannians}, Ph.D.\ thesis,
University of Maryland, 2010.


\bibitem[Yo]{Yo} A. Yong : 
{\em On combinatorics of quiver component formulas}, 
J. Algebraic Combin. {\bf 21} (2005), 351--371. 


\bibitem[Y]{Y} A. Young : 
{\em On quantitative substitutional analysis VI}, Proc. Lond. 
Math. Soc. (2) {\bf 34} (1932), 196--230.


\end{thebibliography}
\end{document}